\documentclass[11pt,reqno,a4paper]{amsart}
\usepackage[british]{babel}
\usepackage{amssymb,amstext,amsthm,eucal, float, upgreek}
\usepackage[utf8]{inputenc}
\usepackage{array,color}
\usepackage[margin=2.8cm]{geometry}
\usepackage{enumitem}
\usepackage[small]{eulervm}
\usepackage{tgpagella}
\usepackage[unicode]{hyperref}
\usepackage{tikz, bbm}
\usetikzlibrary{arrows,decorations.markings,positioning}
\tikzset{inner sep=0pt, node distance=5mm,
  root/.style={circle,draw,minimum size=5pt,thick},
  broot/.style={circle,draw,minimum size=5pt,thick,fill},
  xroot/.style={circle,draw,minimum size=5pt,thick,label=below:$\times$},
  doublearrow/.style={postaction={decorate},   decoration={markings,mark=at position .6 with {\arrow[line width=1.2pt]{>}}},double distance=1.6pt,thick},
  rdoublearrow/.style={postaction={decorate},   decoration={markings,mark=at position .4 with {\arrowreversed[line width=1.2pt]{>}}},double distance=1.6pt,thick},
	rtriplearrow/.style={postaction={decorate},   decoration={markings,mark=at position .4 with {\arrowreversed[line width=1.2pt]{>}}},double distance=2.5pt,thick},
	ltriplearrow/.style={postaction={decorate},   decoration={markings,mark=at position .6 with {\arrow[line width=1.2pt]{>}}},double distance=2.5pt,thick},
  curvedline/.style={bend=right}
} 
\newcommand{\Out}{\mathrm{Out}}
\newcommand\mat[1]{\begin{pmatrix}#1\end{pmatrix}} 

\hypersetup{%
  pdftitle   = {Classification of simple Kantor triple systems},
  pdfkeywords = {},
  pdfauthor  = {Nicoletta Cantarini, Antonio Ricciardo, Andrea Santi},
  pdfcreator = {\LaTeX\ with package \flqq hyperref\frqq}
}
\numberwithin{equation}{section}
\theoremstyle{plain}
\newtheorem{theorem}{Theorem}[section]
\newtheorem{proposition}[theorem]{Proposition}
\newtheorem{corollary}[theorem]{Corollary}
\newtheorem{lemma}[theorem]{Lemma}
\theoremstyle{definition}
\newtheorem{rem}[theorem]{Remark}
\newtheorem{example}[theorem]{Example}
\newtheorem{definition}[theorem]{Definition}

\newcommand{\ph}{\phantom{0}}
\newcommand{\half}{\frac{1}{2}}

\newcommand{\End}{\operatorname{End}}

\newcommand{\SO}{\operatorname{SO}}

\newcommand{\Spin}{\operatorname{Spin}}

\newcommand{\Cl}{C\ell}

\newcommand{\stab}{\mathfrak{stab}}
\newcommand{\der}{\mathfrak{der}}

\newcommand{\fsl}{\mathfrak{sl}}
\newcommand{\fso}{\mathfrak{so}}

\newcommand{\fsu}{\mathfrak{su}}
\newcommand{\fm}{\mathfrak{m}}
\newcommand{\fp}{\mathfrak{p}}

\newcommand{\fu}{\mathfrak{u}}
\newcommand{\fg}{\mathfrak{g}}
\newcommand{\fh}{\mathfrak{h}}
\newcommand{\fk}{\mathfrak{k}}

\newcommand{\1}{\mathbbm{1}}

\newcommand{\RR}{\mathbb{R}}
\newcommand{\ZZ}{\mathbb{Z}}

\newcommand{\CC}{\mathbb{C}}
\newcommand{\FF}{\mathbb{F}}

\renewcommand{\SS}{\mathbb S}
\newcommand{\be}{\boldsymbol{e}}

\newcommand{\vol}{\operatorname{vol}}

\def\F{\mathbb{F}}

\allowdisplaybreaks
\begin{document}

\title[Classification of Kantor triple systems]{Classification of simple linearly compact Kantor triple systems over the complex numbers}
\author{Nicoletta Cantarini}
\author{Antonio Ricciardo}
\author{Andrea Santi}
\address{Dipartimento di Matematica, Universit\`a di Bologna, 
Piazza di Porta San Donato 5, 40126, Bologna, Italy}
\thanks{}
\begin{abstract}
Simple finite-dimensional Kantor triple systems over the complex numbers are classified in terms of Satake diagrams. We prove that every simple and linearly compact Kantor triple system has finite dimension and give an explicit presentation of all the classical and exceptional systems. 
\end{abstract}
\maketitle
\vskip-0.5cm\par\noindent
\tableofcontents
\vskip-0.5cm\par\noindent
\section{Introduction}
\label{sec:introduction}
Let $(A,\cdot)$ be an associative algebra. Then the commutative product 
$$a\circ b=\frac{1}{2}(a\cdot b+b\cdot a)$$
and the skew-commutative product 
$$[a,b]=\frac{1}{2}(a\cdot b-b\cdot a)$$
define on $A$ a Jordan algebra and a Lie algebra structure, respectively. A deep relationship
between these two kinds of algebras is given by the so-called Tits-Kantor-Koecher construction (TKK)
\cite{MR0217223, MR0214631, MR0146231},
which establishes a bijection between isomorphism classes of Jordan algebras and isomorphism classes of Lie algebras endowed with a short grading induced
by an $\mathfrak{sl}(2)$-triple.

Motivated by the work of Koecher on bounded symmetric domains \cite{MR0261032}, Meyberg extended the TKK correspondence to Jordan triple systems \cite{MR0340353}.
These are in fact particular examples of the so-called Kantor triple systems:
\begin{definition}\cite{MR0321986}
\label{def:KTS}
A {\it Kantor triple system} (shortly, KTS) is a complex vector space $V$ endowed with a trilinear map $(\cdot,\cdot,\cdot):\otimes^3 V\to V$ satisfying the following axioms:
\begin{itemize}
	\item[(i)] $(uv(xyz))=((uvx)yz)-(x(vuy)z)+(xy(uvz))$,
	\item[(ii)] $K_{K_{uv}(x)y}=K_{(yxu)v}-K_{(yxv)u}$,
	\end{itemize}
where $u,v,x,y,z\in V$ and $K_{xy}:V\to V$
is defined by $K_{xy}(z)=(xzy)-(yzx)$. 
\end{definition}
Kantor triple systems are also known as generalized Jordan triple systems of the second kind or $(-1,1)$-Freudenthal-Kantor triple systems.
In this paper we will deal with linearly compact systems which may have infinite dimension. In this case we also assume that the triple product is continuous. 
We will refer to $K_{xy}$ as the ``Kantor tensor'' associated to $x,y\in V$. Note that a Jordan triple system is precisely a KTS all of whose associated Kantor tensors vanish. 
\vskip0.1cm\par
\begin{definition}
\label{def:ideal}
A subspace $I\subset V$ of a Kantor triple system $V$ is called:
\begin{itemize}
\item[(i)] an {\it ideal} if $(VVI)+(VIV)+(IVV)\subset I$,
\item[(ii)] a {\it $K$-ideal} if $(VVI)+(IVV)\subset I$,
\item[(iii)] a {\it left-ideal} if $(VVI)\subset I$.
\end{itemize}
 If $V$ is linearly compact, we also assume that $I$ is closed in $V$.
\end{definition}
We say that $V$ is {\it simple} (resp. {\it $K$-simple}, resp. {\it irreducible}) if
it has no non-trivial ideals (resp. K-ideals, resp. left-ideals).
\vskip0.1cm\par

Simple finite-dimensional Jordan triple systems over an algebraically closed field were classified by Loos \cite{MR0325717}. The main aim of this paper is to systematically address the case where not all Kantor tensors are trivial. We remark that
such a KTS is either K-simple or it is polarized, i.e., it has a direct sum decomposition $V=V^+\oplus V^-$ satisfying $(V^\pm V^\mp V^\pm)\subset V^\pm$ and $(V^\pm V^\pm V)=0$ (see, e.g., \cite{MR3679693}). The classification problem is thus reduced to the study of K-simple linearly compact KTS, that is the primary object of study in this paper.
\begin{definition}
Let $V$ and $W$ be two KTS.
A bijective linear map $\varphi: V\to W$ is called:
\begin{itemize}
	\item[(i)] an {\it isomorphism} if $\varphi(xyz)=(\varphi x\varphi y\varphi z)$ for all $x,y,z\in V$,
	\item[(ii)] a {\it weak-isomorphism} if there exists another bijective linear map $\varphi':V\to W$ such that $\varphi(xyz)=(\varphi x\varphi'y\varphi z)$ for all $x,y,z\in V$.
\end{itemize}
If $V$ and $W$ are linearly compact, we also assume that $\varphi$ and $\varphi'$ are continuous.
\end{definition}
 
Starting from a (finite-dimensional) Kantor triple system $V$, Kantor constructed a $\ZZ$-graded Lie algebra $\fg=\fg_{-2}\oplus\cdots\oplus\fg_{2}$ with $\fg_{-1}\simeq V$ and endowed with a $\CC$-linear grade-reversing involution $\sigma:\fg\to\fg$ (see \S \ref{sec:TKK} for the definition). The Lie algebra $\fg$ is defined as an appropriate quotient of a subalgebra of the (infinite-dimensional) universal graded Lie algebra generated by $V$ (see also e.g. \cite[\S 3]{MR974266}) and if $ V $ is a Jordan triple system then $ {\fg_{-2}=\fg_{2}=0}$. Kantor then used this correspondence to classify the K-simple finite-dimensional KTS, up to weak-isomorphisms \cite{MR0321986}. Dealing with weak-isomorphisms instead of isomorphisms amounts to the fact that triple systems associated to the same grading but with {\it different involutions} are actually regarded as equivalent. The structure theory of Kantor triple systems (up to weak-isomorphisms) has been the subject of recent investigations, see \cite{MR3679693, MR2857919, MR2208167, MR3261738, MR3661634}. 

The classification problem of K-simple KTS up to isomorphisms has been completely solved only in the {\it real classical} case by H.\ Asano and S.\ Kaneyuki \cite{Asano1991, MR974266}.
The exceptional case is more intricate and an interesting class of models has been constructed in the compact real case by D.\ Mondoc in \cite{MR2275379, MR2362679}, making use of the structure theory of tensor products of composition algebras pioneered in \cite{MR933312}. Upon complexification, the finite-dimensional K-simple KTS obtained in \cite{MR2362679} correspond exactly to the class of KTS of {\it extended Poincar\'e type} that we introduce in this paper in terms of spinors and Clifford algebras. 
 
Although various examples of K-simple KTS are available in literature, a complete list is still missing. Our work aims to fill this gap and it provides the classification of {\it linearly compact complex K-simple KTS up to isomorphisms}. 

We depart in \S \ref{sec:KTS} with a simplified version of Kantor's original correspondence $V\Leftrightarrow (\fg,\sigma)$, which makes use of Tanaka's approach to transitive Lie algebras of vector fields \cite{MR0266258, MR533089} and can easily be adapted to linearly compact KTS --- reducing the problem of classifying KTS to the problem of classifying such pairs. 
Furthermore, in \S\ref{sec:finDim}  we develop a structure theory of grade-reversing involutions which holds for all finite-dimensional simple $\ZZ$-graded complex Lie algebras and establish in this way an intimate relation with real forms (see Theorem \ref{thm:graderevCartan}). 

The isomorphism classes of finite-dimensional K-simple KTS can be deduced by an analysis of the Satake diagrams, which is carried out in \S\ref{sec:finDim} and summarized in Corollary \ref{cor:enumerate}. In Theorem \ref{thm:symmetryalg} we show that also the Lie algebra of derivations of any finite-dimensional K-simple KTS can be easily read off from the associated Satake diagram. 

It is worth pointing out that the results contained in section \S\ref{sec:finDim} hold for gradings of finite-dimensional simple Lie algebras of {\it any} depth and therefore provide an abstract classification of all the so-called generalized Jordan triple systems of any kind $\nu\geq 1$ ($\nu=1$ are the Jordan triple systems, $\nu=2$ the Kantor triple systems).

A complete list of simple, linearly compact, infinite-dimensional Lie algebras
consists, up to isomorphisms, of the four simple Cartan algebras, namely, $W(m)$,
$S(m)$, $H(m)$ and $K(m)$, which are respectively the Lie algebra of
all formal vector fields in $m$ indeterminates and its subalgebras
of divergence free vector fields, of vector fields annihilating a
symplectic form (for $m$ even), and of vector fields multiplying a contact
form by a function (for $m$ odd) \cite{MR1509105, MR0268233}. It can be easily shown using \cite{MR0498755, MR1652530} that none of these algebras admits a non-trivial $\ZZ$-grading of finite length and hence, by the TKK construction for linearly compact KTS (Theorem \ref{thm:correspondence} and Theorem \ref{thm:simplicity}), we immediately arrive at the following result.
\begin{theorem}\label{noinf}
Any K-simple linearly compact Kantor triple system has finite dimension.
\end{theorem}
Note that, consistently, a similar statement holds for simple linearly compact Jordan algebras, see \cite{MR1398369}.

In the finite-dimensional case a complete list, up to isomorphisms, of K-simple KTS consists of eight infinite series, corresponding to classical Lie algebras, and 23 exceptional cases, corresponding to exceptional Lie algebras. 

The classical KTS are described in \S\ref{sec:classicalKTS} and they are the complexifications of the compact simple KTS classified in \cite{MR974266} (see Theorem \ref{thm:classicalKTSList}). 

The KTS of exceptional type can be divided into three main classes, depending on the graded component $\fg_{-2}$ of the associated Tits-Kantor-Koecher algebra $\fg=\fg_{-2}\oplus\cdots\oplus\fg_{2}$. (Some authors refer to this Lie algebra simply as the Kantor algebra.) We say that $V$ is:
\begin{itemize}
	\item[(i)] of {\it contact type} if $\dim\fg_{-2}=1$;
	\item[(ii)] of {\it extended Poincar\'e type} if $\fg_{-2}=U$ and $\fg_{0}$ is the Lie algebra direct sum of $\fso(U)$, of the grading element $\CC E$ and of a reductive subalgebra acting trivially on $\fg_{-2}$;
	\item[(iii)] of {\it special type} otherwise.
	\end{itemize}
We determine these KTS in \S\ref{sec:extendedKTS}, \S\ref{sec:contactKTS} and \S\ref{sec:specialKTS} respectively; we start with those of extended Poincar\'e type as this requires some preliminaries on Clifford algebras, which turn out to be useful for some KTS of contact type too.
Our description of the products of extended Poincar\'e type in terms of spinors gives an alternative realization of the KTS studied by D.\ Mondoc and it is inspired by the appearance of triple systems in connection with different kinds of symmetries in supergravity theories (see e.g. \cite{MR1846901,MR2858058}). 
In particular, we use some results from \cite{MR3255456, MR3218266} and rely on Fierz-like identities which are deduced from the Lie bracket in the exceptional Lie algebras.
The KTS of contact type are all associated to the unique contact grading of a simple complex Lie algebra and they are supported over $S^3\mathbb C^2$, $\Lambda^3_0\mathbb C^6$, $\Lambda^3 \mathbb C^6$, the semispinor module $\mathbb S^+$ in dimension $12$ and the $56$-dimensional representation $\mathfrak F$ of $E_7$. We shall stress again that a contact grading has usually more than one associated KTS; for instance $\mathfrak F$ admits two different products, with algebras of derivations $E_6\oplus\mathbb C$ and $\mathfrak{sl}(8,\mathbb C)$, respectively. Finally, there are two KTS of special type, which are supported over $V=\Lambda^2 (\mathbb C^5)^*\otimes\mathbb C^2$ and $V=\Lambda^3(\mathbb C^7)^*$, and associated to the Lie algebras $E_6$ and $E_7$, respectively. 

All details on products, including their explicit expressions, are contained in the main results ranging from Theorem \ref{thm:F4KTSsecond} in \S\ref{sec:ePF4} to Theorem \ref{thm:E7NonSpin} in \S \ref{sec:endpaper}.

\vskip0.3cm\par
Before concluding, we would like to briefly discuss the role of KTS in Physics. Nonlinear realizations of $3$-graded Lie algebras associated via the TKK correspondence to Jordan triple systems have been usually referred to as ``conformal realizations'', in analogy with the natural decomposition of the $ d $-dimensional conformal algebra $\fso(2,d)$. ``Quasi-conformal'' geometric realizations of $5$-graded Lie algebras $\fg=\fg_{-2}\oplus\cdots\oplus\fg_{2}$ arise naturally within the framework of $N=2$ Maxwell-Einstein supergravity in four dimensions \cite{MR756588} and they have been considered in e.g. \cite{MR1846901, MR2208167}. 
An explicit characterization of the triple products appears to be crucial to construct exceptional hidden symmetries, such as the so-called Freudenthal duality 
\cite{MR2570046}; we hope that the classification obtained in this paper will be useful
for further applications in this area. Recently, special examples of irreducible $N=5$ $3$-algebras relevant to the dynamics of multiple $M2$-branes have been constructed by Kim and Palmkvist \cite{MR2858058}, using a generalized TKK correspondence with $5$-graded Lie superalgebras. A systematic study of the $N=5$ $3$-algebras that parallels the classification of the KTS obtained in this paper will be the content of a future work. In sharp contrast to Theorem \ref{noinf}, several infinite-dimensional examples are to be expected.
\vskip0.3cm\par
{\it Acknowledgements.} We would like to thank the referee for carefully reading the paper and for useful comments and suggestions.
\section{Triple systems and the Tits-Kantor-Koecher construction}
\label{sec:KTS}
\subsection{Basic definitions and results}
\label{sec:Basics}
\subsubsection{Kantor triple systems}
\label{subsec:KTS}
Let us first recall some basic facts about Kantor triple systems. 
\begin{lemma}
If a KTS $V$ is irreducible, then it is $K$-simple. If it is $K$-simple, then it is simple.
\end{lemma}
The center $Z$ of a KTS $V$ is
\begin{equation}
\label{eq:center}
Z=\left\{v\in V\;|\;(xvy)=0\;\;\text{for all}\;\;x,y\in V\right\}\;.
\end{equation}
We remark that if $V$ is linearly compact then $Z$ is closed in $V$. If $Z=0$, we say that $V$ is centerless (some authors refer to this condition as ``condition (A)'', cf. \cite{MR974266}).
\begin{lemma}
The center $Z$ is an ideal of $V$. In particular, any simple KTS is centerless.
\end{lemma} 
\begin{proof}
If $v\in Z$ then $(x(vuy)z)=0$ by Definition \ref{def:KTS}$(i)$, i.e., $(vuy)\in Z$ for all $u,y\in V$;
one similarly shows that $(uyv)\in Z$ for all $u,y\in V$. Since $(VZV)=0$ by the definition of center, the first claim follows. The second claim is clear.
\end{proof}
In order to formulate the main Theorem \ref{thm:correspondence} of this section, we first need to recall some basic notions, in the form suitable for our purposes.
\subsubsection{The Tanaka prolongation}\label{sec:Tanaka}
A (possibly infinite-dimensional) Lie algebra $\fg$ with direct product decomposition $\fg=\Pi_{p\in\ZZ}\fg_p$ satisfying $[\fg_{p},\fg_{q}]\subset \fg_{p+q}$ for all $p,q\in\ZZ$ is called a $\ZZ$-{\it graded Lie algebra}.
\begin{definition}
A $\ZZ$-graded Lie algebra $\fg=\Pi\fg_p$ with negatively graded part $\fg_-=\Pi_{p<0}\fg_p$ is called:
\begin{itemize}
	\item[(i)] {\it fundamental} if $\fg_-$ is generated by $\fg_{-1}$,
	\item[(ii)] {\it transitive} if for any $D\in\fg_p$ with $p\geq 0$ the condition $[D,\fg_{-1}]=0$ implies $D=0$.
\end{itemize} 
It has {\it finite depth d} if $\fg_{p}=0$ for all $p<-d$ and $\fg_{-d}\neq 0$, where $d$ is some positive integer.
\end{definition} 
We recall that the \emph{maximal transitive prolongation} (in the sense of N.\ Tanaka) of a negatively graded fundamental Lie algebra $
\fm=\Pi_{p<0}\fm_p$ of finite depth is a $\mathbb{Z}$-graded Lie algebra 
\begin{equation}
\fg^\infty=\Pi_{p\in \ZZ} \fg^\infty_p
\label{eq:MTP}
\end{equation}
such that:
\begin{enumerate}
\item[(i)]
$\fg^{\infty}_-=\fm$ as $\ZZ$-graded Lie algebras;
\item[(ii)]
$\fg^{\infty}$ is transitive;
\item[(iii)]
$\fg^\infty$ is maximal with these properties, i.e., if $\fg$ is another $\ZZ$-graded Lie algebra which satisfies (i) and (ii), then $\fg\subset \fg^\infty$ as a $\ZZ$-graded subalgebra.
\end{enumerate}
The existence and uniqueness of $\fg^\infty$ is proved  in \cite{MR0266258} (the proof is for finite-dimensional Lie algebras but it extends verbatim to the infinite-dimensional case). 

The maximal transitive prolongation \eqref{eq:MTP} is easily described as follows. First $\fg^{\infty}_0=\der_0(\fm)$ is the Lie algebra of all $0$-degree derivations of $\fm$ and $[D,x]:=Dx$ for all $D\in\fg^{\infty}_0$ and $x\in\fm$. The spaces $\fg^{\infty}_p$ for all $p>0$ are defined inductively: the component
\begin{multline}
\label{eq:MTPp}
\fg^{\infty}_p=\left\{D:\fm\to(\fm\oplus\fg^{\infty}_0\oplus\cdots\oplus\fg^\infty_{p-1})\;\text{s.t.}\;(i)\;D[x,y]=[Dx,y]+[y,Dx]\;\text{for all}\;x,y\in\fm\;\right.\\
\left.(ii)\;D(\fm_q)\subset \fg^\infty_{q+p}\;\text{for all}\;q<0\right\}\hskip1.85cm
\end{multline}
is the space of $p$-degree derivations of $\fm$ with values in the $\fm$-module $\fm\oplus\fg^{\infty}_0\oplus\cdots\oplus\fg^\infty_{p-1}$. Again $[D,x]:=Dx$, for all $D\in\fg^\infty_p$, $p>0$, and $x\in\fm$. The brackets between non-negative elements of \eqref{eq:MTP} are determined uniquely 
by transitivity; for more details and their explicit expression, we refer to the original source \cite[\S 5]{MR0266258}.

In the infinite-dimensional linearly compact case the prolongation can be constructed in complete analogy, provided we take continuous derivations.
\subsection{The Tits-Kantor-Koecher construction revisited}\label{sec:TKK}\hfill
\vskip0.1cm\par\noindent
\begin{definition}
Let $\fg=\Pi\fg_p$ be a $\ZZ$-graded Lie algebra. An automorphism $\sigma:\fg\to\fg$ of Lie algebras is a {\it (grade-reversing) involution} if
\begin{itemize}
	\item[(i)] $\sigma^2(x)=x$ for all $x\in\fg$,
	\item[(ii)] $\sigma(\fg_p)=\fg_{-p}$ for all $p\in\ZZ$.
\end{itemize}
If $\fg$ is linearly compact, we also assume that $\fg=\Pi_{p\in\ZZ}\fg_p$ is the topological direct product of (closed, hence
linearly compact) subspaces $\fg_p$ and that $\sigma$ is continuous.
\end{definition}
From now on we will be interested in $5$-graded Lie algebras, i.e., with $\fg=\fg_{-2}\oplus\cdots\oplus\fg_{2}$. 
We say that two linearly compact $\mathbb Z$-graded Lie algebras with involutions $(\fg,\sigma)$ and $(\fg',\sigma')$ are isomorphic if there exists a continuous grade-preserving isomorphism of Lie algebras $\phi:\fg\to\fg'$ such that $\phi\circ\sigma=\sigma'\circ\phi$.
\begin{proposition}
\label{prop:KTSfromLie}
Let $(\fg,\sigma)$ be a pair consisting of a $5$-graded Lie algebra $\fg=\fg_{-2}\oplus\cdots\oplus\fg_{2}$ and a (grade-reversing) involution $\sigma:\fg\to\fg$. Then:
\begin{enumerate}
	\item $\fg_{-1}$ with the triple product
	\begin{equation}
\label{eq:tripleproduct}
(xyz):=[[x,\sigma(y)],z]\;,\qquad x,y,z\in\fg_{-1}\;,
\end{equation}
is a KTS, which we denote by $\mathcal K=\mathcal{K}(\fg,\sigma)$,
\item if $\fg$ is transitive, then $\mathcal{K}$ is centerless,
\item if $\fg$ is linearly compact, then so is $\mathcal{K}$.
\end{enumerate}
Furthermore, if $(\fg',\sigma')$ is isomorphic to $(\fg,\sigma)$ then $\mathcal{K}(\fg',\sigma')$ 
and $\mathcal{K}(\fg,\sigma)$ are isomorphic. 
 \end{proposition}
\begin{proof}
We first compute
\begin{align*}
(uv(xyz))&=[[u,\sigma(v)], [[x,\sigma(y)],z]]\\
&=[[[[u,\sigma(v)],x],\sigma(y)],z]+[[x,[[u,\sigma(v)],\sigma(y)]],z]
+[[x,\sigma(y)],[[u,\sigma(v)],z]]\\
&=((uvx)yz)+[[x,\sigma[[\sigma(u),v],y]],z]+(xy(uvz))\\
&=((uvx)yz)-(x(vuy)z)+(xy(uvz))\;,
\end{align*}
for all $u,v,x,y,z\in\fg_{-1}$. This proves axiom (i) in Definition \ref{def:KTS}. We now turn to describe the Kantor tensor  associated to $x,y\in\fg_{-1}$, i.e.,
\begin{align*}
K_{xy}(z)&=(xzy)-(yzx)=[[x,\sigma(z)],y]-[[y,\sigma(z)],x]\\
&=[[x,\sigma(z)],y]+[x,[y,\sigma(z)]]\\
&=[[x,y],\sigma(z)]\;,
\end{align*}
for all $z\in\fg_{-1}$; axiom (ii) in Definition \ref{def:KTS} follows directly from this fact and the following chain of equations
\begin{align*}
K_{K_{uv}(x)y}&=[[[u,v],\sigma(x)],y]=[[u,v],[\sigma(x),y]]+[[[u,v],y],\sigma(x)]\\
&=[[u,v],[\sigma(x),y]]=[u,[v,[\sigma(x),y]]]+[[u,[\sigma(x),y]],v]\\
&=-[[[y,\sigma(x)],v],u]+[[[y,\sigma(x)],u],v]\\
&=K_{(yxu)v}-K_{(yxv)u}\;,
\end{align*}
for all $u,v,x,y\in\fg_{-1}$. Note that we used $\fg_{-3}=0$ at the third step, so that $[[u,v],y]=0$. The proof of (1) is completed.

Let $v\in Z$ be an element of the center \eqref{eq:center} of $\mathcal K$. Then $[[x,\sigma(v)],y]=0$ for all $x,y\in\fg_{-1}$ and $\sigma(v)=0$, if $\fg$ is transitive. In other words $v=0$ and (2) is proved. 
If $\fg$ is linearly compact then $\fg_{-1}$ is a closed subspace of $\fg$, hence linearly compact; the triple product \eqref{eq:tripleproduct} is clearly continuous. This proves (3). The last claim is straightforward.
\end{proof}
We shall now associate to a centerless Kantor triple system $V$, a $5$-graded Lie algebra 
$$\fg=\fg(V)=\fg_{-2}\oplus\cdots\oplus\fg_{2}$$ with an involution $\sigma:\fg\to\fg$ such that $\mathcal K(\fg,\sigma)=V$ and the following properties are satisfied:
\begin{itemize}
\item[(P1)] $\fg$ is transitive and fundamental,
\item[(P2)] $[\fg_{-1},\fg_{1}]=\fg_{0}$.
\end{itemize}
In the following, we will introduce certain spaces of operators. In the infinite-dimensional linearly compact case, we will tacitly take the closure of such spaces (see \cite[Lemma 2.1]{MR2326139}).

We start with the negatively graded $\fm=\fm_{-2}\oplus \fm_{-1}$ defined by
\begin{equation}
\label{eq:fundbrackets}
\fm_{-1}=V\;,\qquad\fm_{-2}=\left\langle K_{xy}\;|\; x,y\in V\right\rangle\;,
\end{equation}
where the only non-trivial bracket is $[x,y]=K_{xy}$, for all $x,y\in\fm_{-1}$ (if $V$ is a Jordan triple system, then $\fm_{-2}=0$ and $\fm=\fm_{-1}$ is trivially fundamental). 
The Lie algebra $\fg=\fg_{-2}\oplus\cdots\oplus\fg_{2}$ we are after is a subalgebra of the maximal prolongation $\fg^{\infty}$ of $\fm$. 

We set $\fg_{p}=0$ for all $|p|>2$, $\fg_{p}=\fm_{p}$ for $p=-1,-2$, introduce linear maps $L_{xy}:\fm\to\fm$, $\varphi_x:\fm\to\fm_{-1}\oplus\fg_0^\infty$ and $D_{xy}:\fm\to\fg_0^\infty\oplus\fg_1^{\infty}$ given by
\begin{equation}
\label{eq:prol012}
\begin{array}{lll}
[L_{xy},z]=(xyz)\;, & [\varphi_x,z]=L_{zx}\;, & [D_{xy},z]=-\varphi_{K_{xy}(z)}\;,\\[0cm]
[L_{xy}, K_{uv}]=K_{K_{uv}(y)x}\;, & [\varphi_x,K_{uv}]=K_{uv}(x)\;, & [D_{xy},K_{uv}]=L_{K_{uv}(y)x}-L_{K_{uv}(x)y}\;,
\end{array}
\end{equation}
%
%
where $u,v,x,y,z\in V$.  Here $\fg_0=\left\langle L_{xy}\,|\,x,y\in V\right\rangle$, 
$\fg_1=\left\langle \varphi_x\,|\, x\in V\right\rangle$ and $\fg_2=\left\langle D_{xy}\,|\, x,y\in V\right\rangle$.
By definition the maps $L_{xy}$, $\varphi_x$ and $D_{xy}$ are linear in each of their subscripts.
\begin{proposition}
\label{prop:LiefromKTS}
The vector space $\fg=\fg_{-2}\oplus\cdots\oplus\fg_{2}$ defined above is a subalgebra of $\fg^\infty$ satisfying (P1)-(P2). It is finite-dimensional (resp. linearly compact) if and only if $V$ is finite-dimensional (resp. linearly compact).
\end{proposition}
\begin{proof}
By axiom (ii) of Definition \ref{def:KTS}, we have
\begin{align*}
[L_{xy},[u,v]]&=[L_{xy},K_{uv}]=K_{K_{uv}(y)x}=K_{(xyu)v}-K_{(xyv)u}\\
&=[[L_{xy},u],v]+[u,[L_{xy},v]]\;,
\end{align*}
hence $L_{xy}\in\fg_0^{\infty}$. We also note that, by axiom (i) of Definition \ref{def:KTS}, 
\begin{align*}
[[L_{uv},L_{xy}],z]&=[L_{uv},(xyz)]-[L_{xy},(uvz)]=(uv(xyz))-(xy(uvz))\\
&=((uvx)yz)-(x(vuy)z)=[L_{(uvx)y}-L_{x(vuy)},z]
\end{align*}
for all $z\in V$, hence, by transitivity,
\begin{equation}
\label{eq:bracket00}
[L_{uv},L_{xy}]=L_{(uvx)y}-L_{x(vuy)}\;,
\end{equation}
for all $u,v,x,y\in V$ and $\fg_{0}$ is a subalgebra of $\fg_0^\infty$.
In a similar way
\begin{align*}
[\varphi_{x},[u,v]]&=[\varphi_{x},K_{uv}]=K_{uv}(x)=(uxv)-(vxu)\\
&=[[\varphi_{x},u],v]+[u,[\varphi_{x},v]]
\end{align*}
and $\varphi_x\in\fg^\infty_1$; we note that the inclusion $[\fg_1,\fg_{-2}]\subset \fg_{-1}$ and the equality $[\fg_1,\fg_{-1}]=\fg_0$ hold by construction. Furthermore
\begin{align*}
[[L_{uv},\varphi_{x}],z]&=[L_{uv},L_{zx}]-[\varphi_x,(uvz)]=L_{(uvz)x}-L_{z(vux)}-L_{(uvz)x}\\
&=-L_{z(vux)}=-[\varphi_{(vux)},z]
\end{align*}
for all $z\in V$, hence 
\begin{equation}
\label{eq:bracket01}
[L_{uv},\varphi_{x}]=-\varphi_{(vux)}\;,
\end{equation}
for all $u,v,x\in V$ and $[\fg_1,\fg_0]\subset \fg_1$ as well.

To prove $D_{xy}\in\fg_2^\infty$, it is first convenient to observe that $[\varphi_x,\varphi_y]\in\fg_2^\infty$ and then compute
\begin{align*}
[[\varphi_x,\varphi_y],z]&=[\varphi_x,[\varphi_y,z]]+[[\varphi_x,z],\varphi_y]=[\varphi_x,L_{zy}]+[L_{zx},\varphi_y]\\
&=\varphi_{yzx}-\varphi_{xzy}=-\varphi_{K_{xy}(z)}
\end{align*}
and
\begin{align*}
[[\varphi_x,\varphi_y],K_{uv}]&=[\varphi_x,[\varphi_y,K_{uv}]]+[[\varphi_x,K_{uv}],\varphi_y]=
[\varphi_x,K_{uv}(y)]+[K_{uv}(x),\varphi_y]\\
&=L_{K_{uv}(y)x}-L_{K_{uv}(x)y}\;,
\end{align*}
where $u,v,x,y,z\in V$. In other words 
\begin{equation}
\label{eq:bracket11}
[\varphi_x,\varphi_y]=D_{xy}\,,
\end{equation}
hence $D_{xy}\in\fg_2^\infty$ and $[\fg_1,\fg_1]=\fg_2$. We also note that
\begin{equation}
\label{eq:bracket02}
\begin{split}
[L_{uv},D_{xy}]&=[L_{uv},[\varphi_x,\varphi_y]]=[[L_{uv},\varphi_x],\varphi_y]+[\varphi_x,[L_{uv},\varphi_y]]\\
&=-[\varphi_{(vux)},\varphi_y]-[\varphi_x,\varphi_{(vuy)}]=-D_{(vux)y}-D_{x(vuy)}
\end{split}
\end{equation}
for all $u,v,x,y\in V$; in particular $[\fg_0,\fg_2]\subset \fg_2$.

Finally, we consider $[\varphi_u,D_{xy}]\in\fg^\infty_3$ and compute
\begin{align*}
[[\varphi_u,D_{xy}],z]&=[\varphi_u,[D_{xy},z]]+[[\varphi_u,z],D_{xy}]=[\varphi_{K_{xy}(z)},\varphi_u]+[L_{zu},D_{xy}]\\
&=D_{K_{xy}(z)u}-D_{(uzx)y}-D_{x(uzy)}\;,
\end{align*}
for all $u,x,y,z\in V$. Applying both sides to $v\in V$ immediately yields $[[[\varphi_u,D_{xy}],z],v]=\varphi_w$, where
\begin{align*}
w&=-K_{K_{xy}(z)u}(v)+K_{(uzx)y}(v)-K_{(uzy)x}(v)\\
&=0\,,
\end{align*}
by axiom (ii) of Definition \ref{def:KTS}. In summary $[\varphi_u, D_{xy}]=0$ by a repeated application of transitivity, $[\fg_1,\fg_2]=0$ and $[\fg_2,\fg_2]=[\fg_2,[\fg_1,\fg_1]]=0$. The first claim of the proposition is proved.

The second claim is straightforward. We only note here that if $V$ is linearly compact then $\fg$ is linearly compact too. 
This can be shown by the same arguments as \cite[Lemma 2.1]{MR2326139}.
\end{proof}
Now we define $\sigma:\fg\to\fg$ by
\begin{equation}
\label{eq:involution}
\begin{split}
\sigma(K_{xy})&=D_{xy}\;,\qquad\sigma(x)=-\varphi_x\;,\\
\sigma(L_{xy})&=-L_{yx}\qquad\text{and}\\
\sigma(\varphi_x)&=-x\;,\quad\,\sigma(D_{xy})=K_{xy}\;,
\end{split}
\end{equation}
where $x,y\in V$. It can be easily checked that it is a well defined map as the center of $V$ is zero; using \eqref{eq:fundbrackets}-\eqref{eq:prol012} one gets the following.
\begin{lemma}
\label{lem:invol}
$\sigma$ is a grade reversing involution of $\fg$.
\end{lemma}
It is straightforward to show that $\mathcal{K}(\fg,\sigma)=V$. If $V$ and $V'$ are isomorphic centerless KTS, then the
corresponding pairs $(\fg,\sigma)$ and $(\fg',\sigma')$ described in 
Proposition \ref{prop:LiefromKTS} and Lemma \ref{lem:invol} are isomorphic. Indeed, if 
$\varphi:V\to V'$ is an isomorphism then $\phi:\fm\to\fm'$, $\phi(x)=\varphi(x)$,
$\phi(K_{xy})=K'_{\varphi(x)\varphi(y)}$ extends to a unique grade-preserving $\phi:\fg\to\fg'$ satisfying $\phi\circ \sigma=\sigma'\circ \phi$. 
\begin{theorem}
\label{thm:correspondence}
There exists a one-to-one correspondence between isomorphism classes of centerless KTS and isomorphism classes of pairs $(\fg,\sigma)$, where $\fg=\fg_{-2}\oplus\cdots\oplus\fg_{2}$ is a transitive fundamental $5$-graded Lie algebra such that $[\fg_{-1},\fg_{1}]=\fg_0$ and $\sigma$ a grade-reversing involution of $\fg$.
\end{theorem}
\begin{proof}
Let $\mathcal K$ be the KTS associated to $(\fg,\sigma)$. To prove the theorem, it remains only to show that the pair
$(\widetilde\fg,\widetilde\sigma)$ constructed from $\mathcal K$
is isomorphic to $(\fg,\sigma)$. 
The required isomorphism is given by
\begin{equation} 
\phi:
\begin{array}{c}
\fg_2\\
\fg_1\\
\fg_0\\
\fg_{-1}\\
\fg_{-2}
\end{array}
\;
\left|\;
\begin{array}{lcr}
[\sigma(x),\sigma(y)] & \mapsto & D_{xy} \\
\sigma(y) & \mapsto & -\varphi_y \\
{[x,\sigma(y)]} & \mapsto & L_{xy} \\
x &\mapsto & x\;\;\;\\
{[x,y]} & \mapsto & K_{xy} 
\end{array}
\right.\;,
\end{equation}
where $x,y\in \mathcal K=\fg_{-1}=\widetilde\fg_{-1}$. It is not difficult to check that $\phi$ is well defined and invertible. For instance, $L_{xy}=0$ if and only if $[L_{xy},z]=(xyz)=[[x,\sigma(y)],z]=0$ by transitivity of $\widetilde\fg$, hence if and only if $[x,\sigma(y)]=0$ by transitivity of $\fg$; this proves that $\phi|_{\fg_0}$ is well defined and injective. One can also directly check that 
$\phi\fg_0=\phi[\fg_{-1},\fg_{1}]=[\phi\fg_{-1},\phi\fg_{1}]=[\widetilde\fg_{-1},\widetilde\fg_{1}]=\widetilde\fg_{0}$, hence
$\phi|_{\fg_0}:\fg_0\to\widetilde\fg_{0}$ is invertible. The proof for non-zero degrees is similar, we give details for $\fg_{-2}$:
$K_{xy}(z)=0$ implies $(xzy)=(yzx)$ for every $z\in\mathcal K$, so $[[x,\sigma(z)],y]=[[y,\sigma(z)],x]$ and hence by the Jacobi identity 
$[[x,y],\sigma(z)]=0$. The rest follows from transitivity since $\sigma$ is an involution.
The facts that $\phi$ is a Lie algebra morphism and $\phi\circ\sigma=\widetilde\sigma\circ\phi$ are immediate by construction.
\end{proof}

\subsection{Simplicity in the Tits-Kantor-Koecher construction}
\label{sec:simple}
Let $V$ be a centerless KTS and $(\fg,\sigma)$ the associated pair, see Theorem \ref{thm:correspondence}.
\begin{theorem}
\label{thm:simplicity}
The Kantor triple system $V$ is $K$-simple if and only if $\fg$ is simple.
\end{theorem}
Theorem \ref{thm:simplicity} was originally proved by Kantor in the finite-dimensional case, cf. \cite{MR0321986}. We provide here a streamlined presentation of an alternative proof which works also in the linearly compact case. The result is a direct consequence of Propositions \ref{prop:1simple}, \ref{prop:2simple} and \ref{prop:3simple}. 

We begin with some preliminary observations on Lie triple systems and $\mathbb Z_2$-graded Lie algebras. We recall that a Lie triple system (shortly, LTS) is a complex vector space $T$ with a trilinear map $[\cdot,\cdot,\cdot]:\otimes^3 T\to T$ satisfying the following axioms (see e.g. \cite{MR0340353}):
\begin{itemize}
	\item[(i)] $[xyz]=-[yxz]$,
	\item[(ii)] $[xyz]+[yzx]+[zxy]=0$,
	\item[(iii)] $[xy[zwu]]=[[xyz]wu]+[z[xyw]u]+[zw[xyu]]$,
	\end{itemize}
for all $x,y,z,w,u\in T$. A subspace $\mathcal J\subset T$ is an {\it ideal} if $[\mathcal JTT]\subset\mathcal J$ (the inclusion $[T\mathcal JT]+[TT\mathcal J]\subset \mathcal J$ is then automatic by (i) and (ii)) and $T$ is called {\it simple} if it has no non-trivial ideals. It is well known that for any $\mathbb Z_2$-graded Lie algebra $\fg=\fg_{\bar 0}\oplus\fg_{\bar 1}$ the space $T=\fg_{\bar 1}$with product $$[xyz]:=[[x,y],z]$$
is a LTS.

Let now $V$ be a centerless KTS, with associated pair $(\fg,\sigma)$. The $5$-graded Lie algebra $\fg$ admits the coarser $\mathbb Z_2$-grading 
\begin{equation}
\label{eq:coarser}
\fg=\fg_{\bar 0}\oplus\fg_{\bar 1}\;,\qquad\fg_{\bar 0}=\fg_{-2}\oplus\fg_{0}\oplus\fg_{2}\;,\qquad \fg_{\bar 1}=\fg_{-1}\oplus\fg_{1}\;,
\end{equation}
and $T=\fg_{\bar 1}=\fg_{-1}\oplus\fg_{1}$ inherits the natural structure of LTS. 
We note that \eqref{eq:coarser} coincides with the so-called ``standard embedding'' of $T$, as the action of $\fg_{\bar 0}$  on $\fg_{\bar 1}$ is faithful and $[\fg_{\bar 1},\fg_{\bar 1}]=\fg_{\bar 0}$.
\begin{proposition}
\label{prop:1simple}
Let $V$ be a centerless KTS, with associated $(\fg,\sigma)$ and $T=\fg_{-1}\oplus\fg_{1}$. Then $\fg$ is simple if and only if $T$ is a simple Lie triple system.
\end{proposition}
\begin{proof}
We first consider $\mathbb Z_2$-graded ideals of $\fg$, along the lines of \cite{MR0340353}. A direct computation shows that a $\mathbb Z_2$-graded subspace
$$\fk=\fk_{\bar 0}\oplus\fk_{\bar 1}\;,\qquad \fk_{\bar 0}\subset \fg_{\bar 0}\;,\qquad\fk_{\bar 1}\subset \fg_{\bar 1}\;,$$
of $\fg$ is an ideal  if and only if
\begin{itemize}
\item[(i)] $\fk_{\bar 0}$ is an ideal of the Lie algebra $\fg_{\bar 0}$,
\item[(ii)] $\fk_{\bar 1}$ is an ideal of the Lie triple system $T$, and
\item[(iii)] $[\fk_{\bar 1},\fg_{\bar 1}]\subset \fk_{\bar 0}\subset \stab(\fg_{\bar 1},\fk_{\bar 1})$,
\end{itemize}
where $\stab(\fg_{\bar 1},\fk_{\bar 1})=\left\{x\in\fg_{\bar 0}\mid[x,\fg_{\bar 1}]\subset \fk_{\bar 1}\right\}$. Using (i)-(iii), one can readily check that $[\fk_{\bar 1},\fg_{\bar 1}]\oplus\fk_{\bar 1}$ and $\stab(\fg_{\bar 1},\fk_{\bar 1})\oplus\fk_{\bar 1}$ are ($\mathbb Z_2$-graded) ideals of $\fg$ exactly when $\fk_{\bar 1}$ is an ideal of $T$. In particular, $T$ is simple if and only if $\fg$ does not have any non-trivial $\mathbb Z_2$-graded ideal, see e.g. \cite[p. 49]{MR0340353}.

We have seen that $\fg$ simple implies $T$ simple. To finish the proof, it is convenient to consider the involutive automorphism $\theta:\fg\to\fg$ associated to the decomposition $\fg=\fg_{\bar 0}\oplus\fg_{\bar 1}$ of $\fg$; note that $\mathbb Z_2$-graded ideals of $\fg$ are the ideals stable under $\theta$.

Let $T$ be simple and let us assume by contradiction that there exists a non-trivial ideal $\fk$ of $\fg$; by the previous result, $\fk$ is not $\mathbb Z_2$-graded. The ideals $\theta(\fk)\cap\fk$ and $\theta(\fk)+\fk$ are $\mathbb Z_2$-graded, hence
$\theta(\fk)\cap\fk=0$, $\theta(\fk)+\fk=\fg$, and
\begin{equation}
\label{eq:split}
\begin{split}
\fg&=\fk\oplus\theta(\fk)\\
&\simeq \fk\oplus\fk
\end{split}
\end{equation}
is the direct sum of two ideals isomorphic to $\fk$. It is immediate to see that $\fk$ is simple: if $\mathfrak i$ is a non-zero ideal of $\fk$ then $\mathfrak i \oplus\theta(\mathfrak i)$ is a $\mathbb Z_2$-graded ideal of $\fg$, $\mathfrak i \oplus\theta(\mathfrak i)=\fg=\fk\oplus\theta(\fk)$ and $\mathfrak i=\fk$. 

The $\mathbb Z$-grading of $\fg$ induces the $\mathbb Z_2$-grading 
\begin{align*}
\fg_{\bar 0}&=\fg_{-2}\oplus\fg_{0}\oplus\fg_{2}=\left\{(x,\theta(x))\mid x\in\fk\right\}\;,\\
\fg_{\bar 1}&=\fg_{-1}\oplus\fg_{1}=\left\{(x,-\theta(x))\mid x\in\fk\right\}\;,
\end{align*}
and therefore sits diagonally with respect to \eqref{eq:split}. This is a contradiction, as every simple ideal of a 
$\mathbb Z$-graded semisimple Lie algebra is itself $\mathbb Z$-graded. Note that if $S$ is a simple,
 infinite-dimensional linearly compact Lie algebra, then $\mathfrak{der}(S\oplus S)=\mathfrak{der}(S)\oplus \mathfrak{der}(S)$.
In summary, if $T$ is simple then $\fg$ is simple and the proof is completed.
\end{proof}
In order to relate the simplicity of $V$ with that of $T$, 
we first introduce an auxiliary notion of ideal, motivated by the decomposition $T=\fg_{-1}\oplus\fg_{1}$ and the identification $\fg_{1}\underset{\sigma}{\simeq}\fg_{-1}$ (see also \cite{MR2857919} for the more general framework of ``Kantor pairs''). 
We say that a pair $\mathcal I=(\mathcal I^+,\mathcal I^-)$ of subspaces $\mathcal I^\pm\subset V$ constitutes a {\it $P$-ideal} of $V$ if 
\begin{equation}
\label{eq:Pideal}
(\mathcal I^\varepsilon VV)+(V\mathcal I^{-\varepsilon} V)+(VV\mathcal I^\varepsilon)\subset \mathcal I^\varepsilon\;,
\end{equation}
where $\varepsilon=\pm 1$.
\begin{proposition}
\label{prop:2simple}
The Kantor triple system $V$ is $K$-simple if and only if it does not admit any non-trivial $P$-ideal.
\end{proposition}
\begin{proof}
\vskip0.1cm\par
$(\Longrightarrow)$ Let $V$ be $K$-simple and $\mathcal I=(\mathcal I^+,\mathcal I^-)$ a non-zero $P$-ideal; without loss of generality, we may assume $\mathcal I^+\neq 0$. By \eqref{eq:Pideal} with $\varepsilon=1$, it is clear that $\mathcal I^+$ is a $K$-ideal of $V$ hence $\mathcal I^+=V$. If $\mathcal I^-=0$, then $(V\mathcal I^+V)=0$ by \eqref{eq:Pideal} with $\varepsilon=-1$ and $\mathcal I^+=0$, as $V$ is centerless. This implies that $\mathcal I^-$ too is non-zero, hence $\mathcal I^-=V$ and $\mathcal I=(V,V)$.
\vskip0.1cm\par
$(\Longleftarrow)$ We assume that $V$ does not have non-trivial $P$-ideals and let $I$ be a non-zero $K$-ideal of $V$. Consider the pair $\mathcal I=(\mathcal I^+,\mathcal I^-)$, where 
\begin{equation}
\label{eq:Pidealassociated}
\begin{split}
\mathcal I^+&=I\;,\\
\mathcal I^-&=\left\{v\in V\mid (VvV)\subset I\right\}\;.
\end{split}
\end{equation}
We want to show to $\mathcal I=(\mathcal I^+,\mathcal I^-)$ is a $P$-ideal. 

Identity \eqref{eq:Pideal} with $\varepsilon=1$ is satisfied by construction. 
To prove \eqref{eq:Pideal} with $\varepsilon=-1$, we first recall that
\begin{equation}
\label{eq:useful}
(x(vuy)z)=((uvx)yz)+(xy(uvz))-(uv(xyz))\;,
\end{equation}
for all $u,v,x,y,z\in V$, by axiom (i) in Definition \ref{def:KTS}. If $v\in\mathcal I^-$, the r.h.s. of \eqref{eq:useful} is in $I$, by the definition of $\mathcal I^-$ and the fact that $I$ is a $K$-ideal, therefore $(\mathcal I^-VV)\subset\mathcal I^-$; one similarly shows the inclusion $(VV\mathcal I^-)\subset \mathcal I^-$. 
Finally, the r.h.s. of \eqref{eq:useful} is in $I$ if $u\in I$, hence
$(V\mathcal I^+V)\subset\mathcal I^-$ too. 

It follows that $\mathcal I=(\mathcal I^+,\mathcal I^-)$ as in \eqref{eq:Pidealassociated} is a non-zero $P$-ideal, hence $\mathcal I^\pm=V$. In particular $I=\mathcal I^+=V$ and $V$ is $K$-simple. 
\end{proof}
The following result finally relates $P$-ideals with ideals of $T$; it can be regarded as the version for KTS of \cite[Lemma 1.6, Proposition 1.7]{MR2857919}.
\begin{proposition}
\label{prop:3simple}
Let $V$ be a centerless KTS, with associated $(\fg,\sigma)$ and $T=\fg_{-1}\oplus\fg_{1}$. Then:
\begin{itemize}
	\item[(1)] If $\mathcal I=(\mathcal I^+,\mathcal I^-)$ is a $P$-ideal, then $\mathcal J=\mathcal I^-\oplus\sigma(\mathcal I^+)$ is an ideal of $T$;
	\item[(2)] If $\mathcal J$ is an ideal of $T$, then 
	\begin{equation}
	\label{eq:twoideals}
	\begin{split}
	\mathcal J\cap V&:=(\mathcal J\cap \fg_{-1},\sigma(\mathcal J\cap\fg_{1}))\,,\\
	\pi(\mathcal J)&:=(\pi^{-1}(\mathcal J),\sigma(\pi^{+1}(\mathcal J)))\;,
	\end{split}
	\end{equation}
	are $P$-ideals, where $\pi^{\pm 1}:T\to\fg_{\pm 1}$ are the natural projections from $T$ to $\fg_{\pm 1}$;
    \item[(3)] $T$ is simple if and only if $V$ does not admit non-trivial $P$-ideals.
\end{itemize}
\end{proposition}
\begin{proof}
\vskip0.1cm\par
$(1)$ We compute
\begin{equation}
\label{eq:PandT+}
\begin{split}
[\mathcal I^-\fg_{-1}\fg_{-1}]&=0\;,\\
[\mathcal I^-\fg_{1}\fg_{-1}]&=(\mathcal I^-VV)\subset \mathcal I^-\,,\\
[\mathcal I^-\fg_{1}\fg_{1}]&=[\fg_{1}\mathcal I^-\fg_{1}]=\sigma((V\mathcal I^-V))\subset\sigma(\mathcal I^+)\,,\\
[\mathcal I^-\fg_{-1}\fg_{1}]&\subset[\mathcal I^-\fg_{1}\fg_{-1}]+[\fg_{-1}\fg_{1}\mathcal I^-]=(\mathcal I^-VV)+(VV\mathcal I^-)
\subset \mathcal I^-\,,\\
\end{split}
\end{equation}
since $\fg_{-3}=0$ and $\mathcal I$ is a $P$-ideal. Similar identities hold for $\sigma(\mathcal I^+)$, hence $\mathcal J$ is an ideal of $T$.
\vskip0.15cm\par
$(2)$ We omit the direct computations for the sake of brevity.
\vskip0.15cm\par
$(3)$ If $T$ is simple, then $V$ does not admit any non-trivial $P$-ideal by (1). Conversely, let 
us assume $V$ does not have non-trivial $P$-ideals and consider a non-zero ideal $\mathcal J$ of $T$. We have $\pi(\mathcal J)=(V,V)$ by (2) and
\begin{align*}
0\neq (VVV)&=(V\sigma(\pi^{+1}(\mathcal J))V)\\
&=[\fg_{-1}\pi^{+1}(\mathcal J)\fg_{-1}]\\
&=[\fg_{-1}\mathcal J\fg_{-1}]\subset \mathcal J\cap \fg_{-1}\;,
\end{align*}
since $\mathcal J\subset \pi^{-1}(\mathcal J)\oplus \pi^{+1}(\mathcal J)$ and $[\fg_{-1}\pi^{-1}(\mathcal J)\fg_{-1}]=0$.
It follows that $\mathcal J\cap V$ is a non-zero $P$-ideal, hence $\mathcal J\cap V=(V,V)$, $\mathcal J=T$ and $T$ is simple.
\end{proof}
\section{Finite dimensional Kantor triple systems}\label{sec:finDim}
\subsection{Preliminaries on real and complex  \texorpdfstring{$\ZZ$}{Z}-graded Lie algebras}
We recall the following relevant result, see e.g. \cite[Lemma 1.5]{MR533089}.
\begin{proposition}
A finite-dimensional simple $\ZZ$-graded Lie algebra $\fg=\bigoplus_{p\in\ZZ}\fg_{p}$ over $\FF$ ($\FF=\CC$ or $\FF=\RR$) always admits a grade-reversing Cartan involution.
\end{proposition}
In other words, there exists a grade-reversing involution $\theta:\fg\to\fg$ such that the form 
\begin{equation}
\label{eq:CI}
B_{\theta}(x,y)=-B(x,\theta y)\;,\qquad x,y\in\fg\;,
\end{equation}
is positive-definite symmetric ($\FF=\RR$) or positive-definite Hermitian ($\FF=\CC$), where 
$B$ is the Killing form of $\fg$. Note that $\theta:\fg\to\fg$ is antilinear when $\FF=\CC$ --- for uniformity of exposition in this section, we will still refer to any $\RR$-linear involutive automorphism of a complex Lie algebra as an ``involution''.
\vskip0.1cm\par
Due to Theorems \ref{thm:correspondence} and \ref{thm:simplicity}, our interest is in a related but slighly different problem, i.e., {\it classifying} grade-reversing {\it $\CC$-linear} involutions of $\ZZ$-graded complex simple  Lie algebras, up to {\it zero-degree} automorphisms. We will shortly see that the solution of this problem is tightly related to
suitable real forms of complex Lie algebras.

We begin with an auxiliary result -- the $\ZZ$-graded counterpart of a known fact in the classification of real forms. Here and in the following section, we denote by 
$(\fg,\sigma)$ a pair consisting of a (finite-dimensional) $\ZZ$-graded simple Lie algebra $\fg=\bigoplus_{p\in\ZZ}\fg_{p}$ over $\FF=\RR$ or $\CC$ and a grade-reversing involution $\sigma:\fg\to\fg$.
We let $G_0$ be the connected component of the group of inner automorphisms of $\fg$ of degree zero.
\begin{proposition}
\label{prop:commuting}
For any pair $(\fg,\sigma)$ and a grade-reversing Cartan involution $\theta:\fg\to\fg$, there exists $\phi\in G_0$ such that $\phi\circ\theta\circ\phi^{-1}$ commutes with $\sigma$.
\end{proposition}
This result is proved in a similar way as in \cite[Lemma 6.15, Theorem 6.16]{MR1920389} and
we omit details for the sake of brevity.
The following corollary is proved as in \cite[Corollary 6.19]{MR1920389}.
\begin{corollary}
\label{cor:equiv0deg}
Any two grade-reversing Cartan involutions of a simple $\fg=\bigoplus_{p\in\ZZ}\fg_{p}$ are conjugate by some $\phi\in G_0$.
\end{corollary}
\subsection{Grade-reversing involutions and aligned pairs}
We now introduce the main source of grade-reversing $\CC$-linear involutions
on complex simple Lie algebras. Let 
\begin{equation}
\fg^o=\bigoplus_{p\in\ZZ}\fg^o_{p}
\label{eq:gradedreal}
\end{equation}
be a {\it real} absolutely simple $\ZZ$-graded Lie algebra and $\theta:\fg^o\to\fg^o$ a grade-reversing Cartan involution. The complexification $\fg$ of $\fg^o$ is a simple $\ZZ$-graded Lie algebra $\fg=\bigoplus_{p\in\ZZ}\fg_{p}$ and the $\CC$-linear extension  of 
$\theta$ a grade-reversing involution $\sigma:\fg\to\fg$.
\begin{definition}
\label{def:standpair}
The pair $(\fg,\sigma)$ is the (complexified) {\it pair aligned} to $\fg^o$ and $\theta$. 
\end{definition}
\begin{definition}
\label{def:isomstandpair}
Two aligned pairs $(\fg,\sigma)$ and $(\fg',\sigma')$ are {\it isomorphic} if there is a zero-degree Lie algebra isomorphism $\phi:\fg\to\fg'$ such that $\phi\circ\sigma=\sigma'\circ\phi$.
\end{definition}
Proposition \ref{prop:well defined} below says that the isomorphism class of an aligned pair does not depend on the choice of the Cartan involution but only on the real Lie algebra. To prove that, we first need
a technical but useful result.
\begin{lemma}
\label{lem:simpliedproof}
Let $\fg^o$ and $\widetilde\fg^o$ be $\ZZ$-graded real forms of
a $\ZZ$-graded simple complex Lie algebra $\fg$. If the grade-reversing Cartan involutions
\begin{equation}
\label{eq:equality}
\theta:\fg^o\to\fg^o\;,\qquad\widetilde\theta:\widetilde\fg^o\to\widetilde\fg^o\;,
\end{equation} 
have equal $\CC$-linear extensions $\sigma,\widetilde\sigma:\fg\to\fg$ then $\fg^o$ and $\widetilde\fg^o$ are isomorphic as $\ZZ$-graded Lie algebras.
\end{lemma}
\begin{proof}
Let $\fg^o=\fk\oplus\fp$ and $\widetilde\fg^o=\widetilde\fk\oplus\widetilde\fp$ be the Cartan decompositions associated to \eqref{eq:equality}. In other words, the real forms
$\fu^o=\fk\oplus i\fp$ and $\widetilde\fu^o=\widetilde\fk\oplus i\widetilde\fp$ are compact and
\begin{equation}
\begin{split}
\vartheta_{\fg^o}&=\vartheta_{\fu^o}\circ\sigma=\sigma\circ\vartheta_{\fu^o}\;,\\
\vartheta_{\widetilde\fg^o}&=\vartheta_{\widetilde\fu^o}\circ\widetilde\sigma=\widetilde\sigma\circ\vartheta_{\widetilde\fu^o}\;,
\end{split}
\end{equation}
where $\vartheta_{\fg^o},\ldots, \vartheta_{\widetilde\fu^o}$ are the antilinear involutions of $\fg$ associated to the real forms $\fg^o,\ldots,\widetilde\fu^o$. It follows that $\vartheta_{\fu^o}$ and $\vartheta_{\widetilde\fu^o}$ are grade-reversing Cartan involutions of $\fg$ and $\vartheta_{\widetilde\fu^o}=\phi\circ\vartheta_{\fu^o}\circ\phi^{-1}$ for some zero-degree automorphism $\phi$ of $\fg$, by Corollary \ref{cor:equiv0deg}. 

We replace $\widetilde\fg^o$ with the isomorphic $\ZZ$-graded real Lie algebra $$\widehat\fg^o:=\phi^{-1}(\widetilde\fg^o)$$ 
with grade-reversing Cartan involution $\widehat\theta=\phi^{-1}\circ\widetilde\theta\circ\phi:\widehat\fg^o\to\widehat\fg^o$. 
By construction the compact real form $\widehat u^o=\phi^{-1}(\widetilde u^o)$ 
coincides with $\fu^o$ and it is therefore stable under the action of both $\sigma$ and the $\CC$-linear extension $\widehat\sigma$ of $\widehat\theta$. We also note that $\widehat\sigma=\phi^{-1}\circ\widetilde\sigma\circ\phi=\phi^{-1}\circ \sigma\circ\phi$, where the last identity follows from our hypothesis $\widetilde\sigma=\sigma$.

The Lie algebra $\fg_0=Lie(G_0)$ of the connected component $G_0$ of the group of zero-degree inner automorphisms of $\fg$ decomposes into 
$$
\fg_0=(\fg_0\cap\fu^o)\oplus(\fg_0\cap i\fu^o)\;,
$$
as $\vartheta_{\fu^o}(\fg_0)=\fg_0$. We denote by $U$ the analytic subgroup of $G_0$ with Lie algebra $Lie(U)=\fg_0\cap\fu^o$ and note that the mapping 
$$\Phi:(\fg_0\cap i\fu^o)\times U\to G_0\;,\qquad \Phi(X,u)=(\exp X)u\;,$$
where $X\in\fg_0\cap i\fu^o$, $u\in U$, is a diffeomorphism, cf. \cite[Theorem 1.1, p. 252]{MR1834454} and \cite[Theorem 6.31]{MR1920389}. In particular $\phi^{-1}=p\circ u$, 
for some $u\in U$ and an element $p\in \exp(\fg_0\cap i\fu^o)$ which commutes with $u\circ\sigma\circ u^{-1}$, cf. a standard argument in the proof of \cite[Proposition 1.4, p. 442]{MR1834454}. It follows that $$\widehat\sigma=u\circ\sigma\circ u^{-1}$$ 
and  $u|_{\fg^o}:\fg^o\to \widehat \fg^o$ is the required isomorphism of $\ZZ$-graded real Lie algebras.
\end{proof}
We now deal with the isomorphism classes of aligned pairs.
\begin{proposition}
\label{prop:well defined}
Let $(\fg=\fg^o\otimes\CC,\sigma)$ and $(\fg'=\fg'^o\otimes\CC,\sigma')$ be aligned pairs with underlying real Lie algebras
\begin{equation}
\fg^o=\bigoplus_{p\in\ZZ}\fg^o_{p}\;,\qquad\fg'^o=\bigoplus_{p\in\ZZ}\fg'^o_{p}\,.
\label{eq:gradedrealII}
\end{equation}
Then $(\fg,\sigma)$ and $(\fg',\sigma')$ are isomorphic if and only if $\fg^o$ and $\fg'^o$ are isomorphic as $\ZZ$-graded Lie algebras.
\end{proposition}
\begin{proof}
\vskip0.1cm\par
$(\Longleftarrow)$ If the $\ZZ$-graded real Lie algebras $\fg^o$ and $\fg'^o$ are isomorphic, then, without any loss of generality, we may assume that they coincide. Hence $\sigma, \sigma':\fg\to\fg$ are the $\CC$-linear extensions of grade-reversing Cartan involutions $\theta, \theta':\fg^o\to\fg^o$ and, by Corollary \ref{cor:equiv0deg}, there exists a zero-degree automorphism $\psi:\fg^o\to\fg^o$ such that $\psi\circ\theta=\theta'\circ\psi$. The $\CC$-linear extension $\phi:\fg\to\fg$ of $\psi$ is the required isomorphism of aligned pairs.
\vskip0.1cm\par\noindent
$(\Longrightarrow)$ Let $\phi:\fg\to\fg'$ be a zero-degree Lie algebra isomorphism such that $\phi\circ\sigma=\sigma'\circ\phi$. We replace $\fg'^o$ by the isomorphic $\ZZ$-graded real Lie algebra $\widetilde\fg^o:=\phi^{-1}(\fg'^o)$ with grade-reversing Cartan involution
$\widetilde\theta=\phi^{-1}\circ\theta'\circ\phi|_{\widetilde\fg^o}:\widetilde\fg^o\to\widetilde\fg^o$.
The $\CC$-linear extension of $\widetilde\theta$ is $\sigma=\phi^{-1}\circ\sigma'\circ\phi:\fg\to\fg$, hence Lemma \ref{lem:simpliedproof} applies and $\fg^o$ and $\widetilde\fg^o$ are isomorphic.
\end{proof}
In summary, we have proved most of the following.
\begin{theorem}
\label{thm:graderevCartan}
Let $(\fg,\sigma)$ be a pair consisting of a $\ZZ$-graded simple complex Lie algebra $\fg=\bigoplus_{p\in\ZZ}\fg_{p}$ and a grade-reversing involution $\sigma:\fg\to\fg$. Then $(\fg,\sigma)$ is the aligned pair associated to a $\ZZ$-graded real form $\fg^o=\bigoplus_{p\in\ZZ}\fg^o_p$ of $\fg$.
Isomorphic pairs correspond exactly to isomorphic $\ZZ$-graded real forms. 
\end{theorem}
\begin{proof}
In view of Proposition \ref{prop:well defined}, it remains only to show that $(\fg,\sigma)$ is an aligned pair. By Proposition \ref{prop:commuting}, there exists a grade-reversing Cartan involution $\vartheta:\fg\to\fg$ which commutes with $\sigma$, we denote by $\fu^o$ the associated compact real form of $\fg$. We note that $\sigma(\fu^o)=\fu^o$ and let $\fu^o=\fk\oplus i\fp$ be the $\pm 1$-eigenspace decomposition of $\sigma|_{\fu^o}$. 

Let $\fg^o=\fk\oplus\fp$ be the real form of $\fg$ with Cartan involution $\theta=\sigma|_{\fg^o}:\fg^o\to\fg^o$ and $E$ the grading element of $\fg$, that is, the unique element $E\in\fg$ satisfying $[E,x]=p x$ for all $x\in\fg_p$. We have
\begin{align*}
\sigma E&=-E\Longrightarrow E\in\fp\oplus i\fp,\\
\vartheta E&=-E\Longrightarrow E\in i\fk\oplus \fp,
\end{align*}
hence $E\in\fp\subset\fg^o$. It follows that $\fg^o$ is a $\ZZ$-graded real form of $\fg$ and $\sigma$ is the $\CC$-linear extension of the grade-reversing Cartan involution $\theta:\fg^o\to\fg^o$.
\end{proof}

\subsection{Classification of finite dimensional Kantor triple systems}\label{sub:classKTS}
In this section, we describe all fundamental $5$-gradings $\fg=\fg_{-2}\oplus\cdots\oplus\fg_{2}$ of finite dimensional complex simple Lie algebras and their real forms. When $\dim\fg_{\pm 2}=1$, such gradings are usually called of {\it contact type}; in the general case we call them {\it admissible}. Admissible real gradings are in one-to-one correspondence with finite dimensional $K$-simple complex KTS, cf. Theorems \ref{thm:correspondence}, \ref{thm:simplicity}, \ref{thm:graderevCartan}.
\vskip0.1cm\par
We first recall the description of $\ZZ$-gradings of complex simple Lie algebras (see e.g. \cite{MR1064110}). Let $\fg$ be a complex simple Lie algebra. Fix a Cartan subalgebra $\mathfrak{h}\subset\fg$, denote by $\Delta=\Delta(\fg,\mathfrak{h})$ the root system and by 
\begin{equation*}
\fg^\alpha=\left\{X\in\fg\;|\;[H,X]=\alpha(H)X\;\;\text{for all}\;\; H\in \fh\phantom{C^{C^C}}\!\!\!\!\!\!\!\!\!\!\!\right\}
\end{equation*}
the associated root space of $\alpha\in\Delta$.
Let $\mathfrak{h}_{\RR}\subset\mathfrak{h}$ be the real subspace where all the roots are real valued; 
any element $\lambda\in(\mathfrak{h}_\RR^*)^*\simeq\mathfrak{h}_\RR$ with $\lambda(\alpha)\in\mathbb{Z}$ for all $\alpha\in\Delta$ defines a $\ZZ$-grading $\displaystyle\fg=\bigoplus_{p\in\ZZ}\fg_p$ on $\fg$ by setting:
\[  \fg_0=\;\;\mathfrak{h}\oplus
\displaystyle\bigoplus_{\substack{\alpha\in\Delta\\\lambda(\alpha)=0}}\fg^\alpha\;,\qquad\ \fg_p=\displaystyle\bigoplus_{\substack{\alpha\in\Delta\\\lambda(\alpha)=p}}\fg^\alpha\;,\;\; \text{for all}\,\; p\in\ZZ^\times, \]
and all possible gradings of $\fg$ are of this form, for some choice of $\mathfrak{h}$ and $\lambda$. We refer to $\lambda(\alpha)$ as the \emph{degree} of the root $\alpha$.

There exists a set of positive roots $\Delta^+\subset\Delta$ such that $\lambda$ is dominant, i.e., $\lambda(\alpha)\geq0$ for all $\alpha\in\Delta^+$. Let $\Pi=\left\{\alpha_1,\ldots,\alpha_\ell\right\}$ be the set of positive simple roots, which we identify with the nodes of the Dynkin diagram.
The depth of $\fg$ is the degree $d=\lambda(\alpha_{\text{max}})$ of the maximal root 
$\alpha_{\text{max}}=\sum_{i=1}^\ell m_i\alpha_i$.
 A grading is fundamental if and only if $\lambda(\alpha)\in\{0,1\}$ for all simple roots $\alpha\in\Pi$. Fundamental gradings on $\fg$ are denoted by marking with a
cross the nodes of the Dynkin diagram of $\fg$ corresponding to simple roots $\alpha$ with $\lambda(\alpha)=1$. 

The Lie subalgebra $\fg_0$ is reductive; the Dynkin diagram of its semisimple ideal is obtained from the Dynkin diagram of $\fg$ by removing all crossed nodes, and any line issuing from them. 
\begin{proposition}
\label{prop:admbij}
There exists a bijection from the isomorphism classes of fundamental $\ZZ$-gradings of complex simple Lie algebras and the isomorphism classes of marked Dynkin diagrams. In particular admissible gradings correspond to Dynkin diagrams with either one marked root $\alpha_i$ with Dynkin label $m_i=2$ or two marked roots $\alpha_i,\alpha_j$ with Dynkin labels $m_i=m_j=1$.
\end{proposition}
Table \ref{tab:dynkinDiag} below lists  the Dynkin diagrams of all complex simple Lie algebras together with the Dynkin label $m_i$ for each simple root $\alpha_i$ and the group of outer automorphisms. Proposition \ref{prop:admbij} together with a routine examination of the diagrams gives the admissible $\ZZ$-gradings of each complex simple Lie algebra. Their overall number is illustrated in Table \ref{tab:dynkinDiag}.
\vskip0.3cm\par\noindent
{\tiny
\begin{table}[H]
\begin{tabular}{|c|c|c|c|c|}
\hline
\hline
$\fg$ & \text{Dynkin diagram} & $\#\;\text{nodes}$ & $\Out(\fg)$ & $\#\;\text{admissible gradings}$\\
\hline
\hline
$\fsl(\ell+1,\CC) $&
\begin{tikzpicture}
\node[root]   (1) [label=\;\;\;${{}^{1^{\phantom{T^T}}}}$] {} ;
\node[root](2) [right=of 1] [label=${{}^1}$]{} edge [-] (1);
\node[]   (3) [right=of 2] {$\;\cdots\,$} edge [-] (2);
\node[root]   (4) [right=of 3] [label=${{}^1}$]{} edge [-] (3);
\end{tikzpicture}&
$\ell\geq 1$ & $\begin{gathered} \1\;\;\;\text{if}\;\ell=1\\ \ZZ_2\;\text{if}\;\ell>1\end{gathered}$ & $\begin{gathered}\\ 
\frac{\ell^2-1}{4}\;\;\;\text{if}\;\ell\;\text{is}\;\text{odd}
\\ \\ 
\;\;\;\;\frac{\ell^2}{4}\;\;\;\;\;\;\text{if}\;\ell\;\text{is}\;\text{even}
\\ \\
\end{gathered}$    
\\
\hline

$\fso(2\ell+1,\CC)$&
\begin{tikzpicture}
\node[root]   (1)   [label=\;\;\;${{}^{1^{\phantom{T^T}}}}$]                  {};
\node[root] (2) [right=of 1] [label=${{}^2}$]{} edge [-] (1);
\node[] (3) [right=of 2] {$\;\cdots\,$} edge [-] (2);
\node[root]   (4) [right=of 3] [label=${{}^2}$]{} edge [-] (3);
\node[root]   (5) [right=of 4] [label=${{}^2}$]{} edge [rdoublearrow] (4);
\end{tikzpicture}&
$\ell\geq 2$ & $\1$ & $\ell-1$
\\
\hline

$\mathfrak{sp}(\ell,\CC)$&
\begin{tikzpicture}
\node[root]   (1)   [label=\;\;\;${{}^{2^{\phantom{T^T}}}}$]                  {};
\node[root] (2) [right=of 1] [label=${{}^2}$] {} edge [-] (1);
\node[]   (3) [right=of 2] {$\;\cdots\,$} edge [-] (2);
\node[root]   (4) [right=of 3] [label=${{}^2}$] {} edge [-] (3);
\node[root] (5) [right=of 4] [label=${{}^2}$] {} edge [-] (4);
\node[root]   (6) [right=of 5] [label=${{}^1}$] {} edge [doublearrow] (5);
\end{tikzpicture}&
$\ell \geq 3$ & $\1$ & $\ell-1$
\\
\hline

$\begin{gathered} \fso(2\ell,\CC)\\ \\ \end{gathered}$&
\begin{tikzpicture}
\node[root]   (1)  [label=${{}^1}$]                   {};
\node[root] (2) [right=of 1] [label=${{}^2}$]{} edge [-] (1);
\node[]   (3) [right=of 2] {$\;\cdots\,$} edge [-] (2);
\node[root]   (4) [right=of 3] [label=${{}^2}$]{} edge [-] (3);
\node[root]   (5) [right=of 4] [label=${{}^2}$] {} edge [-] (4);
\node[root]   (6) [above right=of 5] [label=\;\;\;${{}^{1^{\phantom{T^T}}}}$]{} edge [-] (5);
\node[root]   (7) [below right=of 5] [label=${{}^1}$]{} edge [-] (5);
\end{tikzpicture}&
$\begin{gathered} \ell\geq 4\\ \\ \end{gathered}$ & $\begin{gathered} S_3\;\,\text{if}\;\ell=4\\ \ZZ_2\;\text{if}\;\ell>4\\ \\ \end{gathered}$ & $\begin{gathered} 2\;\;\;\;\;\text{if}\;\ell=4\\ \!\!\!\!\ell-1\;\text{if}\;\ell>4\\ \\ \end{gathered}$
\\
\hline

$G_2$&
\begin{tikzpicture}
\node[root]   (1)  [label=\;\;\;${{}^{3^{\phantom{T^T}}}}$]                   {};
\node[root] (2) [right=of 1] [label=${{}^2}$] {}   edge [ltriplearrow] (1) edge [-] (1);
\end{tikzpicture}&
$2$ & $\1$ & $1$\\
\hline

$F_4$&
\begin{tikzpicture}
\node[root]   (1)   [label=${{}^{2^{\phantom{T^T}}}}$]                  {};
\node[root] (2) [right=of 1] [label=${{}^3}$] {} edge [-] (1);
\node[root]   (3) [right=of 2] [label=${{}^4}$] {} edge [rdoublearrow] (2);
\node[root]   (4) [right=of 3] [label=${{}^2}$] {} edge [-] (3);
\end{tikzpicture}&
$4$ & $\1$ & $2$\\
\hline

$E_6$&
\begin{tikzpicture}
\node[root]   (1)  [label=\;\;\;${{}^{1^{\phantom{T^T}}}}$]                   {};
\node[root] (2) [right=of 1] [label=${{}^2}$]{} edge [-] (1);
\node[root]   (3) [right=of 2] [label=${{}^3}$]{} edge [-] (2);
\node[root]   (4) [right=of 3] [label=${{}^2}$]{} edge [-] (3);
\node[root]   (5) [right=of 4] [label=${{}^1}$]{} edge [-] (4);
\node[root]   (6) [below=of 3] [label=\;\;\;\;${{}_2}$] {} edge [-] (3);
\end{tikzpicture}&
$6$ & $\ZZ_2$ & $3$
\\
\hline

$E_7$&
\begin{tikzpicture}
\node[root]   (1) [label=\;\;\;${{}^{2^{\phantom{T^T}}}}$]             {};
\node[root]   (3) [right=of 1] [label=${{}^3}$]{} edge [-] (1);
\node[root]   (4) [right=of 3] [label=${{}^4}$]{} edge [-] (3);
\node[root]   (5) [right=of 4] [label=${{}^3}$]{} edge [-] (4);
\node[root]  (6) [right=of 5] [label=${{}^2}$]{} edge [-] (5);
\node[root]   (7) [right=of 6] [label=${{}^1}$]{} edge [-] (6);
\node[root]   (2) [below=of 4] [label=\;\;\;\;${{}_2}$]{} edge [-] (4);
\end{tikzpicture}&
$7$ & $\1$ & $3$
\\
\hline

$E_8$&
\begin{tikzpicture}
\node[root]  (1)  [label=\;\;\;${{}^{2^{\phantom{T^T}}}}$]            {};
\node[root]   (3) [right=of 1] [label=${{}^4}$]{} edge [-] (1);
\node[root]   (4) [right=of 3] [label=${{}^6}$]{} edge [-] (3);
\node[root]   (5) [right=of 4] [label=${{}^5}$]{} edge [-] (4);
\node[root]   (6) [right=of 5] [label=${{}^4}$]{} edge [-] (5);
\node[root]   (7) [right=of 6] [label=${{}^3}$]{} edge [-] (6);
\node[root]   (2) [below=of 4] [label=\;\;\;\;${{}_3}$]{} edge [-] (4);
\node[root]   (8) [right=of 7] [label=${{}^2}$]{} edge [-] (7);
\end{tikzpicture}&
$8$ & $\1$ & $2$
\\
\hline
\hline
\end{tabular}
\vskip0.1cm\par\noindent
\caption{Dynkin diagrams with labels of complex simple Lie algebras, their group of outer automorphisms and number of admissible gradings.}
\label{tab:dynkinDiag}
\end{table}
}
\vskip0.3cm\par\noindent

We now recall the description of $\ZZ$-gradings of real simple Lie algebras, cf. \cite{MR661861}. Let $\fg^o$ be a real simple Lie algebra. Fix a Cartan decomposition $\fg^o=\mathfrak{k}\oplus\mathfrak{p}$, a maximal abelian subspace $\fh_{\circ}\subset\mathfrak p$ and a maximal torus $\fh_\bullet$ in the centralizer of $\fh_\circ$ in $\mathfrak k$. Then $\fh^o=\fh_\bullet\oplus\fh_\circ$ is a maximally noncompact Cartan subalgebra of $\fg^o$.

Denote by $\Delta=\Delta(\fg,\fh)$ the root system of $\fg=\fg^o\otimes\CC$ with respect to $\fh=\fh^o\otimes\CC$ and by $\fh_{\RR}=\mathrm{i}\fh_\bullet\oplus\fh_\circ\subset\fh$ the real subspace where all the roots have real values. 
Conjugation $\vartheta:\fg\longrightarrow\fg$ of $\fg$ with respect to the real form $\fg^o$ leaves $\fh$ invariant and induces an 
involution $\alpha\mapsto\bar\alpha$ on $\fh_\RR^*$, trasforming roots into roots.
We say that a root $\alpha$ is compact if $\bar\alpha=-\alpha$ and denote by $\Delta_\bullet$ the set of compact roots.
There exists a set of positive roots $\Delta^+\subset\Delta$, with corresponding system of simple roots $\Pi$, and an involutive automorphism $\varepsilon\colon\Pi\to\Pi$ of the Dynkin diagram of $\fg$ such that $\bar\alpha=-\alpha$ for all $\alpha\in\Pi\cap\Delta_\bullet$, $\varepsilon(\Pi\setminus\Delta_\bullet)\subseteq \Pi\setminus\Delta_\bullet$ and
\begin{align*}
\bar\alpha&=\varepsilon(\alpha)+\!\sum_{\beta\in\Pi\cap\Delta_\bullet}b_{\alpha,\beta}\beta\;\;\;\text{ for all }\alpha\in\Pi\setminus\Delta_\bullet\ .
\end{align*}
The Satake diagram of $\fg^o$ is the Dynkin diagram of $\fg$ with the following additional data:
\begin{enumerate}
\item
nodes in $\Pi\cap\Delta_\bullet$ are painted black;
\item
if $\alpha\in\Pi\setminus\Delta_\bullet$ and $\varepsilon(\alpha)\neq\alpha$ then $\alpha$ and $\varepsilon(\alpha)$ are joined by a curved arrow.
\end{enumerate}
A list of Satake diagrams can be found in e.g. \cite{MR1064110}.

Let $\lambda\in(\fh_\RR^*)^*\simeq\fh_\RR$ be 
an element such that the induced grading on $\fg$ is fundamental. Then the grading on $\fg$ induces a grading on $\fg^o$ if and only if $\bar\lambda=\lambda$ \cite[Theorem 3]{MR661861}, or equivalently the following two conditions on the set 
$\;\Phi=\left\{\alpha\in\Pi\mid\lambda(\alpha)=1\phantom{C^{C^C}}\!\!\!\!\!\!\!\!\!\!\!\!\right\}$ are satisfied:
$$1)\ \Phi\cap\Delta_\bullet=\emptyset\;;\qquad\qquad 2)\text{ if } \alpha\in\Phi\;\;\text{ then }\;\;\varepsilon(\alpha)\in\Phi\;.$$



{\tiny
\begin{table}[H]
\begin{tabular}{|c|c|c|c|c|c|}
\hline
\hline
$\fg^o$ & \text{Satake diagram} & $\#\;\text{nodes}$ & $\#\;\text{white nodes}$ & $\#\;\text{admissible gradings}$\\
\hline
\hline

$\fsl(\ell+1,\RR)$&
\begin{tikzpicture}
\node[root]   (1) {} ;
\node[root](2) [right=of 1] {} edge [-] (1);
\node[]   (3) [right=of 2] {$\;\cdots\,$} edge [-] (2);
\node[root]   (4) [right=of 3] {} edge [-] (3);
\end{tikzpicture}
&
$\ell\geq 1$ & $\ell$ &
$\begin{gathered}\\ 
\frac{\ell^2-1}{4}\;\;\;\text{if}\;\ell\;\text{is}\;\text{odd}
\\ \\ 
\;\;\;\;\frac{\ell^2}{4}\;\;\;\;\;\;\text{if}\;\ell\;\text{is}\;\text{even}
\\ \\
\end{gathered}$ 
\\
\hline

$\fsl(m+1,\mathbb{H})$&
\begin{tikzpicture}
\node[broot]   (1)                     {};
\node[root] (2) [right=of 1] {} edge [-] (1);
\node[broot] (3) [right=of 2] {} edge [-] (2);
\node[]   (4) [right=of 3] {$\;\cdots\,$} edge [-] (3);
\node[root]   (5) [right=of 4] {} edge [-] (4);
\node[broot]   (6) [right=of 5] {} edge [-] (5);
\end{tikzpicture}
&
$\begin{gathered}\\ \ell\geq3\\ \ell=2m+1\;\text{odd}\\ \end{gathered}$ & $m$ &
$\begin{gathered}\\ 
\frac{m^2-1}{4}\;\;\;\text{if}\;m\;\text{is}\;\text{odd}
\\ \\ 
\;\;\;\;\frac{m^2}{4}\;\;\;\;\;\;\text{if}\;m\;\text{is}\;\text{even}
\\ \\
\end{gathered}$ 
\\
\hline

$\begin{gathered} \fsu(p,\ell+1-p)\\ 1\leq p\leq [\frac{\ell}{2}] \\ \\ \end{gathered}$&
\begin{tikzpicture}
\node[root]   (1)                     {}; 
\node[root] (2) [right=of 1] {} edge [-] (1);
\node[]   (3) [right=of 2] {$\;\cdots\,$} edge [-] (2);
\node[root]   (4) [right=of 3] {} edge [-] (3);
\node[broot]   (5) [right=of 4] {} edge [-] (4);
\node[]   (6) [right=of 5] {$\;\cdots\,$} edge [-] (5);
\node[broot]   (7) [right=of 6] {} edge [-] (6);
\node[root] (8) [right=of 7] {} edge [-] (7) edge [<->,out=150, in=30] (4);
\node[]   (9) [right=of 8] {$\;\cdots\,$} edge [-] (8);
\node[root]   (10) [right=of 9] {} edge [-] (9) edge [<->,out=150, in=30] (2);
\node[root]   (11) [right=of 10] {} edge [-] (10) edge [<->,out=145, in=35] (1);
\end{tikzpicture}&
$\begin{gathered} \ell\geq2\\ \\ \end{gathered}$ & $\begin{gathered} 2p\\ \\ \end{gathered}$ & $\begin{gathered} p \\ \\ \end{gathered}$
\\
\hline

$\begin{gathered} \fsu(p+1,p+1)\\ \\ \end{gathered}$&
\begin{tikzpicture}
\node[root]   (1)                     {}; 
\node[root] (2) [right=of 1] {} edge [-] (1);
\node[]   (3) [right=of 2] {$\;\cdots\,$} edge [-] (2);
\node[root]   (4) [right=of 3] {} edge [-] (3);
\node[root]   (5) [right=of 4] {} edge [-] (4);
\node[root] (6) [right=of 5] {} edge [-] (5) edge [<->,out=150, in=30] (4);
\node[]   (7) [right=of 6] {$\;\cdots\,$} edge [-] (6);
\node[root]   (8) [right=of 7] {} edge [-] (7) edge [<->,out=150, in=30] (2);
\node[root]   (9) [right=of 8] {} edge [-] (8) edge [<->,out=145, in=35] (1);
\end{tikzpicture}&
$\begin{gathered} \ell\geq3\\ \ell=2p+1\;\text{odd}\\ \\ \end{gathered}$ & $\begin{gathered} \ell\\ \\ \end{gathered}$ & $\begin{gathered} p \\ \\ \end{gathered}$
\\
\hline

$\begin{gathered}\\ \fso(p,2\ell+1-p)
\\ 1\leq p \leq\ell \\ \\\end{gathered}$& 
\begin{tikzpicture}
\node[root]   (1)                     {};
\node[]   (2) [right=of 1] {$\;\cdots\,$} edge [-] (1);
\node[root] (3) [right=of 2] {} edge [-] (2);
\node[broot]   (4) [right=of 3] {} edge [-] (3);
\node[]   (5) [right=of 4] {$\;\cdots\,$} edge [-] (4);
\node[broot] (6) [right=of 5] {} edge [-] (5);
\node[broot]   (7) [right=of 6] {} edge [rdoublearrow] (6);
\end{tikzpicture}&
$\ell\geq 2$ & $p$ & $p-1$
\\
\hline

$\begin{gathered}\\ \mathfrak{sp}(\ell,\RR) \\ \\ \end{gathered}$&
\begin{tikzpicture}
\node[root]   (1)                   {};
\node[root] (2) [right=of 1]  {} edge [-] (1);
\node[]   (3) [right=of 2] {$\;\cdots\,$} edge [-] (2);
\node[root]   (4) [right=of 3]  {} edge [-] (3);
\node[root] (5) [right=of 4]  {} edge [-] (4);
\node[root]   (6) [right=of 5]  {} edge [doublearrow] (5);
\end{tikzpicture}&
$\ell \geq 3$ & $\ell$ & $\ell-1$
\\
\hline

$\begin{gathered}\\ \mathfrak{sp}(p,\ell-p)\\ 1\leq p\leq [\frac{\ell-1}{2}]\\ \\ \end{gathered}$&
\begin{tikzpicture}
\node[broot]   (1)                     {};
\node[root] (2) [right=of 1] {} edge [-] (1);
\node[broot] (3) [right=of 2] {} edge [-] (2);
\node[]   (4) [right=of 3] {$\;\cdots\,$} edge [-] (3);
\node[root]   (5) [right=of 4] {} edge [-] (4);
\node[broot]   (6) [right=of 5] {} edge [-] (5);
\node[]   (7) [right=of 6] {$\;\cdots\,$} edge [-] (6);
\node[broot]   (8) [right=of 7] {} edge [-] (7);
\node[broot]   (9) [right=of 8]  {} edge [doublearrow] (8);
\end{tikzpicture}
&
$\ell\geq3$ & $p$ &
$p$ 
\\
\hline

$\begin{gathered}\\ \mathfrak{sp}(p,p)\\ \\ \end{gathered}$&
\begin{tikzpicture}
\node[broot]   (1)                     {};
\node[root] (2) [right=of 1] {} edge [-] (1);
\node[broot] (3) [right=of 2] {} edge [-] (2);
\node[]   (4) [right=of 3] {$\;\cdots\,$} edge [-] (3);
\node[root]   (5) [right=of 4] {} edge [-] (4);
\node[broot]   (6) [right=of 5] {} edge [-] (5);
\node[root]   (7) [right=of 6]  {} edge [doublearrow] (6);
\end{tikzpicture}
&
$\begin{gathered}\\ \ell\geq 4 \\ \ell=2p\;\text{even}\\ \\ \end{gathered}$ & $p$ &
$p-1$ 
\\
\hline

$\begin{gathered} \fso(\ell,\ell)\\ \\ \end{gathered}$&
\begin{tikzpicture}
\node[root]   (1)                   {};
\node[root] (2) [right=of 1] {} edge [-] (1);
\node[]   (3) [right=of 2] {$\;\cdots\,$} edge [-] (2);
\node[root]   (4) [right=of 3] {} edge [-] (3);
\node[root]   (5) [right=of 4]  {} edge [-] (4);
\node[root]   (6) [above right=of 5] [label=${{}^{\phantom{T}^{\phantom{T^T}}}}$]{} edge [-] (5);
\node[root]   (7) [below right=of 5] {} edge [-] (5);
\end{tikzpicture}&
$\begin{gathered} \ell\geq 4\\ \\ \end{gathered}$ & $\begin{gathered} \ell \\ \\ \end{gathered}$ & $\begin{gathered} 2\;\;\;\;\;\text{if}\;\ell=4\\ \!\!\!\!\ell-1\;\text{if}\;\ell>4\\ \\ \end{gathered}$
\\
\hline

$\begin{gathered} \fso(\ell-1,\ell+1)\\ \\ \end{gathered}$&
\begin{tikzpicture}
\node[root]   (1)                   {};
\node[root] (2) [right=of 1] {} edge [-] (1);
\node[]   (3) [right=of 2] {$\;\cdots\,$} edge [-] (2);
\node[root]   (4) [right=of 3] {} edge [-] (3);
\node[root]   (5) [right=of 4]  {} edge [-] (4);
\node[root]   (6) [above right=of 5] [label=${{}^{\phantom{T}^{\phantom{T^T}}}}$]{} edge [-] (5);
\node[root]   (7) [below right=of 5] {} edge [-] (5) edge [<->] (6);
\end{tikzpicture}&
$\begin{gathered} \ell\geq 4\\ \\ \end{gathered}$ & $\begin{gathered} \ell \\ \\ \end{gathered}$ & $\begin{gathered} 2\;\;\;\;\;\text{if}\;\ell=4\\ \!\!\!\!\ell-2\;\text{if}\;\ell>4\\ \\ \end{gathered}$
\\
\hline

$\begin{gathered} \fso(p,2\ell-p)\\ 1\leq p\leq \ell-2 \\ \\ \end{gathered}$&
\begin{tikzpicture}
\node[root]   (1)                   {};
\node[root] (2) [right=of 1] {} edge [-] (1);
\node[]   (3) [right=of 2] {$\;\cdots\,$} edge [-] (2);
\node[root]   (4) [right=of 3] {} edge [-] (3);
\node[broot]   (5) [right=of 4]  {} edge [-] (4);
\node[]   (6) [right=of 5] {$\;\cdots\,$} edge [-] (5);
\node[broot]   (7) [right=of 6]  {} edge [-] (6);
\node[broot]   (8) [above right=of 7] [label=${{}^{\phantom{T}^{\phantom{T^T}}}}$]{} edge [-] (7);
\node[broot]   (9) [below right=of 7] {} edge [-] (7);
\end{tikzpicture}&
$\begin{gathered} \ell\geq 4\\ \\ \end{gathered}$ & $\begin{gathered} p \\ \\ \end{gathered}$ & $\begin{gathered} p-1\\ \\ \end{gathered}$
\\
\hline

$\begin{gathered} \fso^*(2\ell)\\ \\ \end{gathered}$&
\begin{tikzpicture}
\node[broot]   (1)                   {};
\node[root] (2) [right=of 1] {} edge [-] (1);
\node[]   (3) [right=of 2] {$\;\cdots\,$} edge [-] (2);
\node[broot]   (4) [right=of 3]  {} edge [-] (3);
\node[root]   (5) [right=of 4]  {} edge [-] (4);
\node[broot]   (6) [above right=of 5] [label=${{}^{\phantom{T}^{\phantom{T^T}}}}$]{} edge [-] (5);
\node[root]   (7) [below right=of 5] {} edge [-] (5);
\end{tikzpicture}&
$\begin{gathered} \ell\geq 6\\ l=2m\;\text{even}\\ \\ \end{gathered}$ & $\begin{gathered} m \\ \\ \end{gathered}$ & $\begin{gathered} m-1 \\ \\ \end{gathered}$
\\
\hline

$\begin{gathered} \fso^*(2\ell)\\ \\ \end{gathered}$&
\begin{tikzpicture}
\node[broot]   (1)                   {};
\node[root] (2) [right=of 1] {} edge [-] (1);
\node[broot] (3) [right=of 2] {} edge [-] (2);
\node[]   (4) [right=of 3] {$\;\cdots\,$} edge [-] (3);
\node[root]   (5) [right=of 4]  {} edge [-] (4);
\node[broot]   (6) [right=of 5]  {} edge [-] (5);
\node[root]   (7) [above right=of 6] [label=${{}^{\phantom{T}^{\phantom{T^T}}}}$]{} edge [-] (6);
\node[root]   (8) [below right=of 6] {} edge [-] (6) edge [<->] (7);
\end{tikzpicture}&
$\begin{gathered} \ell\geq 5\\ \ell=2m+1\;\text{odd}\\ \\ \end{gathered}$ & $\begin{gathered} m+1 \\ \\ \end{gathered}$ & $\begin{gathered} m \\ \\ \end{gathered}$
\\
\hline
\hline
\end{tabular}
\vskip0.1cm\par\noindent
\caption{Satake diagrams of absolutely simple non-compact real Lie algebras of classical type and their number of admissible gradings.}
\label{tab:satakeClassical}
\end{table}
}

{\tiny
\begin{table}[H]
\begin{tabular}{|c|c|c|c|c|c|}
\hline
\hline
$\fg^o$ & \text{Satake diagram} & $\#\;\text{nodes}$ & $\#\;\text{white nodes}$ & $\#\;\text{admissible gradings}$\\
\hline
\hline
$\begin{gathered} \\ G_2\\ \text{split}\\ \\ \end{gathered}$&
\begin{tikzpicture}
\node[root]   (1)                   {};
\node[root] (2) [right=of 1]  {}   edge [ltriplearrow] (1) edge [-] (1);
\end{tikzpicture}&
$2$ & $2$ & $1$\\
\hline 

$\begin{gathered} \\ F\,I \\ \\ \end{gathered}$&
\begin{tikzpicture}
\node[root]   (1)                     {};
\node[root] (2) [right=of 1] {} edge [-] (1);
\node[root]   (3) [right=of 2] {} edge [rdoublearrow] (2);
\node[root]   (4) [right=of 3] {} edge [-] (3);
\end{tikzpicture}&
$4$ & $4$ & $2$
\\
\hline

$\begin{gathered} \\ F\,II \\ \\ \end{gathered}$&
\begin{tikzpicture}
\node[broot]   (1)                     {};
\node[broot] (2) [right=of 1] {} edge [-] (1);
\node[broot]   (3) [right=of 2] {} edge [rdoublearrow] (2);
\node[root]   (4) [right=of 3] {} edge [-] (3);
\end{tikzpicture}&
$4$ & $1$ & $1$
\\
\hline

$\begin{gathered}\\ E\,I \\ \\ \end{gathered}$&
\begin{tikzpicture}
\node[root]   (1)                     {};
\node[root] (2) [right=of 1] [label=${}^{\phantom{T^T}}$]{} edge [-] (1);
\node[root]   (3) [right=of 2] {} edge [-] (2);
\node[root]   (4) [right=of 3] {} edge [-] (3);
\node[root]   (5) [right=of 4] {} edge [-] (4);
\node[root]   (6) [below=of 3] {} edge [-] (3);
\end{tikzpicture}&
$\begin{gathered} 6\\ \\ \end{gathered}$ & $\begin{gathered} 6\\ \\ \end{gathered}$ & $\begin{gathered} 3 \\ \\ \end{gathered}$
\\
\hline

$\begin{gathered}\\ E\,II \\ \\ \end{gathered}$&
\begin{tikzpicture}
\node[root]   (1)                     {};
\node[root] (2) [right=of 1] {} edge [-] (1);
\node[root]   (3) [right=of 2] {} edge [-] (2);
\node[root]   (4) [right=of 3] {} edge [-] (3)  edge [<->,out=150, in=30] (2);
\node[root]   (5) [right=of 4] {} edge [-] (4)  edge [<->,out=145, in=35] (1);
\node[root]   (6) [below=of 3] {} edge [-] (3);
\end{tikzpicture}&
$6$ & $6$ & $2$
\\
\hline

$\begin{gathered} \\ E\,III \\ \\ \end{gathered}$&
\begin{tikzpicture}
\node[root]   (1)                     {};
\node[broot] (2) [right=of 1] {} edge [-] (1);
\node[broot]   (3) [right=of 2] {} edge [-] (2);
\node[broot]   (4) [right=of 3] {} edge [-] (3);
\node[root]   (5) [right=of 4] {} edge [-] (4)  edge [<->,out=150, in=30] (1);
\node[root]   (6) [below=of 3] {} edge [-] (3);
\end{tikzpicture}&
$6$ & $3$ & $2$
\\
\hline

$\begin{gathered} \\ E\,IV \\ \\ \end{gathered}$&
\begin{tikzpicture}
\node[root]   (1)                     {};
\node[broot] (2) [right=of 1] [label=${}^{\phantom{T^T}}$]{} edge [-] (1);
\node[broot]   (3) [right=of 2] {} edge [-] (2);
\node[broot]   (4) [right=of 3] {} edge [-] (3);
\node[root]   (5) [right=of 4] {} edge [-] (4);
\node[broot]   (6) [below=of 3] {} edge [-] (3);
\end{tikzpicture}&
$6$ & $2$ & $1$
\\
\hline

$\begin{gathered} \\ E\,V \\ \\ \end{gathered}$&
\begin{tikzpicture}
\node[root]   (1) [label=\;\;\;${{}^{^{\phantom{T^T}}}}$]             {};
\node[root]   (3) [right=of 1] {} edge [-] (1);
\node[root]   (4) [right=of 3] {} edge [-] (3);
\node[root]   (5) [right=of 4] {} edge [-] (4);
\node[root]  (6) [right=of 5]{} edge [-] (5);
\node[root]   (7) [right=of 6]{} edge [-] (6);
\node[root]   (2) [below=of 4] {} edge [-] (4);
\end{tikzpicture}&
$7$ & $7$ & $3$
\\
\hline

$\begin{gathered} \\ E\,VI \\ \\ \end{gathered}$&
\begin{tikzpicture}
\node[root]   (1) [label=\;\;\;${{}^{^{\phantom{T^T}}}}$]             {};
\node[root]   (3) [right=of 1] {} edge [-] (1);
\node[root]   (4) [right=of 3] {} edge [-] (3);
\node[broot]   (5) [right=of 4] {} edge [-] (4);
\node[root]  (6) [right=of 5]{} edge [-] (5);
\node[broot]   (7) [right=of 6]{} edge [-] (6);
\node[broot]   (2) [below=of 4] {} edge [-] (4);
\end{tikzpicture}&
$7$ & $4$ & $2$
\\
\hline

$\begin{gathered} \\ E\,VII \\ \\ \end{gathered}$&
\begin{tikzpicture}
\node[root]   (1) [label=\;\;\;${{}^{^{\phantom{T^T}}}}$]             {};
\node[broot]   (3) [right=of 1] {} edge [-] (1);
\node[broot]   (4) [right=of 3] {} edge [-] (3);
\node[broot]   (5) [right=of 4] {} edge [-] (4);
\node[root]  (6) [right=of 5]{} edge [-] (5);
\node[root]   (7) [right=of 6]{} edge [-] (6);
\node[broot]   (2) [below=of 4] {} edge [-] (4);
\end{tikzpicture}&
$7$ & $3$ & $2$
\\
\hline

$\begin{gathered}\\ E\,VIII \\ \\ \end{gathered}$&
\begin{tikzpicture}
\node[root]  (1)  [label=\;\;\;${{}^{^{\phantom{T^T}}}}$]            {};
\node[root]   (3) [right=of 1] {} edge [-] (1);
\node[root]   (4) [right=of 3] {} edge [-] (3);
\node[root]   (5) [right=of 4] {} edge [-] (4);
\node[root]   (6) [right=of 5] {} edge [-] (5);
\node[root]   (7) [right=of 6] {} edge [-] (6);
\node[root]   (2) [below=of 4] {} edge [-] (4);
\node[root]   (8) [right=of 7] {} edge [-] (7);
\end{tikzpicture}&
$8$ & $8$ & $2$
\\
\hline

$\begin{gathered}\\ E\,IX \\ \\ \end{gathered}$&
\begin{tikzpicture}
\node[root]  (1)  [label=\;\;\;${{}^{^{\phantom{T^T}}}}$]            {};
\node[broot]   (3) [right=of 1] {} edge [-] (1);
\node[broot]   (4) [right=of 3] {} edge [-] (3);
\node[broot]   (5) [right=of 4] {} edge [-] (4);
\node[root]   (6) [right=of 5] {} edge [-] (5);
\node[root]   (7) [right=of 6] {} edge [-] (6);
\node[broot]   (2) [below=of 4] {} edge [-] (4);
\node[root]   (8) [right=of 7] {} edge [-] (7);
\end{tikzpicture}&
$8$ & $4$ & $2$
\\
\hline
\hline
\end{tabular}
\vskip0.1cm\par\noindent
\caption{Satake diagrams of absolutely simple non-compact real Lie algebras of exceptional type and their number of admissible gradings.}
\label{tab:satakeExceptional}
\end{table}
}
There exists a bijection from the isomorphism classes of fundamental $\ZZ$-gradings of real simple Lie algebras and the isomorphism classes of marked Satake diagrams. In the real case too, the Lie subalgebra $\fg_0$ is reductive and the Satake diagram of its semisimple ideal is the Satake diagram of $\fg^o$ with all crossed nodes and any line issuing from them removed. 
A grading of $\fg^o$ is admissible if and only if the induced grading on $\fg$ is admissible. 

Table \ref{tab:satakeClassical} and Table \ref{tab:satakeExceptional} above list the Satake diagrams of all real absolutely simple Lie algebras; compact real forms do not admit any gradings, hence their diagrams are not displayed. A routine examination yields the admissible $\ZZ$-gradings, their overall number is illustrated.

Theorem \ref{thm:graderevCartan} and a simple computation using Tables \ref{tab:satakeClassical} and \ref{tab:satakeExceptional} gives the enumeration of the finite dimensional KTS with 
given Tits-Kantor-Koecher Lie algebra.
\begin{corollary}
\label{cor:enumerate}
Let $\fg$ be a complex simple Lie algebra. Then the number $K(\fg)$ of the K-simple KTS up to isomorphism with associated Tits-Kantor-Koecher Lie algebra $\fg$ is:
\begin{itemize}
\item[(1)] $\fg=\fsl(\ell+1,\CC)$ with $\ell\geq 1$:
\begin{itemize}
\item[-] if $\ell=2m+1$ is odd with $m$ even then $K(\fg)=\frac{7m^2+10m}{4}$;
\item[-] if $\ell=2m+1$ is odd with $m$ odd then $K(\fg)=\frac{7m^2+10m-1}{4}$;
\item[-] if $\ell=2m$ is even then $K(\fg)=\frac{3m^2+m}{2}$;
\end{itemize} 
\item[(2)] $\fg=\fso(2\ell+1,\CC)$ with $\ell\geq 2$ then $K(\fg)=\frac{\ell(\ell-1)}{2}$;
\item[(3)] $\fg=\mathfrak{sp}(\ell,\CC)$ with $\ell\geq 3$:
\begin{itemize}
\item[-] if $\ell=2m+1$ is odd then $K(\fg)=\frac{m^2+5m}{2}$;
\item[-] if $\ell=2m$ is even then $K(\fg)=\frac{m^2+5m-4}{2}$;
\end{itemize}
\item[(4)] $\fg=\fso(2\ell,\CC)$ with $\ell\geq 4$: 
\begin{itemize}
\item[-] if $\ell=2m+1$ is odd then $K(\fg)=2m^2+2m$;
\item[-] if $\ell=4$ then $K(\fg)=5$;
\item[-] if $\ell=2m$ is even with $m>2$ then $K(\fg)=2m^2-1$;
\end{itemize}
\item[(5)] if $\fg=G_2$ then $K(\fg)=1$;
\item[(6)] if $\fg=F_4$ then $K(\fg)=3$;
\item[(7)] if $\fg=E_6$ then $K(\fg)=8$;
\item[(8)] if $\fg=E_7$ then $K(\fg)=7$;
\item[(9)] if $\fg=E_8$ then $K(\fg)=4$.
\end{itemize}
\end{corollary}
\subsection{The algebra of derivations}\label{subsec:algDer}
In this section we study the Lie algebra 
$$\mathfrak{der}(V)=\left\{\delta:V\to V\mid \delta(xyz)=((\delta x) yz)+(x(\delta y)z)+(xy(\delta z))\right\}$$ 
of derivations of a KTS $V$. This is an important invariant of a $K$-simple KTS and, as we will shortly prove, it can easily  be described a priori by means of Theorem \ref{thm:graderevCartan}.

In \S \ref{sec:TKK} we saw that isomorphisms of 
$V$ correspond to those of $(\fg,\sigma)$. It follows that $\der(V)$
consists of (the restriction to $V=\fg_{-1}$ of) the $0$-degree derivations of $\fg$ commuting with $\sigma$. If $V$ is $K$-simple, then $\fg$ is simple, any derivation is inner and
$\mathfrak{der}(V)=\left\{D\in\fg_0\mid\sigma(D)=D\right\}$.
Now note that $\fg_0$ is $\sigma$-stable reductive with the center which is at most $2$-dimensional. When the grading corresponds to a positive subalgebra that is maximal parabolic, the center of $\fg_0$ is $1$-dimensional and generated by the grading element $E$. Clearly $\sigma(E)=-E$.
Otherwise, the marked Dynkin diagram of $\fg$ has two crossed nodes, this happens for all admissible gradings of $\fsl(\ell+1,\CC)$, two of $\fso(2\ell,\CC)$ and one of $E_6$. 

In summary, both the center and the semisimple part $\fg_0^{ss}=[\fg_0,\fg_0]$ of $\fg_0$ 
decompose into the direct sum of $\pm 1$-eigenspaces of $\sigma$ and
$$
\der(V)=\left\{D\in\fg_0^{ss}\mid \sigma(D)=D\right\}\;,
$$
possibly up to a $1$-dimensional central subalgebra. 

Let $\fg^o$ be a real form of $\fg$ compatible with the grading and
$(\fg_0^o)^{ss}=\fk_0\oplus\fp_0$ the Cartan decomposition of the semisimple part of $\fg^o_0$.
The following result holds since $\sigma:\fg\to\fg$ is the $\CC$-linear extension of the Cartan involution $\theta:\fg^o\to\fg^o$ of $\fg^o$, whose restriction $$\theta|_{(\fg^o_0)^{ss}}:(\fg^o_0)^{ss}\to (\fg^o_0)^{ss}$$ to $(\fg_0^o)^{ss}$ is still a Cartan involution, see e.g. \cite[Lemma 1.5]{MR0147566}.
\begin{theorem}
\label{thm:symmetryalg}
Let $V$ be a $K$-simple KTS and $(\fg=\fg^o\otimes\CC,\sigma)$ the associated aligned pair. Then 
$$\der(V)=\fk_0\otimes\CC$$
possibly up to a $1$-dimensional central subalgebra.
\end{theorem}
The list of Cartan decompositions of real semisimple Lie algebras can be found in \cite{MR1064110}. Theorem \ref{thm:symmetryalg} and an a priori identification of the algebra of derivations as the following example shows will be crucial ingredients to classify the KTS of exceptional type.
\begin{example}
\label{ex:G2I}
By Corollary \ref{cor:enumerate}, there is just one KTS $V$ with associated Tits-Kantor-Koecher algebra $\fg=G_2$. The aligned real form $\fg^o$ is the split form with marked Satake diagram 
$$
\begin{tikzpicture}
\node[root]   (1)                   {};
\node[xroot] (2) [right=of 1]  {}   edge [ltriplearrow] (1) edge [-] (1);
\end{tikzpicture}
$$
see Tables \ref{tab:dynkinDiag} and \ref{tab:satakeExceptional}. It follows $(\fg_0^o)^{ss}\simeq \fsl(2,\RR)$, the maximal compact subalgebra $\fk_0\simeq\fso(2,\RR)$, and $\der(V)\simeq\fso(2,\CC)$. 
\end{example}
\section{Classical Kantor triple systems}\label{sec:classicalKTS}
In this section we first give explicit constructions of KTS based on matrix algebras and then classify the K-simple KTS associated to {\it classical} complex Lie algebras. The section ends with an observation on subsystems of generalized Jordan triple systems of higher kinds.  
\subsection{Main examples}
\label{subsec:mainexamples}
The following examples detail general constructions of KTS over the field $\F$ of complex or real numbers, generalizing similar results for special Jordan triple systems, see \cite{2372102}.
We introduce some notation for the complex KTS in view of \S \ref{subsec:classclass}.
\begin{example}[$\mathfrak{Ksl}(m,n,r)$]
\label{exm:Ksl}
Let $A$ be an associative algebra and ${}^*:A\rightarrow A$ an anti-involution of $A$, 
i.e., for every $x,y\in A$, $(x^*)^*=x$ and $(xy)^*=y^*x^*$.  
Then $A\oplus A$ with the $3$-product
\begin{equation}
\mat{\mat{x_1\\x_2}, \mat{y_1\\y_2}, \mat{z_1\\z_2}}=\mat{x^{\ph}_1y_1^*z^{\ph}_1+z^{\ph}_1y_1^*x^{\ph}_1-y_2^*x^{\ph}_2z^{\ph}_1\\
							  x^{\ph}_2y_2^*z^{\ph}_2+z^{\ph}_2y_2^*x^{\ph}_2-z^{\ph}_2x^{\ph}_1y_1^*}
\label{eq:specialKTSTranspose}
\end{equation}
is a KTS. 

For instance, if $A=M_n(\F)$ is the space of $n\times n$ matrices with entries in $\F$ and ${}^*$ is the usual transposition ${}^t: 
x\mapsto x^t$, then $A\oplus A$ with the $3$-product \eqref{eq:specialKTSTranspose} is a KTS. The same holds if $n=2k$ is even and ${}^*$ is the symplectic transposition
\begin{equation}
\label{eq:sympltran}
{}^{st}: x\mapsto J_{2k}x^tJ_{2k}^{-1}\;,
\end{equation}
where
\begin{equation}
\label{eq:gei}
J_{2k}=\mat{0&S_k\\-S_k&0}
\end{equation}
and $S_k$ is the $k\times k$ matrix with $1$ on the anti-diagonal and $0$ elsewhere.

More generally, one can consider $M_{m,n}(\F)\oplus M_{r,m}(\F)$ 
and an invertible linear map $${}^*: M_{m,n}(\F)\oplus M_{r,m}(\F) \rightarrow M_{n,m}(\F)\oplus M_{m,r}(\F)$$ compatible with the two direct sum decompositions and which satisfies the 
properties:
\begin{equation}
\label{eq:antiInvolution}
\begin{array}{cc}
(x^{\ph}_1y_1^*z^{\ph}_1)^*=z_1^*y^{\ph}_1x_1^*\;,& (x^{\ph}_2y_2^*z^{\ph}_2)^*=z_2^*y^{\ph}_2x_2^*\;,\\
(x_2^*y^{\ph}_2z^{\ph}_1)^*=z_1^*y_2^*x^{\ph}_2\;,& (z^{\ph}_2y^{\ph}_1x_1^*)^*=x^{\ph}_1y_1^*z_2^*\;,
\end{array}
\end{equation} 
for $x_1,y_1,z_1\in M_{m,n}(\F),$ $x_2,y_2,z_2\in M_{r,m}(\F)$.
Then $M_{m,n}(\F)\oplus M_{r,m}(\F)$ with $3$-product \eqref{eq:specialKTSTranspose} is a KTS. Note that if $m=n=r$, then property \eqref{eq:antiInvolution}  precisely says
that ${}^*$ is an anti-involution.

As main examples over $\F=\CC$ we introduce the triple systems ${\mathfrak{Ksl}(m,n,r;t)}$, given by the vector spaces $M_{m,n}(\CC)\oplus M_{r,m}(\CC)$ with $3$-product \eqref{eq:specialKTSTranspose} and
\begin{equation*}\label{non-symplectic}
 {}^{*}: \mat{x_1\\x_2}\mapsto \mat{x_1^t\\ x_2^t}\;.
\end{equation*} 
Likewise, if $m=2h,\ n=2k$ and $r=2l$ are even, it is easy to see that
 the map
\begin{equation*}\label{symplectic}
 {}^{*}: \mat{x_1\\x_2}\mapsto \mat{J_{2k}\, x_1^t\, J_{2h}^{-1}\\J_{2h}\, x_2^t\, J_{2l}^{-1}}\;,
\end{equation*} 
satisfies \eqref{eq:antiInvolution} and we denote the corresponding KTS by ${\mathfrak{Ksl}(m,n,r;st)}$.
\end{example}

\begin{example}[$\mathfrak{Ksl}(m,n)$]
\label{exm:suKTS}
Let $\varphi:M_{m,n}(\F)\to M_{m,n}(\F)$ and $\psi:M_{n,m}(\F)\to M_{n,m}(\F)$ be any invertible linear maps of the form 
 \begin{equation*}
 \label{eq:involutionSUAssociative}
  \varphi(x)=BxA\;,\qquad\psi(x)=AxB\;,
 \end{equation*}
for some $A\in GL_n(\F)$ and $B\in GL_m(\F)$ such that $B^2=\pm Id$.
Then $M_{m,n}(\F)\oplus M_{n,m}(\F)$ with the $3$-product 
\begin{equation*}
\mat{\mat{x_1\\x_2}, \mat{y_1\\y_2}, \mat{z_1\\z_2}}=
\mat{x_1\psi(y_2)z_1+z_1\psi(y_2)x_1-\varphi(y_1)x_2z_1\\
x_2\varphi(y_1)z_2+z_2\varphi(y_1)x_2-z_2x_1\psi(y_2)}
\label{identity}
\end{equation*}
is a KTS.

In particular, we will use the notation $\mathfrak{Ksl}(m,n;k)$ for the KTS over $\F=\CC$ associated to $A=Id$ and $B=diag(-Id_k,Id_{m-k})$.
\end{example}
We note that the Tits-Kantor-Koecher Lie algebras associated to the KTS $\mathfrak{Ksl}(m,n,r;t)$, $\mathfrak{Ksl}(m,n,r;st)$ and $\mathfrak{Ksl}(m,n;k)$ are all isomorphic to $\fsl(N,\CC)$, for some $N$, but with different gradings and involutions. In particular the involution is
outer in the first two cases.

The Tits-Kantor-Koecher Lie algebras associated to the KTS described in the next two examples are all of orthogonal, respectively, symplectic type.
\begin{example}[$\mathfrak{Kso}(m,n)$]\label{exm:subKTSso}
 Let $A\in GL_n(\F)$, $B\in GL_m(\F)$ be such that $B^2=\varepsilon Id$, $\varepsilon=\pm 1$, and let us consider the associated KTS structure on $M_{m,n}(\F)\oplus M_{n,m}(\F)$ described in Example \ref{exm:suKTS}. 
 
For any $x\in M_{m,n}(\F)$ we define
\begin{equation}
\label{eq:sub3product}x'=S_nx^tS_m\;,
\end{equation}
where $S_k$ is the $k\times k$ matrix with $1$ on the anti-diagonal and $0$ elsewhere, and set
$$M=\{\mat{x_1\\x_2}\in M_{m,n}(\F)\oplus M_{n,m}(\F)\mid x_2=x_1'\}\;.$$ 
 One can check that $M$ is a subsystem of $M_{m,n}(\F)\oplus M_{n,m}(\F)$ whenever $A'=\varepsilon A$, $B'=\varepsilon B$.
 (Note that $A'=A$ means that $A$ is reflexive while $A'=-A$ anti-reflexive.)
 If this is the case, projecting onto the first component yields a KTS structure on $M_{m,n}(\F)$ with $3$-product
 \begin{equation*}
(xyz)=xAy'Bz+zAy'Bx-ByAx'z\;,
 \end{equation*}
for all $x,y,z\in M_{m,n}(\F)$.

The reflexive matrices 
$$A=Id\;,\qquad B=\mat{0&0&S_k\\0&Id_{m-2k}&0\\S_k&0&0}\;,$$ give rise to a family of KTS over $\F=\CC$ which we denote by $\mathfrak{Kso}(m,n;k)$. Another natural class is obtained when $n=2j$, $m=2l$ are both even and $A,B$ anti-reflexive; we let 
$$A=J_{2j}S_{2j}\;,\qquad B=iJ_{2l}S_{2l}\;,$$ 
where $J_{2k}$ is as in \eqref{eq:gei}, and
denote the associated KTS by $\mathfrak{Kso}(m,n;JS)$.  
 \end{example}

 \begin{example}[$\mathfrak{Ksp(2m,n)}$]\label{exm:subKTSsp}
 Let $A\in GL_n(\F)$, $B\in GL_{2m}(\F)$ be such that $B^2=\varepsilon Id$, $\varepsilon=\pm 1$, and consider the associated KTS structure on $M_{2m,n}(\F)\oplus M_{n,2m}(\F)$ of Example \ref{exm:suKTS}. 
 
We define $\xi(x)=S_nx^tJ_{2m}^{-1}$ for all matrices $x\in M_{2m,n}(\F)$ and set
$$M=\{\mat{x_1\\x_2}\in M_{2m,n}(\F)\oplus M_{n,2m}(\F)\mid x_2=\xi(x_1)\}\;.$$ 
If $A'=-\varepsilon A$ and $B^{st}=\varepsilon B$ (recall \eqref{eq:sympltran}, \eqref{eq:sub3product}), then $M$ is a subsystem of $M_{2m,n}(\F)\oplus M_{n,2m}(\F)$ and upon projecting onto the first factor $M_{2m,n}(\F)$ one gets a KTS with $3$-product
 \begin{equation*}
 \label{eq:sympproduct}
  (xyz)=xA\xi(y)Bz+zA\xi(y)Bx-ByA\xi(x)z\;,
 \end{equation*}
 for all $x,y,z\in M_{2m,n}(\F)$.
 
We note that the matrices $A=Id$ and $B=J_{2m}$ satisfy the required conditions with $\varepsilon=-1$ and denote the corresponding KTS over $\F=\CC$ by $\mathfrak{Ksp}(2m,n;J)$. 

If $n=2l$ is even, we may also consider $A=J_{2l}S_{2l}$ and $B=diag(-Id_k,Id_{2m-2k},-Id_k)$, as they satisfy the conditions with $\varepsilon=1$. The associated KTS is denoted by $\mathfrak{Ksp}(2m,n;k)$.
 \end{example}
 
 \begin{example}[$\mathfrak{Kar}(n)$]\label{exm:subKTSar}
 We let $\mathfrak{Kar}(n)=\CC^n\oplus Aref_n(\CC)$ be the KTS with triple product
 $$\mat{\mat{x_1\\x_2}\mat{y_1\\y_2}\mat{z_1\\z_2}}=
    \left(\begin{array}{l}
    x^{\ph}_1y_1^tz^{\ph}_1+z^{\ph}_1y_1^tx^{\ph}_1-y_2^tx^{\ph}_2z^{\ph}_1\\
    x^{\ph}_2y_2^tz^{\ph}_2+z^{\ph}_2y_2^tx^{\ph}_2-z^{\ph}_2x^{\ph}_1y_1^t-{(y_1')}^tx_1'z^{\ph}_2\end{array}\right)$$
where $Aref_n(\CC)$ is the space of complex anti-reflexive $n\times n$ matrices (recall \eqref{eq:sub3product}) .  

The Tits-Kantor-Koecher Lie algebra associated to the triple system $\mathfrak{Kar}(n)$ is $\mathfrak{so}(2n+2,\CC)$ with the grading given by marking the first and last node. The grade-reversing involution is the Chevalley involution.
 \end{example}

  \subsection{The classification of classical Kantor triple systems}
  \label{subsec:classclass}
The list, up to isomorphism, of all {\it K-simple  classical KTS over} $\RR$ is due to Kaneyuki and Asano, see \cite{MR974266, Asano1991}. We give here the classification over $\CC$, written accordingly to the examples and conventions of \S \ref{subsec:mainexamples}.
\begin{theorem}\label{thm:classicalKTSList}
 A K-simple KTS over $\CC$ with classical Tits-Kantor-Koecher Lie algebra is isomorphic to one of the following list:
  $$\begin{array}{lll}
\bullet\ \mathfrak{Ksl}(r,m,n-m-r;t),		&n\geq3,	\ 1\leq m\leq [(n-1)/2],\,1\leq r\leq n-m-1;\\[10pt]
\bullet\ \mathfrak{Ksl}(2r,2m,2(n-m-r);st),	&n\geq3,	\ 1\leq m\leq [(n-1)/2],\,1\leq r\leq n-m-1;\\[10pt]
\bullet\ \mathfrak{Ksl}(n-2m,m;k),		&n\geq3;	\ 1\leq m\leq [(n-1)/2],\,0\leq k\leq [(n-2m)/2];\\[10pt]									      
\bullet\ \mathfrak{Kso}(n-2m,m;k),		&n\geq7,\ n\neq 6,\ 2\leq m\leq [(n-1)/2],\,0\leq k\leq [(n-2m)/2];\\[10pt]
\bullet\ \mathfrak{Kso}(2(n-2m),2m;JS),	&n\geq4;\ 2\leq 2m\leq n-1;\\[10pt]
\bullet\ \mathfrak{Ksp}(2(n-m),m;J),		&n\geq3,\ 1\leq m\leq n-1;\\[10pt]
\bullet\ \mathfrak{Ksp}(2(n-2m),2m;k),		&n\geq3,\ 2\leq 2m\leq n-1,\,\,0\leq k\leq [(n-2m)/2];\\[10pt]
\bullet\ \mathfrak{Kar}(n),			&n\geq4\;.\\[10pt]
\end{array}$$
The ranges of natural numbers are chosen so that there are no isomorphic KTS in the above list.
\end{theorem}
\begin{proof}
 By Theorem \ref{thm:graderevCartan} specialized to $5$-gradings, complexification of the classical KTS of compact type classified in \cite{MR974266} exactly gives  the classification of all classical KTS over $\CC$.
\end{proof}


\begin{rem}
 We remark that from the description of classical KTS in \S \ref{subsec:mainexamples}, there are natural embeddings of triple systems
 $$\begin{array}{lll}
    \mathfrak{Kso}(m,n)&\subset&\mathfrak{Ksl}(m,n)\;,\\
    \mathfrak{Ksp}(2m,n)&\subset&\mathfrak{Ksl}(2m,n)\;.
\end{array}$$
The system $\mathfrak{Kar}(n)$ has a different interesting embedding, which is dealt with in the final remark of this section.
\end{rem}
We recall that the so-called generalized Jordan triple systems of the $\nu$-th kind are those systems with associated Tits-Kantor-Koecher Lie algebra that is $(2\nu+1)$-graded. In particular those of the $1$-st and $2$-nd kind are the usual Jordan and, respectively, Kantor triple systems.  

An interesting phenomenon happens for some JTS. The $3$-graded Lie algebras associated to the JTS $4,5,11,13$ of Table $III$, \cite[p. 110]{MR974266}, i.e., the orthogonal Lie algebras with only the first node marked, are in a natural way graded subalgebras of some $\fsl(N,\CC)$ with a $5$-grading. This gives rise to an embedding of the associated JTS in a special KTS.
In other words, simple KTS admit suitable subsystems which are of the $1$-st kind and yet simple.
 
A similar fact holds for $\mathfrak{Kar}(n)$, which admits a natural embedding in a simple generalized Jordan triple system of higher kind.
\begin{rem}
 Let $A=M_{n_2,n_1}(\CC)\oplus M_{n_3,n_2}(\CC)\oplus M_{n_4,n_3}(\CC)$ with the $3$-product
 $$\mat{\mat{x_1\\x_2\\x_3} \mat{y_1\\y_2\\y_3} \mat{z_1\\z_2\\z_3}}=
 \left(\begin{array}{l}
      x^{\ph}_1y^t_1z^{\ph}_1+z^{\ph}_1y^t_1x^{\ph}_1-y^t_2x^{\ph}_2z^{\ph}_1\\
      x^{\ph}_2y^t_2z^{\ph}_2+z^{\ph}_2y^t_2x^{\ph}_2-z^{\ph}_2x^{\ph}_1y^t_1-y^t_3x^{\ph}_3z^{\ph}_2\\
      x^{\ph}_3y^t_3z^{\ph}_3+z^{\ph}_3y^t_3x^{\ph}_3-z^{\ph}_3x^{\ph}_2y^t_2
      \end{array}\right).$$
 It can be shown that the product satisfies only condition $(i)$ of Definition \ref{def:KTS} and that it is a generalized JTS of the $3$-rd kind. We are interested in the case 
\begin{equation}
n_1=n_4=1\;,\qquad n_2=n_3=n\;,
\end{equation}
for which the associated Tits-Kantor-Koecher Lie algebra is $\fsl(2n+2,\CC)$ with the $7$-grading given by marking the nodes $\{1,n+1,2n+1\}$.

 Let $n_1,\ldots,n_4$ as above and consider the subspace of $A$ given by
 $$M=\big\{\mat{x_1\\x_2\\x_3}\in A\ |\ x_3=-x_1',\ x'_2=-x_2\big\}\;.$$ The space $M$
 is a simple subsystem of $A$ of the $2$-nd kind and it is isomorphic to $\mathfrak{Kar}(n)$. 
 
 To see this, it is sufficient to consider the subalgebra $\fso(2n+2,\CC)$ of anti-reflexive matrices of $\fsl(2n+2,\CC)$ and note that $\fso(2n+2,\CC)$ inherits a grading from $\fsl(2n+2,\CC)$ that is actually a $5$-grading, the one mentioned in Example \ref{exm:subKTSar}. Finally, the Chevalley involution of $\fsl(2n+2,\CC)$ restricts to the Chevalley involution of $\fso(2n+2,\CC)$ and our claim follows. 
\end{rem}

\section{The exceptional Kantor triple systems of extended Poincar\'e type}
\label{sec:extendedKTS}\vskip0.2cm\par
This section relies on the so-called gradings of {\it extended Poincar\'e type}, see \cite[Theorem 3.1]{MR3218266}, which have been extensively investigated by third author both in the complex and real case. We recall here only the facts that we need and refer to \cite{MR3255456, MR3218266} for more details.

Let $(U,\eta)$ be a finite-dimensional complex vector space $U$ endowed with a non-degenerate symmetric bilinear form
$\eta$ and $\Cl(U)=\Cl(U)_{\bar 0}\oplus\Cl(U)_{\bar 1}$ the associated Clifford algebra with 
its natural parity decomposition. We will make use of the notation of \cite{MR1031992} and also of \cite{MR1462223}; in particular, we adopt the following conventions:
\begin{itemize}
\item the product in $\Cl(U)$ satisfies $uv+vu=-2\eta(u,v)1$ for all $u,v\in U$.
\item the symbol $\SS$ denotes the complex spinor representation, i.e., an irreducible complex $\Cl(U)$-module, and Clifford multiplication on $\SS$ is denoted by ``$\circ$''.
\item the cover with fiber $\mathbb Z_4=\left\{\pm 1,\pm i\right\}$ of the orthogonal group $\mathrm{O}(U)$ is the Pin group $\mathrm{Pin}(U)$, with covering map given by the twisted adjoint action $\widetilde{\operatorname{Ad}}:\mathrm{Pin}(U)\to \mathrm{O}(U)$.
\item if $\dim U$ is even, then $\SS$ is $\fso(U)$-reducible and we denote by $\SS^+$ and $\SS^-$ its irreducible $\fso(U)$-submodules (the semispinor representations).
\item we identify $\Lambda^\bullet U$ with $\Cl(U)$ (as vector spaces) via the isomorphism
$$
u_1\wedge\cdots\wedge u_k\mapsto\frac{1}{k!}\sum_{\pi\in\mathfrak{S}_k}\operatorname{sgn}(\pi)u_{\pi(1)}\cdots u_{\pi(k)}
$$
and let any $\alpha\in\Lambda^\bullet U$ act on $\SS$ via Clifford multiplication, i.e. $\alpha\circ s$ for all $s\in \SS$.
\item we identify $\fso(U)$ with $\Lambda^2 U$ via
$
(u\wedge v)(w)=\eta(u,w)v-\eta(v,w)u
$, where $u,v,w\in U$. We recall that the spin representation of $\fso(U)$ on $\SS$ is half its action as a $2$-form: $$(u\wedge v)\cdot s=\frac{1}{2}(u\wedge v)\circ s$$ for all $s\in \SS$.
\end{itemize}
We will sometimes need the usual concrete realization of the representation of
the Clifford algebra on the spinor module in terms of Kronecker products of matrices. 
Let us consider the $2\times 2$ matrices
$$
E=\begin{pmatrix}
1 & 0 \\
0 & 1
\end{pmatrix}\;,\quad 
T=\begin{pmatrix}
0 & -i \\
i & 0
\end{pmatrix}\;,\quad
g_1=\begin{pmatrix}
i & 0 \\
0 & -i
\end{pmatrix}\;,\quad
g_2=\begin{pmatrix}
0 & i \\
i & 0
\end{pmatrix}\;,
$$
and set
$$
\alpha(j)=
\begin{cases}
1\;\text{if}\;j\;\text{is odd};\\
2\;\text{if}\;j\;\text{is even}.
\end{cases}
$$
If $\dim U=2k$ is even then a Clifford representation is given by the action of
$$\Cl(U)\simeq\underbrace{M_2(\CC)\otimes\cdots\otimes M_2(\CC)}_{k-times}=\End(\SS)
$$
on $\SS=\underbrace{\CC^2\otimes\cdots\otimes\CC^2}_{k-times}$ with the elements of a fixed orthonormal basis $(e_j)_{j=1}^{2k}$ of $U$ realized as
\begin{equation}
\label{eq:exreal}
e_j\mapsto E\otimes\cdots \otimes E\otimes g_{\alpha(j)}\otimes\underbrace{T\otimes\cdots\otimes T}_{[\frac{j-1}{2}]-times}\;,
\end{equation}
for all $j=1,\ldots,2k$. Note that the basis elements satisfy $e_j^2=-1$ as $\eta(e_i,e_j)=\delta_{ij}$ for all $i,j$. It follows that the volume $\vol=e_1\cdots e_{2k}$ squares to $(-1)^k$ so that $\SS^{\pm}$ are the $\pm 1$-eigenspaces of the involution $i^k\vol$. In this paper, we will not need a concrete realization for $\dim U$ odd.
\begin{definition}
\cite{MR1462223}
A nondegenerate bilinear form $\beta: \SS \otimes\SS\to\CC$ is called {\it admissible} if
there exist $\tau,\upsigma\in\left\{\pm 1\right\}$ such that $\beta(u\circ s, t) = \tau \beta(s, u \circ t)$ and $\beta(s,t)=\upsigma \beta(t,s)$ for all $u\in U$, $s,t\in\SS$.
\end{definition}
Admissible forms are automatically $\fso(U)$-equivariant. If $\SS = \SS^+\oplus\SS^-$, we also require $\SS^+$ and $\SS^-$ to be either isotropic or mutually orthogonal w.r.t. $\beta$ (in the
former case we set $\imath = -1$, in the latter $\imath = 1$). The numbers $(\tau, \upsigma)$,
or $(\tau, \upsigma,\imath)$, are the {\it invariants associated with $\beta$}.  

Set $\fm_{-2}=U$ , $\fm_{-1}=\SS$, $\fm=\fm_{-2}\oplus\fm_{-1}$. If $\tau\upsigma=-1$ we define on $\fm$ a structure of graded Lie algebra with Lie bracket given by the so-called ``Dirac current''
$$
\eta([s, t], u) = \beta(u \circ s, t)\;,
$$
for $s, t \in \SS$, $u\in U$. 
\begin{definition}
Any graded Lie algebra as defined above
is called an {\it extended translation algebra}. 
\end{definition}
The main result of \cite{MR3255456, MR3218266} is the classification of maximal transitive prolongations $$\fg^\infty=\bigoplus_{p\in\ZZ}\fg_p$$ of extended translation algebras. The following result is also important.
\begin{proposition}\cite[Theorem 2.3]{MR3218266}
\label{thm:azione2}
If $\dim U\geq 3$ then for all $p>0$ and $X\in\fg_p$, if $[X,\fg_{-2}]=0$ then $X = 0$. In other words, elements of $\fg_p$, $p>0$, are uniquely determined by their action on $\fg_{-2}$.
\end{proposition}
\subsection{The case $\fg=F_4$}
\label{sec:ePF4}
\hfill\vskip0.1cm\par
There are $3$ inequivalent KTS with Tits-Kantor-Koecher pair $(\fg=F_4,\sigma)$, cf. Corollary \ref{cor:enumerate}. In particular the grading of $F_4$ with marked Dynkin diagram 
$$
\begin{tikzpicture}
\node[root]   (1)   [label=${}$]                  {};
\node[root] (2) [right=of 1] [label=${}$] {} edge [-] (1);
\node[root]   (3) [right=of 2] [label=${}$] {} edge [rdoublearrow] (2);
\node[xroot]   (4) [right=of 3] [label=${}$] {} edge [-] (3);
\end{tikzpicture}
$$
has two grade-reversing involutions, corresponding to real forms $FI$ and  $FII$, cf. Table \ref{tab:satakeExceptional}. 

In this case
$(U,\eta)$ is $7$-dimensional and $\SS$ is the $8$-dimensional spinor module. 
(There are two such modules up to isomorphism, and they are equivalent as $\fso(V)$-representations. We have chosen the module for which the action of the volume
element $\vol\in\Cl(U)$ is $\vol\circ s=s$ for all $s\in S$.) 

More explicitly the graded components of the $5$-grading of $\fg=F_4$ are $\fg_0\simeq \mathfrak{so}(U)\oplus\CC E$ and $\fg_{\pm 2}\simeq U$, $\fg_{\pm 1}\simeq \SS$ as $\mathfrak{so}(U)$-modules.
The negatively graded part $\fm=\fg_{-2}\oplus\fg_{-1}$ is an extended translation algebra w.r.t. a bilinear form $\beta:\SS\otimes\SS\to\CC$ with the invariants $(\tau,\upsigma)=(-1,1)$.

For any $s\in\SS$ we introduce a linear map $\widehat s:\fm\to\fg_{-1}\oplus\fg_0$ of degree $1$ by
\begin{equation}
\label{eq:1F4}
\begin{split}
[\widehat s,t]&=-\Gamma^{(2)}(s,t)+\frac{1}{2}\beta(s,t)E\;,\\
[\widehat s, u]&=u\circ s\;,
\end{split}
\end{equation}
where $t\in\SS$, $u\in U$ and 
\begin{equation*}
\begin{aligned}
\Gamma^{(2)}:\SS\otimes\SS&\to\fso(U)\;,\\
\eta(\Gamma^{(2)}(s,t)u,v)&=\beta(u\wedge v\circ s,t)\;.
\end{aligned}
\end{equation*}
Using Proposition \ref{thm:azione2} (see also \cite[Lemma 2.5]{MR3218266}) one can directly check that  $\fg_{1}=\left\langle \widehat s\,|\,s\in \SS\right\rangle$. Similarly we define a map $\widehat u:\fm\to\fg_0\oplus\fg_1$ 
of degree $2$ for any $u\in U$ by
\begin{equation}
\label{eq:2F4}
\begin{split}
[\widehat u,s]&=-\widehat{u\circ s}\;,\\
[\widehat u, v]&=\lambda(u\wedge v)+\mu\eta(u,v)E\;,
\end{split}
\end{equation}
where $s\in\SS$, $v\in U$ and where $\lambda,\mu\in\CC$ are constants to be determined. 
\begin{lemma}
If $\lambda=2$ and $\mu=-1$ then $\widehat u\in\fg_2$ for all $u\in U$.
\end{lemma}
\begin{proof}
By $\fso(U)$-equivariance and transitivity, the action of $\fg_2$ on $\fm$ is necessarily of the above form for some $\lambda,\mu$. We then compute
 \begin{equation*}
\begin{aligned}
0&=[\widehat u,[s,v]]=[[\widehat u, s],v]+[s,[\widehat u,v]]\\
&=-[\widehat{u\circ s},v]-\lambda[u\wedge v,s]+\mu\eta(u,v)s\\
&=(-1+\frac{\lambda}{4}-\frac{\mu}{2})v\circ u\circ s-(\frac{\lambda}{4}+\frac{\mu}{2})u\circ v\circ s\;,
\end{aligned}
 \end{equation*}
for $u,v\in U$, $s\in \SS$. The claim follows from the fact that the two terms vanish separately.
\end{proof}
The Lie brackets between  elements of non-negative degrees are directly computed using transitivity and \eqref{eq:1F4}-\eqref{eq:2F4}. They are given by the natural structure of Lie algebra of $\fg_0$ and
\begin{equation}
\label{eq:F4missing}
\begin{array}{lll}
[A,\widehat s]=\widehat{As}\;,& [A, \widehat u]=\widehat{Au}\;,&\\[0cm]
[E, \widehat s]=\widehat s\;, & [E, \widehat u]=2\widehat u\;,&[\widehat s, \widehat t]=\widehat{[s,t]}\;,\\[0cm]
\end{array}
\end{equation}
where $A\in \mathfrak{so}(U)$, $\widehat s,\widehat t\in\fg_1$ and $\widehat u\in\fg_{2}$. 
We give details only for the third and last bracket. First compute
\begin{align*}
[[A,\widehat u],v]&=[A,[\widehat u,v]]-[\widehat u,Av]\\
&=2[A, u\wedge v]-[\widehat u,Av]\\
&=2Au\wedge v-\eta(Au,v)E\\
&=[\widehat{Au},v]
\end{align*}
and
\begin{align*}
[[\widehat s,\widehat t],u]&=[\widehat s,u\circ t]-[\widehat t,u\circ s]\\
&\equiv \frac{1}{2}\beta(s,u\circ t)E-\frac{1}{2}\beta(t,u\circ s)E\mod\fso(U)\\
&\equiv\beta(s,u\circ t)E=-\eta([s,t],u)E\mod\fso(U)\\
&\equiv [\widehat{[s,t]},u]\mod\fso(U)\;,
\end{align*}
for all $A\in\fso(U)$, $s,t\in \SS$, $u,v\in U$. The brackets follow from these identities, Proposition \ref{thm:azione2} and the fact that the action of $\widehat u$ on $\fg_{-2}$ is fully determined by its component on $E$, cf. \eqref{eq:2F4}.

The algebras of derivations of the two KTS on $\SS$ are the complexifications of the maximal compact subalgebras of $\fso(7,\RR)$  and $\fso(3,4)$, respectively (see Table \ref{tab:satakeExceptional} and
Theorem \ref{thm:symmetryalg}):
\begin{align*}
\der(\SS)&\simeq\fso(7,\CC)\;\;\text{for}\;\;FII\;, \\
\der(\SS)&\simeq \fso(3,\CC)\oplus\fso(4,\CC)\;\;\text{for}\;\;FI\;.
\end{align*}
The following main result describes both triple systems in a uniform fashion.
\begin{theorem}
\label{thm:F4KTSsecond}
Let $U$ be a $7$-dimensional complex vector space with a non-degenerate symmetric bilinear form $\eta$ and $\SS$ the associated $8$-dimensional spinor $\Cl(U)$-module. Let $\beta:\SS\otimes\SS\to\CC$ be the unique (up to constant) admissible bilinear form on $\SS$ and
$\Gamma^{(2)}:\SS\otimes\SS\to \fso(U)$ the operator given by
$$\eta(\Gamma^{(2)}(s,t)u,v)=\beta(u\wedge v\circ s,t)\;,$$
where $s,t\in \SS$ and $u,v\in U$. Fix an orthogonal decomposition $$U=W\oplus W^\perp$$ of $U$ with $\dim W=7$ (i.e. with $W^\perp=0$, $W=U$) or $\dim W=3$ and let $I=\vol_W\in\Cl(U)$ be the volume of $W$. ($I$ acts on $\SS$ as the identity  if $\dim W=7$ and as a paracomplex structure if $\dim W=3$). 

Then $\SS$ with the triple product
\begin{equation}
\label{eq:F4product}
\begin{split}
(rst)&=-\Gamma^{(2)}(r,I\circ s)\cdot t+\frac{1}{2}\beta(r,I\circ s)t\;,\qquad r,s,t\in \SS\;,
\end{split}
\end{equation}
is a $K$-simple Kantor triple system with Tits-Kantor-Koecher Lie algebra $\fg=F_4$ and derivation algebra $$\der(\SS)=\stab_{\fso(U)}(I)=\begin{cases}\fso(U)\;\text{if}\;\dim W=7\;,\\
\fso(W)\oplus\fso(W^\perp)\;\text{if}\;\dim W=3\;.
\end{cases}$$
\end{theorem}
\begin{proof}
It is immediate to see that $I\in\mathrm{Pin}(U)$ covers the opposite of the orthogonal reflection $r_W:U\to U$ across $W$, that is $\widetilde{\operatorname{Ad}}_I=-r_W$, and that  $\beta(I\circ s,t)=\beta(s, I\circ t)$ for all $s,t\in \SS$. 

Let $\sigma:\fg\to\fg$ be the grade-reversing map defined by
\begin{equation}
\label{eq:F4invextended}
\begin{array}{ll}
\sigma(u)=\widehat{r_W u}\;,& \sigma(s)=\widehat{I\circ s}\;,\\[0cm]
\sigma(A)=r_W A r_W\;,\;\;&\sigma(E)=-E\;,\\[0cm]
\sigma(\widehat s)=I\circ s\;,& \sigma(\widehat u)=r_W u\;,
\end{array}
\end{equation}
where $u\in U$, $s\in\SS$ and $A\in\fso(U)$. Clearly $\sigma^2=1$ and we now show that $\sigma$ is a Lie algebra morphism. First note that $\sigma[A,s]=\widehat{I\circ As}$ and
$[\sigma(A),\sigma(s)]=\widehat t$,
\begin{align*}
t&=[r_W A r_W,I\circ s]\\
&=I\circ A(I\circ I\circ s)\\
&=I\circ As\;,
\end{align*}
hence $\sigma[A,s]=[\sigma(A),\sigma(s)]$ for all $A\in\fso(U)$, $s\in\SS$. Identity $\sigma[A,\widehat{s}]=[\sigma(A),\sigma(\widehat{s})]$ is analogous. Similarly
$\sigma[\widehat{u},s]=-\sigma(\widehat{u\circ s})=-I\circ u\circ s$ while
\begin{align*}
[\sigma(\widehat u),\sigma(s)]&=-[\widehat{I\circ s},r_W u]\\
&=-(r_Wu)\circ I\circ s=\widetilde{\operatorname{Ad}}_I(u)\circ I\circ s\\
&=-I\circ u\circ s\;,
\end{align*}
proving $\sigma[\widehat{u},s]=[\sigma(\widehat u),\sigma(s)]$ and, in the same way, 
$\sigma[u,\widehat{s}]=[\sigma( u),\sigma(\widehat s)]$ for all $u\in U$, $s\in\SS$.
The remaining identities are straightforward, except for $\sigma[\widehat{s},t]=[\sigma(\widehat s),\sigma(t)]$, $s,t\in\SS$, which we now show. We have 
\begin{align*}
\eta(\sigma[\widehat s,t]u,v)&=-\eta(\sigma(\Gamma^{(2)}(s,t))u,v)+\beta(s,t)\eta(u,v)\\
&=-\eta(\Gamma^{(2)}(s,t)r_W u,r_W v)+\beta(s,t)\eta(u,v)\\
&=-\beta((r_W u)\wedge (r_W v)\circ s,t)+\beta(s,t)\eta(u,v)
\end{align*}
and
\begin{align*}
\eta([\sigma(\widehat s),\sigma(t)]u,v)&=-\eta([\widehat{I\circ t},I\circ s]u,v)\\
&=+\eta(\Gamma^{(2)}(I\circ t,I\circ s)u,v)+\beta(I\circ t,I\circ s)\eta(u,v)\\
&=\beta(I\circ u\wedge v\circ I\circ t,s)+\beta(s,t)\eta(u,v)\\
&=\beta((r_W u)\wedge (r_W v)\circ t,s)+\beta(s,t)\eta(u,v)\\
&=\eta(\sigma[\widehat s,t]u,v)
\end{align*}
for all $u,v\in U$ so that $\sigma[\widehat{s},t]=[\sigma(\widehat s),\sigma(t)]$. 

The triple product \eqref{eq:F4product} is the usual formula $(rst)=[[r,\sigma(s)],t]$ and the rest is clear.
\end{proof}

\subsection{The case $\fg=E_6$}
\label{sec:E6with2center}
\hfill\vskip0.1cm\par
The grading of $\fg=E_6$ associated to
$$
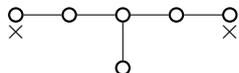
\begin{tikzpicture}
\node[xroot]   (1) [label=\;\;\;${{}^{^{\phantom{T^T}}}}$]             {};
\node[root]   (3) [right=of 1] {} edge [-] (1);
\node[root]   (4) [right=of 3] {} edge [-] (3);
\node[root]   (5) [right=of 4] {} edge [-] (4);
\node[xroot]  (6) [right=of 5]{} edge [-] (5);
\node[root]   (2) [below=of 4] {} edge [-] (4);
\end{tikzpicture}
$$
is of extended Poincar\'e type and all the 4 real forms of $E_6$ are compatible with this grading, leading to 4 non-isomorphic KTS with algebra of derivations equal to
\begin{equation}
\label{eq:derE6}
\mathfrak{der}(\SS)\cong\left\{ \begin{array}{lll}
\fso(4,\CC)\oplus\fso(4,\CC)&\text{for}& EI\,,\\
\fso(5,\CC)\oplus\fso(3,\CC)&\text{for}& EII\,,\\
\fso(7,\CC)&\text{for}& EIII\,,\\
\fso(8,\CC)&\text{for}& EIV\,,
\end{array}\right.
\end{equation}
possibly up to a $1$-dimensional center. We note that this is the unique $5$-grading of a simple exceptional Lie algebra whose associated parabolic subalgebra is not maximal. The fact that the center of $ \fg_0 $ is $2$-dimensional makes our previous (and usual) approach to describe KTS practically inconvenient. In this subsection, we will use another approach and exploit the root space decomposition of $E_6$ to describe the KTS associated to $EIV$. The remaining products will be reduced to this case by a direct argument.

Let $U$ be an $8$-dimensional complex vector space, $ \eta $ a non-degenerate symmetric bilinear form on $U$ and $ \SS=\SS^+\oplus\SS^- $ the spinor module, where each semispinor representation $\SS^\pm$ is $8$-dimensional. Let $E$ be the grading element and $F$ the operator acting as the volume $\vol\in\Cl(U)$ on $\SS$ and trivially on $U$. It clearly commutes with $E$ and $\fso(U)$. 
The grading of $\fg=E_6$ is then given by
$$\begin{array}{lll}
\fg_{\pm 1}=\SS\;,&\fg_{0}=\fso(U)\oplus\CC E\oplus \CC F\;,
&\fg_{\pm 2}=U\;.
\end{array}
$$
We shall fix an isotropic decomposition of $ U =W\oplus W^*$ and consider the natural $\fso(U)$-equivariant isomorphisms $\SS^+\cong \Lambda^{even}W^* $ and $ \SS^-\cong \Lambda^{odd}W^*$. With these conventions, the action of the Clifford algebra on $\SS$ reads as
\begin{equation}
\label{eq:cliffE6}
\begin{split}
w\circ \alpha&=-2\imath_w\alpha\;,\\
w^*\circ \alpha&=w^*\wedge \alpha\;,
\end{split}
\end{equation}
for all $w\in W$, $w^*\in W^*$ and $\alpha\in\SS$. There is an admissible bilinear form
on $\SS$ with the invariants $(\tau,\sigma,\imath)=(+,+,+)$; to avoid confusion with the elements of $\SS$ we simply denote it by 
$$\alpha\bullet \beta=(-1)^{[\half(\deg(\alpha)+1)]}i_\omega(\alpha\wedge \beta)\;,$$ 
where $\alpha,\beta\in\SS$, $[\cdot]$ is the "ceiling" of a rational number, i.e., its upper integer part, $\deg(\alpha) $ the degree of $\alpha$ as a differential form and $ \omega\in \Lambda^4W$ a fixed volume.
\vskip0.1cm\par
We depart with the KTS of type EIV. Throughout the section, we denote the unit constant in $\SS$ by $\1$  and consider a basis $(\tfrac{\partial}{\partial x^1},\ldots, \tfrac{\partial}{\partial x^4})$ of $W$, with associated dual basis $(dx^1,\ldots,dx^4)$ of $W^*$ and induced Hodge star operator $\star:\SS\to \SS$.
\begin{theorem}\label{thm:ktsEIVPoincare}
Let $U$ be an $8$-dimensional complex vector space with a non-degenerate symmetric bilinear form $ \eta $ and fix an isotropic decomposition $U=W\oplus W^*$. Then the associated $ 16 $-dimensional spinor representation $ \SS\cong \Lambda^\bullet W^*$ is a $K$-simple Kantor triple system with the triple product given by
\begin{equation}
\begin{aligned}
\label{eq:prodE6zero}
(\SS^\pm,\SS^\mp,\SS)&=0\;,\\
(\1,dx^{1234},dx^{i})&=(dx^{1234},\1,dx^{jkl})=(dx^{jkl},dx^i,dx^{1234})=(dx^i,dx^{jkl},\1)=0\;,\\
\end{aligned}
\end{equation}
and, for all other homogeneous differential forms, by
\begin{equation}
\label{eq:prodE6Poincare}
(\alpha,\beta,\gamma)=\left\{\begin{array}{lll}
(\alpha\bullet \beta)\gamma+(\beta\bullet \gamma)\alpha-(\alpha\bullet \gamma)\beta&\text{ if }& \alpha,\beta,\gamma\in \SS^\pm\\[.5cm]
(-1)^{[\half(deg(\alpha)+2)]}\star(\star(\alpha\wedge \gamma)\wedge \star\beta)& \text{ if }& \alpha,\beta\in\SS^\pm,\,\gamma\in \SS^\mp,\\&& deg(\alpha)+deg(\gamma)\leq 4,\\
(-1)^{[\half(deg(\alpha)+1)]}\star(\star\alpha\wedge \star\gamma)\wedge \beta &\text{ if }&	 \alpha,\beta\in\SS^\pm,\,\gamma\in \SS^\mp,\\&& deg(\alpha)+deg(\gamma)> 4.
\end{array}\right.
\end{equation}
The product is $\fso(U)$-equivariant and its associated Tits-Kantor-Koecher Lie algebra is $E_6$.
\end{theorem}
\begin{proof}
Identities (i) and (ii) of Definition \ref{def:KTS} can be checked by tedious but straightforward computations, using the root space decomposition of $E_6$ and $\fso(U)$-equivariance of \eqref{eq:prodE6zero}-\eqref{eq:prodE6Poincare}. We here simply record that $\fso(U)$ is generated by the two abelian parabolic subalgebras which exchange $W$ and $W^*$ so that $\fso(U)$-equivariance follows from the equivariance under the action of these subalgebras. The latter can be directly checked using \eqref{eq:cliffE6}.
\end{proof}
To proceed further, we first note that
any KTS of type $EI$, $EII$, $EIII$ is a {\it modification}, in Asano's sense \cite{Asano1991}, of the KTS of type $EIV$. In other words, we consider new triple products of the form
\begin{equation}
\label{eq:modification}
(\alpha,\beta,\gamma)_{\Phi}=(\alpha,\Phi(\beta),\gamma)\;,
\end{equation}
where the product on the r.h.s. is the EIV product described in Theorem \ref{thm:ktsEIVPoincare} and $\Phi:\SS\to\SS$ an involutive automorphism of it. It is known that a modification of a KTS is still a KTS.

Note that $\Phi$ has to be $\der(\SS)$-equivariant, where $\der(\SS)$ is detailed in \eqref{eq:derE6} for all cases. Furthermore, since any endomorphism of $\SS$ is realized by the action of some element in the Clifford algebra, we are led to consider $\der(\SS)$-equivariant elements in $\Cl(U)$, which are also involutive automorphisms of the triple product of type EIV. 

We introduce the required maps
\begin{equation}
\label{eq:involE6Poincare}
\Phi(\alpha)=\left\{ \begin{array}{lll}
\widetilde\star\alpha&\text{for}& EI\,,\\
&\\
i(dx^1\wedge(\widetilde\star\alpha)-\imath_{\tfrac{\partial}{\partial x^1}}(\widetilde \star\alpha))&\text{for}& EII\,,\\
&\\
i(dx^1\wedge\alpha-\imath_{\tfrac{\partial}{\partial x^1}}\alpha)&\text{for}& EIII\,,
\end{array}\right.
\end{equation}
where $\widetilde\star$ is the modified Hodge star operator given by $\widetilde\star\alpha=\epsilon(\alpha)\star\alpha$ with $\epsilon(\alpha)=-\alpha$ when $\deg(\alpha)=1,2$ and $\epsilon(\alpha)=\alpha$ when $\deg(\alpha)=0,3,4$. It is easy to see that any $\Phi$ is the volume, up to some constant, of a non-degenerate subspace of $U$ of appropriate dimension and hence realized as the action of an involutive and $\der(\SS)$-equivariant element of the Pin group. As an aside, we note that conjugation by $\Phi$ in case EIII is nothing but the outer automorphism of $\fso(U)$ associated to the symmetry of the Dynkin diagram exchanging $\SS^+$ and $\SS^-$.
\begin{theorem}\label{thm:ktsE6Poincare}
Let $U$ be an $8$-dimensional complex vector space with a non-degenerate symmetric bilinear form $ \eta $ and fix an isotropic decomposition $U=W\oplus W^*$. Then the associated $ 16 $-dimensional spinor representation $\SS\cong\Lambda^\bullet W^*$ is a K-simple Kantor triple system with the modification
\begin{equation}
(\alpha,\beta,\gamma)_{\Phi}=(\alpha,\Phi(\beta),\gamma)\;,\qquad\alpha,\beta,\gamma\in\SS\;,
\end{equation}
of the triple product described in Theorem \ref{thm:ktsEIVPoincare} by any of the three endomorphisms \eqref{eq:involE6Poincare} of $\SS$. Each product is $\der(\SS)$-equivariant, see \eqref{eq:derE6}, and its associated Tits-Kantor-Koecher Lie algebra is $E_6$.
\end{theorem}
\begin{proof}
The KTS described in Theorem \ref{thm:ktsEIVPoincare} is $\fso(U)$-equivariant, hence equivariant under the action of the connected component of the identity of $\mathrm{Pin}(U)$. It immediately follows that the map $\Phi$ in case EI is an automorphism of the KTS. A similar statement is directly checked for EIII and, consequently, for EII too. Hence, general properties of modifications apply. \end{proof}
\subsection{The case $\fg=E_7$}
\hfill\vskip0.1cm\par
The grading
$$
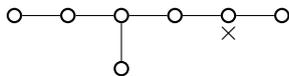
\begin{tikzpicture}
\node[root]   (1) [label=\;\;\;${{}^{^{\phantom{T^T}}}}$]             {};
\node[root]   (3) [right=of 1] {} edge [-] (1);
\node[root]   (4) [right=of 3] {} edge [-] (3);
\node[root]   (5) [right=of 4] {} edge [-] (4);
\node[xroot]  (6) [right=of 5]{} edge [-] (5);
\node[root]   (7) [right=of 6]{} edge [-] (6);
\node[root]   (2) [below=of 4] {} edge [-] (4);
\end{tikzpicture}
$$
of $\fg=E_7$ is of extended Poincar\'e type. Let $(U,\eta)$ be a $10$-dimensional complex vector space with a non-degenerate symmetric bilinear form and $\SS=\SS^+\oplus \SS^-$ the decomposition into semispinors of the $32$-dimensional spinor $\fso(U)$-module $\SS$. 
We have $\fg_0\simeq \fso(U)\oplus\fsl(2,\CC)\oplus\CC E$, where $E$ is the grading element, and
\begin{equation*}
\begin{aligned}
\fg_{-1}&=\;\SS^+\boxtimes \CC^2,\qquad\fg_{1}=\SS^-\boxtimes\CC^2\;,\\
\fg_{-2}&=U\;,\qquad\;\;\;\;\;\;\;\;\;\,\,\fg_{2}=U\;,
\end{aligned}
\end{equation*}
with their natural structure of $\fg_0$-modules. There are $3$ compatible real forms $EV$, $EVI$, $EVII$, and the derivation algebras of the corresponding triple systems are respectively given by $\fso(5,\CC)\oplus\fso(5,\CC)\oplus\fso(2,\CC)$, $\fso(3,\CC)\oplus\fso(7,\CC)\oplus\fsl(2,\CC)$ and $\fso(9,\CC)\oplus\fso(2,\CC)$. 

We recall that in our convention \eqref{eq:exreal} the Clifford algebra $\Cl(U)$
acts on $$\SS=\underbrace{\CC^2\otimes\cdots\otimes\CC^2}_{k-times}\;,\qquad k=5\;,$$
and that $\SS^\pm$ are the $\pm 1$-eigenspaces of the involution $i\vol=T\otimes\cdots\otimes T$.
There is a unique (up to constant) admissible bilinear form $\beta:\SS\otimes\SS\to\CC$ on $\SS$ with invariants $(\tau,\sigma,\imath)=(-1,-1,-1)$. It is given by
\begin{equation}
\label{eq:betadmII}
\beta(x_1\otimes \cdots\otimes x_5,y_1\otimes\cdots\otimes y_5)=\omega(x_1,y_1)<x_2,y_2>\cdots\omega(x_5,y_5)\;,
\end{equation}
where $<-,->$ (resp. $\omega$) is the standard $\CC$-linear product (resp. symplectic form) on $\CC^2$ and $x_i, y_i\in\CC^2$, $i=1,\ldots,5$. We note that semispinors $\SS^\pm$ are isotropic and $(\SS^\pm)^*\simeq \SS^\mp$.

We consider $\fso(U)$-equivariant operators
\begin{equation}
\label{eq:simba1}
\begin{aligned}
\Gamma:\SS\otimes\SS &\to U\;,\\
\eta(\Gamma(s,t),u)&=\beta(u\circ s,t)\;,
\end{aligned}
\end{equation} 
and
\begin{equation}
\label{eq:simba2}
\begin{aligned}
\Gamma^{(2)}:\SS\otimes\SS &\to\fso(U)\;,\\
\eta(\Gamma^{(2)}(s,t)u,v)&=\beta(u\wedge v\circ s,t)\;,
\end{aligned}
\end{equation} 
where $s,t\in\SS$, $u,v\in\ U$. They are symmetric and satisfy $\Gamma(\SS^\pm,\SS^\mp)=\Gamma^{(2)}(\SS^\pm,\SS^\pm)=0$.
\vskip0.1cm\par
The structure of graded Lie algebra on $\fm=\fg_{-2}\oplus\fg_{-1}$ is a mild variation of the usual Dirac current:
$$
[s\otimes c,t\otimes d]=\Gamma(s,t)\omega(c,d)\;,
$$
where $s,t\in\SS^+$, $c,d\in\CC^2$. The Lie algebra $\fg_0$ acts naturally on $\fm$ by $0$-degree derivations and we now in turn describe $\fg_1$ and $\fg_2$. 
The general form of the adjoint action on $\fm$ of positive degree elements of $\fg$ is constrained due to $\fg_0$-equivariance: for any $r\in\SS^-$, $c\in\CC^2$ and $u\in U$, we have that $r\otimes c\in\fg_{1}$ and $\widehat{u}\in\fg_{2}$ act as
\begin{equation}
\label{eq:1E7pg}
\begin{split}
[r\otimes c,t\otimes d]&=c_1\underbrace{\Gamma^{(2)}(r,t)\omega(c,d)}_{\text{element of}\;\fso(U)}\;\;+\;\; c_2\underbrace{\beta(r,t)\omega(c,d)E}_{\text{element of}\;\CC E}\;\;+\;\; c_3\!\!\!\underbrace{\beta(r,t)c\odot d}_{\text{element of }\;\fsl(2,\CC)}\;,\\
[r\otimes c,v]&=v\circ r\otimes c\;,
\end{split}
\end{equation}
and respectively
\begin{equation}
\label{eq:2E7pg}
\begin{split}
[\widehat u,t\otimes d]&=-u\circ t\otimes d\;,\\
[\widehat u,v]&=c_4\!\!\!\!\!\!\underbrace{u\wedge v}_{\text{element of}\;\fso(U)}\!\!\!\!+\;\;\,c_5\!\!\underbrace{\eta(u,v)E}_{\text{element of }\;\CC E}\;,
\end{split}
\end{equation}
for all $t\in\SS^+$, $d\in\CC^2$ and $v\in U$. The values of the constants $c_1,\ldots,c_5$ are now determined.
\begin{proposition}
\label{prop:percitaredopo}
$c_1=-1$, $c_2=\tfrac{1}{2}$, $c_3=-1$ and $c_4=2$, $c_5=-1$.
\end{proposition}
\begin{proof}
We let $r\in \SS^-$, $t\in \SS^+$, $u\in U$, $c,d\in\CC^2$ as above and depart with
\begin{equation*}
\begin{aligned}
0&=[r\otimes c,[t\otimes d,u]]\\
&=[[r\otimes c, t\otimes d],u]+[t\otimes d,[r\otimes c,u]]\\
&=\omega(c,d)(c_1\Gamma^{(2)}(r,t)u-2c_2\beta(r,t)u-\Gamma(t,u\circ r))\;.
\end{aligned}
\end{equation*}
Abstracting $\omega(c,d)$ and taking the inner product with any $v\in U$  yields
\begin{equation*}
\begin{aligned}
0&=c_1\beta(u\wedge v\circ r,t)-2c_2\beta(r,t)\eta(u,v)+\beta(u\circ v\circ t,r)\\
&=(c_1+1)\beta(u\wedge v\circ r,t)+(-2c_2+1)\beta(r,t)\eta(u,v)\;,
\end{aligned}
\end{equation*}
where we used $u\circ v=u\wedge v-\eta(u,v)$. Since $\SS^+\simeq (\SS^-)^*$ and $\End(\SS^-)\simeq\Lambda^4 U\oplus\Lambda^2 U\oplus\Lambda^0 U$ acts faithfully on $\SS^-$, the two terms vanish separately and the values of $c_1$ and $c_2$ follow.

The proof of $c_3=-1$ relies on the explicit realization \eqref{eq:exreal} of the Clifford algebra.  Let $r\in\SS^-$, $t\in \SS^+$ and choose $c,d\in\CC^2$ satisfying $\omega(c,d)=1$, then
\begin{equation*}
\begin{aligned}
{}[r\otimes d,[t\otimes c,t\otimes d]]&=
\Gamma(t,t)\circ r\otimes d
\end{aligned}
\end{equation*}
and
\begin{equation*}
\begin{aligned}
{}[r\otimes d,[t\otimes c,t\otimes d]]&=[[r\otimes d,t\otimes c],t\otimes d]+
[t\otimes c,[r\otimes d,t\otimes d]]\\
&=-c_1\Gamma^{(2)}(r,t)\cdot t\otimes d+c_2\beta(r,t)t\otimes d+3c_3\beta(r,t)t
\otimes d\;.
\end{aligned}
\end{equation*}
Abstracting $d$, substituting the values of $c_1$ and $c_2$ already determined and rearranging terms, we are left with the Fierz-like identity
\begin{equation}
\label{eq:FierzE7}
\Gamma(t,t)\circ r=\Gamma^{(2)}(r,t)\cdot t+\beta(r,t)(\tfrac{1}{2}+3c_3)t\;.
\end{equation}
We now choose suitable spinors
$$
t=\begin{pmatrix}
1 \\ +i
\end{pmatrix}^{\!\!\otimes 5}\;,\qquad r=\begin{pmatrix}
1 \\ -i
\end{pmatrix}^{\!\!\otimes 5}
$$ 
and use \eqref{eq:exreal} and \eqref{eq:betadmII} to get $\beta(r,t)=32i$ and $\Gamma(t,t)=0$. A similar but longer computation says that
\begin{equation*}
\begin{aligned}
\Gamma^{(2)}(r,t)&=-\frac{1}{2}\sum_{l,m}\beta(r,e_l\wedge e_m\circ t)e_l\wedge e_m\\
&=-\sum_{l\;odd}\beta(r,e_l\wedge e_{l+1}\circ t)e_l\wedge e_{l+1}\\
&=i\sum_{l\;odd}\beta(r,t)e_l\wedge e_{l+1}\\
&=-32\sum_{l\;odd}e_l\wedge e_{l+1}\;,
\end{aligned}
\end{equation*}
from which $\Gamma^{(2)}(r,t)\cdot t=i80t$. In summary, identity \eqref{eq:FierzE7} turns into
$
0=i80t+i32(\tfrac{1}{2}+3c_3)t
$,
so that $c_3=-1$. The values of the constants $c_4$, $c_5$ relative to the adjoint action of $\fg_2$ on $\fm$ are determined by similar but easier computations, which we omit.
\end{proof}
The description of the Lie brackets of $\fg=E_7$ is completed, with the exception of the Lie bracket of two elements $r\otimes c, s\otimes d\in\fg_{1}=\SS^-\boxtimes\CC^2$, which is
$$
[r\otimes c,s\otimes d]=\omega(c,d)\widehat{\Gamma(r,s)}\;.
$$
To prove this identity, it is sufficient to note that elements of $\fg_2$ are fully determined by their adjoint action on $\fg_{-2}$ by \eqref{eq:2E7pg} and check that the l.h.s. and r.h.s. yield the same result there.

Theorem \ref{thm:E7KTSep} gives $3$ different KTS associated with the graded Lie algebra just described.
\begin{theorem}
\label{thm:E7KTSep}
Let $U$ be a $10$-dimensional complex vector space with a non-degenerate symmetric bilinear form $\eta$ and $\SS=\SS^+\oplus\SS^-$ the associated $32$-dimensional spinor module. Let $\beta:\SS\otimes\SS\to\CC$ be the admissible bilinear form on $\SS$ with invariants $(\tau,\sigma,\imath)=(-1,-1,-1)$ and
$$\Gamma:\SS\otimes\SS\to U\;,\qquad\Gamma^{(2)}:\SS\otimes\SS\to \fso(U)\;,$$ 
the $\fso(U)$-equivariant operators defined in \eqref{eq:simba1}-\eqref{eq:simba2}. Fix an orthogonal decomposition $$U=W\oplus W^\perp$$ of $U$ with $\dim W=1$, $\dim W=3$ or $\dim W=5$ and let $I=\vol_W\in\Cl(U)$ be the volume of $W$. ($I$ acts on $\SS$ as a complex structure if $\dim W=1,5$ and as a paracomplex structure when $\dim W=3$.) Let $J:\CC^2\to \CC^2$ be the standard complex structure on $\CC^2$, except for $\dim W=3$ where we set $J:=\operatorname{Id}$. 

Then $\SS^+\otimes \CC^2$ with the triple product
\begin{equation}
\label{eq:E7product EP}
\begin{split}
(xyz)&=(-\Gamma^{(2)}(r,I\circ s)\cdot t+\tfrac{1}{2}\beta(r,I\circ s)t)\otimes \omega(b,Jc)d\\
&\;\;\;\;-\beta(r,I\circ s)t\otimes(\omega(b,d)Jc+\omega(Jc,d)b)\;,
\end{split}
\end{equation}
for all elements $x=r\otimes b$, $y=s\otimes c$ and $z=t\otimes d$ of $\SS^+\otimes\CC^2$, is a 
$K$-simple Kantor triple system with Tits-Kantor-Koecher Lie algebra $\fg=E_7$ and derivation algebra 
$$\der(\SS^+\otimes\CC^2)=\stab_{\fg_0}(I,J)\simeq\begin{cases}
\fso(9,\CC)\oplus\fso(2,\CC)\;\text{if}\;\dim W=1\;,\\
\fso(3,\CC)\oplus\fso(7,\CC)\oplus\fsl(2,\CC)\;\text{if}\;\dim W=3\;,\\
\fso(5,\CC)\oplus\fso(5,\CC)\oplus\fso(2,\CC)\;\text{if}\;\dim W=5\;.
\end{cases}$$
\end{theorem}
\begin{proof}
It follows from the fact that the involution
\begin{equation}
\label{eq:E7ep}
\begin{array}{ll}
\sigma(\widehat u)=r_W u\;,&\sigma(r\otimes d)=I\circ r\otimes Jd\;,\\[0cm]
\sigma(A)=\left\{
\begin{array}{l}
r_W A r_W\;\text{if}\;A\in\fso(U)\\
JAJ^{-1}\;\text{if}\;A\in\fsl(2,\CC)
\end{array}
\right.,&\sigma(E)=-E\;,\\[0cm]
\sigma(s\otimes c)=I\circ s\otimes Jd\;,&\sigma(u)=\widehat{r_W u}\;,
\end{array}
\end{equation}
where $u\in U$, $r\otimes d\in\SS^-\boxtimes\CC^2$, $s\otimes c\in\SS^+\boxtimes\CC^2$,
is a Lie algebra morphism. The proof is analogous to that of Theorem \ref{thm:F4KTSsecond},
making use of the explicit expressions of the Lie brackets of $\fg=E_7$, the fact that $I\in\mathrm{Pin}(U)$ covers the opposite of the orthogonal reflection $r_W:U\to U$ across $W$ and the identity
$\beta(I\circ s,I\circ t)=\beta(s,t)$ 
for all $s,t\in \SS$.
\end{proof}
\subsection{The case $\fg=E_8$}
\hfill\vskip0.1cm\par
The grading 
$$
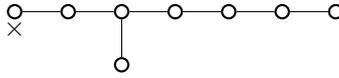
\begin{tikzpicture}
\node[xroot]  (1)  [label=\;\;\;${{}^{^{\phantom{T^T}}}}$]            {};
\node[root]   (3) [right=of 1] {} edge [-] (1);
\node[root]   (4) [right=of 3] {} edge [-] (3);
\node[root]   (5) [right=of 4] {} edge [-] (4);
\node[root]   (6) [right=of 5] {} edge [-] (5);
\node[root]   (7) [right=of 6] {} edge [-] (6);
\node[root]   (2) [below=of 4] {} edge [-] (4);
\node[root]   (8) [right=of 7] {} edge [-] (7);
\end{tikzpicture}
$$
of $\fg=E_8$ is of extended Poincar\'e type, with $\fg_0=\fso(U)\oplus\CC E$ and the other graded components given by $\fg_{\pm 1}=\SS^{\mp}$, $\fg_{\pm 2}=U$. Here $(U,\eta)$ is a $14$-dimensional complex vector space with a non-degenerate symmetric bilinear form and $\SS=\SS^+\oplus\SS^-$ the decomposition into semispinors of the corresponding $128$-dimensional spinor $\fso(U)$-module $\SS$.

There is an admissible bilinear form $\beta:\SS\otimes\SS\to\CC$ with invariants $(\tau,\sigma,\imath)=(-1,+1,-1)$ and usual operators
\begin{equation}
\label{eq:simbaI}
\begin{aligned}
\Gamma:\SS\otimes\SS &\to U\;,\\
\eta(\Gamma(s,t),u)&=\beta(u\circ s,t)\;,
\end{aligned}
\end{equation} 
and
\begin{equation}
\label{eq:simbaII}
\begin{aligned}
\Gamma^{(2)}:\SS\otimes\SS &\to\fso(U)\;,\\
\eta(\Gamma^{(2)}(s,t)u,v)&=\beta(u\wedge v\circ s,t)\;,
\end{aligned}
\end{equation} 
$s,t\in\SS$, $u,v\in\ U$. We note that $(\SS^\pm)^*\simeq\SS^{\mp}$ and that \eqref{eq:simbaI}-\eqref{eq:simbaII} are both $\fso(U)$-equivariant and skewsymmetric. They also
satisfy $\Gamma(\SS^\pm,\SS^\mp)=\Gamma^{(2)}(\SS^\pm,\SS^\pm)=0$.

The arguments used in \S\ref{sec:ePF4} for $F_4$  extend almost verbatim to $E_8$. In particular the non-trivial Lie brackets of $\fg=E_8$ are given by the natural action of $\fg_0$ on each graded component and
\[ \label{eq:bracketsE8spin}
\begin{array}{lll}
	[s_1,s_2]=\Gamma(s_1,s_2)\;,&[t_1,v]=v\circ t_1\;,& [t_1,s_1]=-\Gamma^{(2)}(t_1,s_1)+\frac{1}{2}\beta(t_1,s_1)E\;,\\[0cm]
	[\widehat u,v]=2 u\wedge v-\eta(u,v)E\;,&[\widehat u,s_1]=-u\circ s_1\;,&		[t_1,t_2]=\widehat{\Gamma(t_1,t_2)}\;,
\end{array}
 \]
for all $s_1,s_2\in\fg_{-1}\simeq\SS^+$, $t_1,t_2\in\fg_{+1}\simeq\SS^-$, $v\in\fg_{-2}\simeq U$ and $\widehat u\in\fg_{2}\simeq U$. The proof of the following result is similar to the proof of Theorem \ref{thm:F4KTSsecond} and therefore omitted.
\begin{theorem}
Let $U$ be a $14$-dimensional complex vector space with a non-degenerate symmetric bilinear form $\eta$ and $\SS=\SS^+\oplus \SS^-$ the associated $128$-dimensional spinor module. Let $\beta:\SS\otimes\SS\to\CC$ be the unique (up to constant) admissible bilinear form on $\SS$ with invariants 
$(\tau,\sigma,\imath)=(-1,+1,-1)$ and \eqref{eq:simbaI}-\eqref{eq:simbaII} the naturally associated operators. Fix an orthogonal decomposition $$U=W\oplus W^\perp$$ of $U$ with $\dim W=3$ or $\dim W=7$ and let $I=\vol_W\in\Cl(U)$ be the volume of $W$. ($I$ acts on $\SS$ as a paracomplex structure.) 
Then $\SS^+$ with the triple product
\begin{equation}
\label{eq:E8productspinor}
\begin{split}
(rst)&=-\Gamma^{(2)}(r,I\circ s)\cdot t+\frac{1}{2}\beta(r,I\circ s)t\;,\qquad r,s,t\in \SS^+\;,
\end{split}
\end{equation}
is a $K$-simple Kantor triple system with Tits-Kantor-Koecher Lie algebra $\fg=E_8$ and derivation algebra $$\der(\SS^+)=\stab_{\fso(U)}(I)\simeq\begin{cases}\fso(3,\CC)\oplus\fso(11,\CC)\;\text{if}\;\dim W=3\;,\\
\fso(7,\CC)\oplus\fso(7,\CC)\;\text{if}\;\dim W=7\;.
\end{cases}$$
\end{theorem}
\section{The exceptional Kantor triple systems of contact type}
\label{sec:contactKTS}

\subsection{The case $\fg=G_2$}\hfill\vskip0.1cm\par
We now determine the Kantor triple system $V$ with Tits-Kantor-Koecher pair $(\fg=G_2,\sigma)$, see also Example \ref{ex:G2I}. To this aim, we consider a symplectic form $\omega$ on $\CC^2$ and a compatible complex structure $J:\CC^2\to\CC^2$, i.e., such that $\omega(Jx,Jy)=\omega(x,y)$ for all $x,y\in\CC^2$. 
\begin{theorem}
\label{thm:KTSG2}
The vector space $S^3 \CC^2$ with the triple product given by 
\begin{equation*}
\begin{split}
(x^3y^3z^3)&=-\frac{3}{2}(\omega(x,Jy))^2(\omega(x,z)Jy\odot z^2+\omega(Jy,z)x\odot z^2)+\frac{1}{2}(\omega(x,Jy))^3z^3
\end{split}
\end{equation*}
for all $x,y,z\in\CC^2$ is a $K$-simple Kantor triple system. Its associated Tits-Kantor-Koecher pair $(\fg,\sigma)$ is the unique pair with $\fg=G_2$ and its derivation algebra is $\der(S^3\CC^2)=\fso(2,\CC)$.
\end{theorem}
The proof of Theorem \ref{thm:KTSG2} will occupy the remaining part of the section. 
We first recall that the unique fundamental $5$-grading of $G_2$ is associated with the marked Dynkin diagram 
${\!\!\phantom{t}_{
\begin{tikzpicture}
\node[root]   (1)                 {};
\node[xroot] (2) [right=of 1]  {}   edge [ltriplearrow] (1) edge [-] (1);
\end{tikzpicture}}}
$,
with $\fg_0\simeq \fsl(2,\CC)\oplus\CC E$, where $E$ is the grading element, and 
$$\fg_{\pm 2}\simeq\CC\;,\qquad\fg_{\pm 1}\simeq S^3 \CC^2\;,$$ 
as $\fsl(2,\CC)$-modules, by a routine examination of the roots of $G_2$. We fix a basis $\1$ of $\fg_{-2}$ and denote decomposable elements of $\fg_{-1}$ by $x^3$, $y^3$, where $x,y\in\CC^2$. The negatively graded part $\fm=\fg_{-2}\oplus\fg_{-1}$ of $\fg$ has non-trivial Lie brackets
\begin{equation}
\label{eq:G2negLie}
[x^3,y^3]=(\omega(x,y))^3\1\;,
\end{equation}
and $\fg_0$ acts naturally on $\fm$ by $0$-degree derivations. In particular $\fm$ can be extended to the non-positively $\ZZ$-graded Lie algebra $\fg_{\leq 0}=\fm\oplus\fg_0$ and it is well known that $G_2$ is precisely the maximal transitive prolongation of $\fg_{\leq 0}$, see e.g. \cite{MR1274961}.

For any $x^3\in S^3\CC^2$ we introduce a linear map $\widehat x^3:\fm\to\fg_{-1}\oplus\fg_0$ of degree $1$ by
\begin{equation}
\label{eq:G2grade1}
\begin{split}
[\widehat x^3,y^3]&=\tfrac{1}{2}(\omega(x,y))^2\underbrace{(x\otimes\omega(y,-)+y\otimes\omega(x,-))}_{\text{element of}\;\fsl(2,\CC)}-\tfrac{1}{2}(\omega(x,y))^3 E\;,\\
[\widehat x^3,\1]&=x^3\;,
\end{split}
\end{equation}
where $y^3\in\fg_{-1}$. A straightforward computation tells us that $\widehat x^3$ is an element of the first prolongation $\fg_1$ of $\fg_{\leq 0}$ for all $x^3\in S^3\CC^2$, therefore $\fg_{1}=\left\langle \widehat x^3\,|\,x^3\in S^3\CC^2\right\rangle$. Similarly $\fg_{2}=\CC\widehat\1$, where
$\widehat\1:\fm\to\fg_{0}\oplus\fg_{1}$ is given by
\begin{equation}
\label{eq:G2grade2}
\begin{split}
[\widehat\1,y^3]&=\widehat y^3\;,\qquad
[\widehat\1,\1]=-E\;,
\end{split}
\end{equation}
where $y^3\in\fg_{-1}$. The remaining Lie brackets of $G_2$ can be directly computed using transitivity and \eqref{eq:G2grade1}-\eqref{eq:G2grade2}. They are given by the natural structure of Lie algebra of $\fg_0$ and
\begin{equation}
\label{eq:G2missingbrackets}
\begin{array}{lll}
[A,\widehat y^3]=\widehat{[A,y^3]}\;,&[E, \widehat y^3]=\widehat y^3\;,&[\widehat x^3, \widehat y^3]=(\omega(x,y))^3\widehat\1\;,\\[0cm]
[A, \widehat \1]=0\;,&[E, \widehat \1]=2\widehat\1\;,&
\end{array}
\end{equation}
where $A\in \fsl(2,\CC)$ and $\widehat x^3, \widehat y^3\in\fg_{1}$. 

Now we define a linear map $\sigma:\fg\to\fg$ by
\begin{equation}
\label{eq:G2inv}
\begin{array}{ll}
\sigma(\1)=\widehat \1\;,&\sigma(x^3)=\widehat{Jx^3}\;,\\[0cm]
\sigma(A)=-A^t\;,&\sigma(E)=-E\;,\\[0cm]
\sigma(\widehat x^3)=-Jx^3\;,&\sigma(\widehat\1)=\1\;,
\end{array}
\end{equation}
where we denoted the natural extension of $J$ to a complex structure on $S^3\CC^2$ by the same symbol. 
Using \eqref{eq:G2negLie}-\eqref{eq:G2missingbrackets}, one sees that \eqref{eq:G2inv} is a Lie algebra morphism, therefore a grade-reversing involution of $G_2$. 

The associated triple product is given by the usual formula
and it is invariant under the action of the stabilizer $\stab_{\fsl(2,\CC)}(J)=\fso(2,\CC)$ of $J$ in $\fsl(2,\CC)$ -- this is indeed the associated algebra of derivations, see Example \ref{ex:G2I}.
\subsection{The case $\fg=F_4$}
\hfill\vskip0.1cm\par
\label{subsec:F4contact}
The grading with marked Dynkin diagram 
$
\begin{tikzpicture}
\node[xroot]   (1)   [label=${}$]                  {};
\node[root] (2) [right=of 1] [label=${}$] {} edge [-] (1);
\node[root]   (3) [right=of 2] [label=${}$] {} edge [rdoublearrow] (2);
\node[root]   (4) [right=of 3] [label=${}$] {} edge [-] (3);
\end{tikzpicture}
$
admits one grade-reversing involution, the complexification of the  Cartan involution of the split real form, see Table \ref{tab:satakeExceptional}. By Theorem \ref{thm:symmetryalg}, the derivation algebra of the associated KTS is the complexification of the maximal compact subalgebra $\mathfrak{u}(3)$ of $\mathfrak{sp}(6,\RR)$, i.e. $\mathfrak{gl}(3,\CC)$.

To describe the triple system, we consider a symplectic form $\omega$ on $\CC^6$
and set $V=\Lambda^3_0\CC^6$ to be the space of primitive $3$-forms (kernel of natural contraction $\imath_\omega:\Lambda^3 \CC^6\to\CC^6$ by $\omega$). Let $J:\CC^6\to\CC^6$ be a complex structure on $\CC^6$ compatible with $\omega$ and denote the natural extension to a complex structure on $V$ by the same symbol.

We will tacitly identify $S^2 \CC^6$ with $\mathfrak{sp}(6,\CC)$ by means of the symplectic form
$$
x\odot y:z\mapsto \omega(x,z)y+\omega(y,z)x\;,
$$
where $x,y,z\in\CC^6$, and introduce $\mathfrak{sp}(6,\CC)$-equivariant operators
\begin{align}
\label{eq:bullet}
\bullet&:\Lambda^2V\to\CC\;,\\
\label{eq:proj}
\vee&:S^2 V\to\mathfrak{sp}(6,\CC)\;,
\end{align}
on $V$. The first operator is the non-degenerate skewsymmetric bilinear form on $V$ given by $$\alpha\wedge\beta=(\alpha\bullet\beta)\vol\;,$$ where $\alpha,\beta\in V$ and $\vol=\frac{1}{3!}\omega^3\in\Lambda^6\CC^6$ is the normalized volume. The second operator is symmetric on $V$ and given by (a multiple of) the projection to $S^2 \CC^6$ w.r.t. the decomposition
$$S^2V\simeq (V\circledcirc V)\oplus S^2\CC^6$$
  into $\mathfrak{sp}(6,\CC)$-irreducible submodules of $S^2 V$. In our conventions this operator is normalized so that $$\alpha\vee\beta=\sum_{i=1}^3p_i\odot q_i$$ when $\alpha=p_1\wedge p_2\wedge p_3$ and $\beta=q_1\wedge q_2\wedge q_3$, where
$\left\{p_i,q_i\mid i=1,2,3\right\}$ is a fixed basis of $\CC^6$ satisfying $\omega(p_i,p_j)=\omega(q_i,q_j)=0$, $\omega(p_i,q_j)=\delta_{ij}$.
\begin{theorem}
\label{thm:F4KTSfirst}
The vector space $V=\Lambda^3_0 \CC^6$ with the triple product given by 
\begin{equation*}
\begin{split}
(\alpha\beta\gamma)&=\frac{1}{2}(\alpha\bullet J\beta)\gamma+\frac{1}{2}(\alpha\vee J\beta)\cdot\gamma
\end{split}
\end{equation*}
for all $\alpha,\beta,\gamma\in V$ is a $K$-simple Kantor triple system with derivation algebra $\der(V)=\stab_{\mathfrak{sp}(6,\CC)}(J)\simeq\mathfrak{gl}(3,\CC)$ and Tits-Kantor-Koecher Lie algebra $\fg=F_4$.
\end{theorem}
\begin{proof}The graded components of the $5$-grading of $\fg=F_4$ are $\fg_0\simeq \mathfrak{sp}(6,\CC)\oplus\CC E$ and $\fg_{\pm 2}\simeq\CC$, $\fg_{\pm 1}\simeq V$ as $\mathfrak{sp}(6,\CC)$-modules. The negatively graded part $\fm=\fg_{-2}\oplus\fg_{-1}$ of $\fg$ has non-trivial Lie brackets
\begin{equation}
\label{eq:F4negLie}
[\alpha,\beta]=(\alpha\bullet\beta)\1\;,
\end{equation}
where $\1$ is a fixed basis of $\fg_{-2}$. In particular $\fm$ can be extended to the non-positively $\ZZ$-graded Lie algebra $\fg_{\leq 0}=\fm\oplus\fg_0$ and $F_4$ is the maximal transitive prolongation of $\fg_{\leq 0}$ \cite{MR1274961}.

For any $\alpha\in V$ we introduce a linear map $\widehat \alpha:\fm\to\fg_{-1}\oplus\fg_0$ by
\begin{equation}
\label{eq:F4grade1}
\begin{split}
[\widehat \alpha,\beta]&=\lambda\alpha\vee\beta+\mu(\alpha\bullet\beta) E\;,\\
[\widehat \alpha,\1]&=\alpha\;,
\end{split}
\end{equation}
for all $\beta\in\fg_{-1}$, where $\lambda,\mu\in\CC$ are constants to be determined imposing $\widehat\alpha\in\fg_1$:
\begin{align*}
0&=[\widehat\alpha,[\beta,\1]]=[[\widehat\alpha,\beta],\1]+[\beta,[\widehat\alpha,\1]]\\
&=\mu(\alpha\bullet\beta)[E,\1]+[\beta,\alpha]\\
&=(-2\mu-1)(\alpha\bullet\beta)\1\;.
\end{align*}
Hence $\mu=-\frac{1}{2}$ and a similar computation with $\widehat\alpha$ acting on $\fg_{-2}=[\fg_{-1},\fg_{-1}]$ yields $\lambda=-\frac{1}{2}$. In other words we just showed that $\fg_{1}=\left\langle \widehat \alpha\,|\,\alpha\in V\right\rangle$ when $\mu=\lambda=-\frac{1}{2}$; similarly $\fg_{2}=\CC\widehat\1$, where
$\widehat\1:\fm\to\fg_{0}\oplus\fg_{1}$ is given by
\begin{equation}
\label{eq:F4grade2}
\begin{split}
[\widehat\1,\beta]&=\widehat \beta\;,\qquad
[\widehat\1,\1]=-E\;,
\end{split}
\end{equation}
for all $\beta\in\fg_{-1}$. The remaining Lie brackets of $F_4$ can be directly computed using transitivity and \eqref{eq:F4grade1}-\eqref{eq:F4grade2}. They are given by the natural structure of Lie algebra of $\fg_0$ and 
\begin{equation}
\label{eq:F4missingbrackets}
\begin{array}{lll}
[A,\widehat \alpha]=\widehat{[A,\alpha]}\;,& [E, \widehat \alpha]=\widehat \alpha\;,&[\widehat \alpha, \widehat \beta]=(\alpha\bullet\beta)\widehat\1\;,\\[0cm]
[A, \widehat \1]=0\;,&[E, \widehat \1]=2\widehat\1\;,&
\end{array}
\end{equation}
where $A\in \mathfrak{sp}(6,\CC)$ and $\widehat \alpha, \widehat \beta\in\fg_{1}$. 

Using \eqref{eq:F4negLie}-\eqref{eq:F4missingbrackets}, one can directly check that
the map $\sigma:\fg\to\fg$ defined by
\begin{equation}
\label{eq:F4inv}
\begin{array}{ll}
\sigma(\1)=\widehat \1\;,&\sigma(\alpha)=\widehat{J\alpha}\;,\\[0cm]
\sigma(A)=-A^t\;,&\sigma(E)=-E\;,\\[0cm]
\sigma(\widehat \alpha)=-J\alpha\;,&\sigma(\widehat\1)=\1\;,
\end{array}
\end{equation}
is a grade-reversing involution of $F_4$. 
\end{proof}
\subsection{The case $\fg=E_6$}
\hfill\vskip0.1cm\par
The KTS with Tits-Kantor-Koecher pair $(\fg=E_6,\sigma)$ are $8$, associated with $3$ gradings. We already saw $4$ of them in \S \ref{sec:E6with2center} and we study here those associated with the grading of contact type
$$
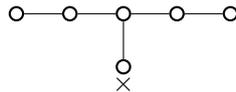
\begin{tikzpicture}
\node[root]   (1)                     {};
\node[root]   (2) [right=of 1] {} edge [-] (1);
\node[root]   (3) [right=of 2] {} edge [-] (2);
\node[root]   (4) [right=of 3] {} edge [-] (3);
\node[root]   (5) [right=of 4] {} edge [-] (4);
\node[xroot]   (6) [below=of 3] {} edge [-] (3);
\end{tikzpicture}
$$
There are $3$ non-equivalent grade-reversing involutions, related to the real forms $EI, EII$ and $EIII$,
with derivation algebras respectively $\fso(6,\CC)$, $(\fsl(3,\CC)\oplus\fsl(3,\CC))\oplus\CC$ and $\fsl(5,\CC)\oplus\CC$.

Let $V=\Lambda^3\CC^6$ and $vol\in\Lambda^6\CC^6$ a fixed volume.
As for the case $F_4$ in \S\ref{subsec:F4contact} we introduce the $\fsl(6,\CC)$-equivariant operators
$$
\begin{array}{cl}
\bullet&:\Lambda^2V\to\CC\;,\\
\vee&:S^2 V\to\mathfrak{sl}(6,\CC)\;
\end{array}
$$
with the $\bullet$ defined via wedge product as in \S\ref{subsec:F4contact} and $\vee$ defined by
$$(\alpha\vee\beta)(x)=i_{\alpha\bullet}(\beta\wedge x)+i_{\beta\bullet}(\alpha\wedge x),$$
for all $\alpha,\beta\in V$, where $i_{\alpha\bullet}$ is the contraction by $\alpha\bullet\in V^*\cong\Lambda^3(\CC^6)^*$.
We denote by $M_{1}$, $M_{2}$, the endomorphism of $\CC^6$ represented by the matrices
\begin{equation}\label{eq:E6MatricecInvolNonsplit}
\small{M_1=\left(\begin{matrix}
   0&0&0&0&0&1\\
   0&0&0&0&-1&0\\
   0&0&0&1&0&0\\
   0&0&1&0&0&0\\
   0&-1&0&0&0&0\\
   1&0&0&0&0&0\\
  \end{matrix}\right)
  \text{, }\quad
  M_2=\left(\begin{matrix}
   0&0&0&0&0&-1\\
   0&1&0&0&0&0\\
   0&0&1&0&0&0\\
   0&0&0&1&0&0\\
   0&0&0&0&1&0\\
   -1&0&0&0&0&0\\
  \end{matrix}\right)}
\end{equation}
and by $\eta$ the scalar product on $\CC^6$ of signature $(+\,-\,+\,-\,+\,-)$. 
We let $\star:\Lambda^p\CC^6\to\Lambda^{6-p}\CC^6$ be the corresponding Hodge star operator.
\begin{theorem}
\label{thm:E6KTSSplit}
 The vector space $V=\Lambda^3\CC^6$ with triple product
 $$(\alpha\beta\gamma)=\half (\alpha\bullet \star\beta)\gamma-\half (\alpha\vee \star\beta)\cdot\gamma$$
 for all $\alpha,\beta,\gamma\in V$ is a K-simple Kantor triple system with derivation algebra 
 $\der(V)=\stab_{\fsl(6,\CC)}(\eta)\simeq\fso(6,\CC)$ and Tits-Kantor-Koecher Lie algebra $\fg=E_6$.
\end{theorem}

\begin{proof}
 The graded components of $\fg$ are $\fg_0\simeq \fsl(6,\CC)\oplus\mathbb C E,\ \fg_{\pm1}\simeq V$ and $\fg_{\pm2}\simeq \CC$.
 Direct arguments similar to those made for the contact grading of $F_4$ show that the Lie brackets are given
 by the natural action of the Lie algebra $\fg_0$ and
 $$\begin{array}{ll}
[\alpha,\beta]=(\alpha\bullet\beta)\1\;,&[\widehat \alpha, \widehat \beta]=(\alpha\bullet\beta)\widehat\1\;,  \\[5pt]
[\widehat \alpha,\beta]=\half\alpha\vee\beta-\half(\alpha\bullet\beta) E\;, &\ [\widehat \alpha,\1]=\alpha\;,\\[5pt]
[\widehat\1,\beta]=\widehat \beta\;,&\ [\widehat\1,\1]=-E\;,\\[5pt]
[A,\widehat \alpha]=\widehat{[A,\alpha]}\;, &\
[A, \widehat \1]=0\;,
\end{array}
$$
with $\alpha,\beta\in\fg_{-1},\ \widehat\alpha,\widehat\beta\in\fg_1$, $\1\in\fg_{-2}$, $\widehat\1\in\fg_2$ and $A\in\fsl(6,\CC)$.
The grade reversing involution is given by
\begin{equation}
\label{eq:E6ChevInv}
\begin{array}{ll}
\sigma(\1)=\widehat \1\;,&\sigma(\alpha)=\widehat{\star\alpha}\;,\\[0cm]
\sigma(A)=-A^t\;,&\sigma(E)=-E\;,\\[0cm]
\sigma(\widehat \alpha)=\star\alpha\;,&\sigma(\widehat\1)=\1\;.
\end{array}
\end{equation}
and the triple product follows at once.
\end{proof}

\begin{theorem}
\label{thm:E6KTSNonSplit}
 The vector space $V=\Lambda^3\CC^6$ with triple product
 $$(\alpha\beta\gamma)=\half (\alpha\bullet M\beta)\gamma-\half (\alpha\vee M\beta)\cdot\gamma$$
 for all $\alpha,\beta,\gamma\in V$ and $M=M_1$ or $M_2$ as in (\ref{eq:E6MatricecInvolNonsplit}) is a K-simple Kantor triple system with
 $$\der(V)=\stab_{\fsl(6,\CC)}(M)\simeq\begin{cases}
                                                   \fsl(3,\CC)\oplus\fsl(3,\CC)\oplus\CC&\text{ if } M=M_1\\
                                                   \fsl(5,\CC)\oplus\CC&\text{ if } M=M_2\\
                                                   \end{cases}
$$
 and Tits-Kantor-Koecher Lie algebra $\fg=E_6$.
\end{theorem}

\begin{proof}
The restriction of the involution $\sigma$ to $\fsl(6,\CC)$ is either the complexification of the Cartan involution of 
$\mathfrak{su}(3,3)$ or $\mathfrak{su}(1,5)$. By \cite{MR1064110}, one has that $\sigma(A)=MAM^{-1}$ for any $A\in\fsl(6,\CC)$, where $M\in GL(6,\CC)$ is
conjugated to either $Diag(Id_3,-Id_{3})$ or $Diag(Id_5,-1)$. In the first case we set $M=M_1$ and in the second $M=M_2$.

If we extend the natural action of $M$ on $\fg_{\pm 1}$ to a grade-reversing involution of $\fg$ then necessarily
$$\begin{array}{ll}
\sigma[\widehat\alpha,\beta]=[\sigma(\widehat \alpha),\sigma(\beta)]&=[M\alpha,\widehat{M\beta}]=-\half(M\alpha\bullet M\beta )E-\half (M\alpha \vee M\beta)\\
&=\half(\alpha\bullet \beta)E+\half M(\alpha\vee\beta)M\;,
  \end{array}
$$
where we used that $(M\alpha\bullet M\beta )=-(\alpha\bullet \beta)$ and $(M\alpha \vee M\beta)=-M(\alpha\vee\beta)M$. In particular 
$\sigma(\alpha\vee\beta)=M(\alpha\vee\beta)M^{-1}$, hence $\sigma(A)=MAM^{-1}$ for all $A\in\fsl(6,\CC)$ and
we arrive at the involution
\begin{equation}
\label{eq:E6NonChevInv}
\begin{array}{ll}
\sigma(\1)=-\widehat \1\;,&\sigma(\alpha)=\widehat{M\alpha}\;,\\[0cm]
\sigma(A)=MAM^{-1}\;,&\sigma(E)=-E\;,\\[0cm]
\sigma(\widehat \alpha)=M\alpha\;,&\sigma(\widehat\1)=-\1\;,
\end{array}
\end{equation}
The claim on the algebra of derivations follows from $\stab_{\fsl(6,\CC)}(M)\simeq\fsl(k,\CC)\oplus\fsl(6-k,\CC)\oplus\CC$ where 
$k=3,1$ for $M$ is conjugate to $M_1$ and $M_2$, respectively. 
\end{proof}
\subsection{The case $\fg=E_7$}
\hfill\vskip0.1cm\par
The Kantor triple systems with Tits-Kantor-Koecher pair $(\fg=E_7,\sigma)$ are $7$, associated to $3$ different $5$-gradings. The contact grading
$$
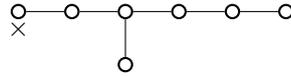
\begin{tikzpicture}
\node[xroot]   (1) [label=\;\;\;${{}^{^{\phantom{T^T}}}}$]             {};
\node[root]   (3) [right=of 1] {} edge [-] (1);
\node[root]   (4) [right=of 3] {} edge [-] (3);
\node[root]   (5) [right=of 4] {} edge [-] (4);
\node[root]  (6) [right=of 5]{} edge [-] (5);
\node[root]   (7) [right=of 6]{} edge [-] (6);
\node[root]   (2) [below=of 4] {} edge [-] (4);
\end{tikzpicture}
$$
of $E_7$ admits $3$ grade-reversing involutions, corresponding to the real forms $EV$, $EVI$ and $EVII$, with derivation algebras respectively $\fso(6,\CC)\oplus\fso(6,\CC)$, $\mathfrak{gl}(6,\CC)$ and $\fso(2,\CC)\oplus\fso(10,\CC)$. 

The corresponding systems are close to those of extended Poincar\'e type studied in \S\ref{sec:extendedKTS}. Indeed $\fg_0\simeq \fso(U)\oplus\CC E$ and $\fg_{\pm 2}\simeq \CC$, $\fg_{\pm 1}\simeq\SS^+$ as $\fso(U)$-modules.
Here $(U,\eta)$ is a $12$-dimensional complex vector space $U$ with a non-degenerate symmetric bilinear form $\eta$ and $$\SS=\SS^+\oplus \SS^-$$ the decomposition into semispinors of the $64$-dimensional spinor $\fso(U)$-module $\SS$. 

In this section, we will make extensive use of the explicit realization \eqref{eq:exreal} of the Clifford algebra $\Cl(U)$ as endomorphisms of $$\SS=\underbrace{\CC^2\otimes\cdots\otimes\CC^2}_{k-times}\;,\qquad k=6\;.$$ 
We note that in our conventions $\SS^\pm$ are the $\pm 1$-eigenspaces of the involution $$i^{k}\vol=T\otimes\cdots\otimes T\;,$$
that is the opposite of $\vol$, as $k=6$. There exists an admissible bilinear form $\beta:\SS\otimes\SS\to\CC$ with the invariants $(\tau,\sigma,\imath)=(-1,-1,1)$. 
It can be easily described using \eqref{eq:exreal}:
if $<-,->$ (resp. $\omega$) is the standard $\CC$-linear product (resp. symplectic form) on $\CC^2$, then
\begin{equation}
\label{eq:betadm}
\beta(x_1\otimes \cdots\otimes x_6,y_1\otimes\cdots\otimes y_6)=<x_1,y_1>\omega(x_2,y_2)<x_3,y_3>\cdots\omega(x_6,y_6)\;,
\end{equation}
where $x_i, y_i\in\CC^2$, $i=1,\ldots,6$. In particular $\SS^\pm$ are orthogonal. As in \S\ref{sec:ePF4}, there exists an operator
\begin{equation*}
\begin{aligned}
\Gamma^{(2)}:\SS\otimes\SS &\to\fso(U)\;,\\
\eta(\Gamma^{(2)}(s,t)u,v)&=\beta(u\wedge v\circ s,t)\;,
\end{aligned}
\end{equation*} 
where $s,t\in \SS$ and $u,v\in U$. It is a symmetric $\fso(U)$-equivariant operator.
\vskip0.1cm\par

We fix a basis $\1$ of $\fg_{-2}$ and write the graded Lie algebra structure on $\fm=\fg_{-2}\oplus\fg_{-1}$ as
$$
[s,t]=\beta(s,t)\1\;,
$$
for all $s,t\in\SS^+$. We now turn to describe the positively-graded elements of $\fg$. For any $s\in\SS^+$ we introduce the linear map $\widehat s:\fm\to\fg_{-1}\oplus\fg_{0}$ of degree $1$ by
\begin{equation}
\label{eq:E7c1}
\begin{split}
[\widehat s,t]&=\lambda\Gamma^{(2)}(s,t)+\mu\beta(s,t)E\;,\\
[\widehat s, \1]&=s\;,
\end{split}
\end{equation}
where $t\in\SS^+$ and $\lambda,\mu\in\CC$ are constants to be determined. We note that the equations \eqref{eq:E7c1} are not identitical to those encountered in the extended Poincar\'e case, see e.g. \eqref{eq:1F4}, as here there is no Clifford multiplication of elements from $\fg_{-2}$ with $\fg_{-1}$.
\begin{proposition}
The map $\widehat s\in\fg_1$ for all $s\in\SS^+$ if and only if $\lambda=\tfrac{1}{2}$, $\mu=-\tfrac{1}{2}$.
\end{proposition}
\begin{proof}
We compute
\begin{equation*}
\begin{aligned}
0&=[\widehat s,[t,\1]]=[[\widehat s,t],\1]+[t,[\widehat s,\1]]\\
&=-2\mu\beta(s,t)\1+\beta(t,s)\1
\end{aligned}
\end{equation*}
for all $s,t\in\SS^+$ and directly infer $\mu=-\tfrac{1}{2}$. The proof of $\lambda=\tfrac{1}{2}$ is more involved and it relies on the explicit realization \eqref{eq:exreal}. We depart with 
\begin{equation*}
\begin{aligned}
{}[\widehat{s},[t,r]]&=\beta(t,r)s
\end{aligned}
\end{equation*}
and note that
\begin{equation*}
\begin{aligned}
{}[\widehat{s},[t,r]]&=[[\widehat s,t],r]+[t,[\widehat s,r]]\\
&=\lambda\Gamma^{(2)}(s,t)\cdot r-\lambda\Gamma^{(2)}(s,r)\cdot t\\
&\;\;\;+\tfrac{1}{2}\beta(s,t)r-\tfrac{1}{2}\beta(s,r)t
\end{aligned}
\end{equation*}
for all $s,t,r\in\SS^+$. We now choose suitable spinors
$$
s=r=\begin{pmatrix}
1 \\ +i
\end{pmatrix}^{\!\!\otimes 6}\;,\qquad t=\begin{pmatrix}
1 \\ -i
\end{pmatrix}^{\!\!\otimes 6}
$$ 
and use \eqref{eq:betadm} to get $\beta(s,r)=0$, $\beta(s,t)=64i$. Longer but straightforward computations similar to those of Proposition \ref{prop:percitaredopo} 
yield
\begin{equation*}
\begin{aligned}
\Gamma^{(2)}(s,t)
&=64\sum_{l\;odd}e_l\wedge e_{l+1}
\end{aligned}
\end{equation*}
and $\Gamma^{(2)}(s,s)=0$.

We are therefore left with the identity
\begin{equation*}
\begin{aligned}
-64i s&=\lambda\Gamma^{(2)}(s,t)\cdot s+32i s\\
&= 64\lambda\sum_{l\;odd}e_l\wedge e_{l+1}\cdot s+32 is\\
&= -\frac{6i}{2}64\lambda s+32is\;,
\end{aligned}
\end{equation*}
from which $\lambda=\frac{1}{2}$ follows.
\end{proof}
Finally we consider the generator $\widehat \1:\fm\to\fg_{0}\oplus\fg_{1}$ of the $1$-dimensional component $\fg_2$ of $\fg$ defined by
\begin{equation}
\label{eq:E7c2}
\begin{split}
[\widehat \1,t]&=\widehat t\;,\qquad
[\widehat \1, \1]=-E\;,
\end{split}
\end{equation}
and note that $[\widehat s,\widehat t]=\beta(s,t)\widehat\1$ for all $s,t\in \SS^+$. This completes the description of Lie brackets of $\fg=E_7$. We now turn to the KTS associated with the $3$ different grade-reversing involutions. 

The proof of Theorem \ref{thm:E7KTScontact} is similar but not identical to that of Theorem \ref{thm:F4KTSsecond} and we therefore outline its main steps.
\begin{theorem}
\label{thm:E7KTScontact}
Let $U$ be a $12$-dimensional complex vector space with a non-degenerate symmetric bilinear form $\eta$ and $\SS=\SS^+\oplus\SS^-$ the associated $64$-dimensional spinor module. Let $\beta:\SS\otimes\SS\to\CC$ be the admissible bilinear form on $\SS$ with invariants $(\tau,\sigma,\imath)=(-1,-1,1)$ and
$\Gamma^{(2)}:\SS\otimes\SS\to \fso(U)$ the operator given by
$$\eta(\Gamma^{(2)}(s,t)u,v)=\beta(u\wedge v\circ s,t)\;,$$
where $s,t\in \SS$ and $u,v\in U$. 
\begin{itemize}
	\item[(1)] Fix an orthogonal decomposition $$U=W\oplus W^\perp$$ of $U$ with $\dim W=6$ or $\dim W=2$ and let $I=i\vol_W$, where $\vol_W\in\Cl(U)$ is the volume of $W$. ($I$ acts on $\SS$ as a paracomplex structure in both cases.) Then the semispinor module $\SS^+$ with the triple product
\begin{equation}
\label{eq:E7Cproduct1}
\begin{split}
(rst)&=-\frac{1}{2}\Gamma^{(2)}(r,I\circ s)\cdot t+\frac{1}{2}\beta(r,I\circ s)t\;,\qquad r,s,t\in \SS^+\;,
\end{split}
\end{equation}
is a $K$-simple Kantor triple system with Tits-Kantor-Koecher Lie algebra $\fg=E_7$ and
$$\der(\SS^+)=\stab_{\fso(U)}(I)\simeq\begin{cases}\fso(6,\CC)\oplus\fso(6,\CC)\;\text{if}\;\dim W=6\;,\\
\fso(2,\CC)\oplus\fso(10,\CC)\;\text{if}\;\dim W=2\;.
\end{cases}$$
\item[(2)] Fix a split decomposition $U=W\oplus W^*$ of $U$ into the direct sum of two isotropic $6$-dimensional subspaces and let 
$$I_t=\exp(tX)\in\Spin(U)\;,\qquad X=\begin{pmatrix} Id_W &0 \\
0 &-Id_{W^*}\end{pmatrix}\;,
$$
be the complex $1$-parameter subgroup of the spin group generated by $X\in\fso(U)\simeq\mathfrak{spin}(U)$. If we set $I=I_t$ for $t=i\frac{\pi}{2}$ then the semispinor module $\SS^+$ with the triple product
\begin{equation}
\label{eq:E7Cproduct2}
\begin{split}
(rst)&=-\frac{i}{2}\Gamma^{(2)}(r,I\circ s)\cdot t+\frac{i}{2}\beta(r,I\circ s)t\;,\qquad r,s,t\in \SS^+\;,
\end{split}
\end{equation}
is a $K$-simple Kantor triple system with Tits-Kantor-Koecher Lie algebra $\fg=E_7$ and
$$\der(\SS^+)=\stab_{\fso(U)}(I)\simeq \mathfrak{gl}(W)\;.$$
\end{itemize}
\end{theorem}
\begin{proof}
(1) Let $\sigma:\fg\to\fg$ be the grade-reversing map defined by
\begin{equation}
\label{eq:E7contactA}
\begin{array}{ll}
\sigma(\1)=-\widehat{\1}\;,&\sigma(s)=\widehat{I\circ s}\;,\\[0cm]
\sigma(A)=r_W A r_W\;,&\sigma(E)=-E\;,\\[0cm]
\sigma(\widehat s)=I\circ s\;,&\sigma(\widehat \1)=-\1\;,
\end{array}
\end{equation}
where $s\in\SS^+$ and $A\in\fso(U)$. It is a Lie algebra involution, as $I\in\mathrm{Spin}(U)$ covers the opposite of the orthogonal reflection $r_W:U\to U$ across $W$, i.e., $\widetilde{\operatorname{Ad}}_I=-r_W$ and $\beta(I\circ s,t)=-\beta(s, I\circ t)$ for all $s,t\in \SS$.\hfill 
\vskip0.1cm\par\noindent
(2) First note that $I$ acts as a complex structure on $U$ via the spin cover $\widetilde{\operatorname{Ad}}:\Spin(U)\to \SO(U)$. In particular $I^2\in\Spin(U)$ sits in the twisted center 
$$\operatorname{twcen}\Cl(U)=\left\{a\in\Cl(U)\mid a\circ v=-v\circ a\;\;\text{for all}\;\;v\in V\right\}$$ 
of $\Cl(U)$, which is known to be generated by $\vol$, as $\dim U$ is even. It follows that
$I^2=\alpha\vol$ for some fourth root of unity $\alpha$ and a direct computation with exponentials reveals that $I^2\circ s=-s$ for the positive-chirality spinor 
$$s=\begin{pmatrix}
1 \\ +i
\end{pmatrix}^{\!\!\otimes 6}\in\mathbb S^+\;.
$$
In other words, $I^2=\vol$ in the Clifford algebra (recall that $\SS^+=+1$-eigenspace of $-\vol$) and $I$ acts as a complex structure both on $U$ and $\SS^+$.

Let $\sigma:\fg\to\fg$ be the grade-reversing map defined by
\begin{equation}
\label{eq:E7contactB}
\begin{array}{ll}
\sigma(\1)=-\widehat{\1}\;,&\sigma(s)=i\widehat{I\circ s}\;,\\[0cm]
\sigma(A)=I A I^{-1}\;,&\sigma(E)=-E\;,\\[0cm]
\sigma(\widehat s)=iI\circ s\;,&\sigma(\widehat \1)=-\1\;,
\end{array}
\end{equation}
and note that $\sigma^2=1$ by the above arguments. Using the expressions of the Lie brackets of $\fg=E_7$ and the equivariance properties
$$\beta(I\circ s,I\circ t)=\beta(s, t)\;,\qquad\Gamma^{(2)}(I\circ s,I\circ t)=I\Gamma^{(2)}(s,t)I^{-1}$$ 
for all $s,t\in \SS$, we get that $\sigma$ is a Lie algebra involution, with associated KTS \eqref{eq:E7Cproduct2}. 
\end{proof}
\subsection{The case $\fg=E_8$}
\label{sec:contactE8}
\hfill\vskip0.1cm\par
Let 
\begin{equation}
\label{eq:contactE8}
\begin{tikzpicture}
\node[root]  (1)  [label=\;\;\;${{}^{^{\phantom{T^T}}}}$]            {};
\node[root]   (3) [right=of 1] {} edge [-] (1);
\node[root]   (4) [right=of 3] {} edge [-] (3);
\node[root]   (5) [right=of 4] {} edge [-] (4);
\node[root]   (6) [right=of 5] {} edge [-] (5);
\node[root]   (7) [right=of 6] {} edge [-] (6);
\node[root]   (2) [below=of 4] {} edge [-] (4);
\node[xroot]   (8) [right=of 7] {} edge [-] (7);
\end{tikzpicture}
\end{equation}
be the contact grading of $\fg=E_8$, with even graded components $\fg_0=\CC E\oplus E_7$, $\fg_{\pm 2}=
\CC$ and where $\fg_{\pm 1}$ is given by the defining $56$-dimensional representation of $E_7$. 

There are two inequivalent associated  KTS, which correspond to the split real form $EVIII$ and the real form $EIX$.  The first has derivation algebra the (complexification of the) maximal compact subalgebra $\fsu(8)$ of $EV$, the second the maximal compact subalgebra of $EVII$:
\begin{equation}
\label{eq:E8symmetry}
\der(V)\simeq\begin{cases}\fsl(8,\CC)\;\text{or}\\
E_6\oplus \CC\;.
\end{cases}
\end{equation}
We anticipate that the KTS with $\der(V)=E_6\oplus\CC$ is a variation of the triple system historically first introduced by Freudenthal to describe the defining representation of $E_7$ \cite{MR0063358}.
We will introduce a paracomplex structure $I:V\to V$ which relates the triple system of Freudenthal with ours, making evident that the KTS product is {\it not} equivariant under the whole $E_7$ (as for the Freudenthal product).

We note that $\fg=E_8$ with the grading \eqref{eq:contactE8} is the only $5$-graded simple Lie algebra with $\fg_0$ of exceptional type. Due to this, instead of looking for a uniform description of $\fg$ and in particular of the adjoint action of $\fg_0$ on $\fg_{-1}$, we shall use two different presentations, for each one of the two KTS. We begin with some introductory background on Jordan algebras. 

\subsubsection{Preliminaries on cubic Jordan algebras and associated structures}
We recall that a Jordan algebra $J$ is a complex vector space equipped with a bilinear product
satisfying
\begin{equation*}
\label{eq:Jidentity}
A\circ B =B\circ A\;,\qquad (A\circ B)\circ A^2=A\circ (B\circ A^2)\;,
\end{equation*}
for all $A,B\in J$. An important class of Jordan algebras is given by the {\it cubic} Jordan algebras
developed in \cite{MR0138661,MR0238916}; we here sketch their construction, 
following the presentation of \cite{MR2014924}.
\begin{definition} A \emph{cubic norm} on a complex vector space $V$ is a homogeneous map of degree three
$N:V\to\CC$ such that its full symmetrization
\begin{equation*}
N(A, B, C):=\frac{1}{6}(N(A+ B+ C)-N(A+B)-N(A+ C)-N(B+
C)+N(A)+N(B)+N(C))
\end{equation*}
is trilinear.
\end{definition}
We say that $c\in V$ is a basepoint if $N(c) = 1$. If $V$ is a vector space equipped with a cubic norm and a fixed base point, one can define the following three maps:
\begin{subequations}\label{eq:cubicdefs}
\begin{enumerate}
\item the trace form,
    \begin{equation}
		\label{eq:tracelinear}
    \operatorname{Tr}(A)=3N(c, c, A),
    \end{equation}
\item a bilinear map,
    \begin{equation}
    S(A, B)=6N(A, B, c),
    \end{equation}
\item a trace bilinear form,
    \begin{equation}\label{eq:tracebilinearform}
    (A, B)=\operatorname{Tr}(A)\operatorname{Tr}(B)-S(A, B).
    \end{equation}
\end{enumerate}
\end{subequations}
A Jordan algebra $J$ with multiplicative identity $1=c$ may be derived from any such vector space $V$ if $N$ is \emph{Jordan}.
\begin{definition} A cubic norm is \emph{Jordan} if
\begin{enumerate}
\item the trace bilinear form \eqref{eq:tracebilinearform} is non-degenerate;
\item the quadratic adjoint map $\sharp\colon\mathfrak{J}\to\mathfrak{J}$ defined by 
$$(A^\sharp, B) = 3N(A, A, B)$$ 
satisfies $(A^{\sharp})^\sharp=N(A)A$ for all $A\in \mathfrak{J}$.
 \end{enumerate}
\end{definition}
We define the linearization of the adjoint map by
\begin{equation}
\label{eq:linearizedsharp}
A\times B := (A+B)^\sharp-A^\sharp-B^\sharp
\end{equation}
and note that $A\times A=2A^\sharp$. (We remark that other authors define $A\times B$ with an additional factor $\tfrac{1}{2}$, in their conventions $A\times A=A^\sharp$.) Every vector space with a Jordan cubic norm gives rise to a Jordan algebra with unit $1=c$  and Jordan product
\begin{equation}
\label{eq:J3prod}
A\circ B := \frac{1}{2}(A\times B+\operatorname{Tr}(A)B+\operatorname{Tr}(B)A-S(A,
B)1).
\end{equation}
We conclude with the definition of reduced structure group; for more details on cubic Jordan algebras, we refer the reader to the original sources \cite{MR0138661,MR0238916, MR2014924}, see also e.g. \cite[\S 2]{MR2344592}.
\begin{definition}
The {\it reduced structure group} of a cubic Jordan algebra $J$ with product \eqref{eq:J3prod} is the group 
$$
	\operatorname{Str}_0(J)=\left\{\tau:J\to J\mid N(\tau A)=N(A)\;\text{for all}\;A\in J\right\}
	$$
of linear invertible transformations of $J$ preserving the cubic norm. 
\end{definition}
\vskip0.1cm\par
It is well known that over the field of complex numbers, there is a unique simple finite-dimensional exceptional Jordan algebra, called Albert algebra. It is a cubic Jordan algebra of dimension $27$. The underlying vector space $V$ is the space $\mathfrak H_3(\mathfrak U)$ of $3\times 3$ Hermitian matrices over the complex Cayley algebra $\mathfrak U$ (=complex octonions) and the cubic norm $N:\mathfrak H_3(\mathfrak U)\to\CC$ the usual determinant of a matrix $A\in\mathfrak H_3(\mathfrak U)$, the only proviso being that the order of the factors and the position of the brackets in multiplying elements from $\mathfrak U$ is important. The interested reader may find the explicit expression of $N$ in e.g. \cite[eq. (16)]{MR2344592}. The trace forms \eqref{eq:tracelinear} and \eqref{eq:tracebilinearform} associated to the identity matrix as basepoint coincide in this case with the regular trace,
$$(A,B)=\frac{1}{2}\operatorname{Tr}(AB+BA)\;,$$
whereas the adjoint $A^\sharp$ of $A\in 
\mathfrak H_3(\mathfrak U)$ is the transpose of the cofactor matrix \cite[pag. 934]{MR2344592}.

The reduced structure group $\operatorname{Str}_0(\mathfrak H_3(\mathfrak U))$ of $J=\mathfrak H_3(\mathfrak U)$ is a simply connected simple Lie group of type $E_6$ and the natural action on $J$ its
$27$-dimensional defining representation. Elements $\tau\in\operatorname{Str}_0(\mathfrak H_3(\mathfrak U))$ can be equally characterized by the following identity
$$
\tau(A)\times\tau(B)=(\tau^*)^{-1}(A\times B)\;,
$$
for all $A,B\in \mathfrak H_3(\mathfrak U)$, where $\tau^*$ is the transposed of $\tau$ relative to the trace bilinear form \eqref{eq:tracebilinearform}. We recall that the defining representation $J$ of $E_6$ and its dual $J^*$ are {\it not} equivalent; in our conventions $J^*$ is still represented by the set $J$ but the action of $\operatorname{Str}_0(\mathfrak H_3(\mathfrak U))$ is $\tau\mapsto(\tau^*)^{-1}$.
We will not distinguish between $J$ and $J^*$ if the action of $\operatorname{Str}_0(\mathfrak H_3(\mathfrak U))$ is clear from the context.

\vskip0.1cm\par
We now turn to recall Freudenthal's construction of the $56$-dimensional defining representation of the Lie algebra $E_7$ from the $27$-dimensional Albert algebra \cite{MR0063358}.
It is a special case of a more general construction by Brown in \cite{MR0248185} which departs from any cubic Jordan algebra,  but for our purposes it is enough to assume
$J=\mathfrak H_3(\mathfrak U)$ from now on.   

We consider a vector space $\mathfrak F=\mathfrak F(J)$ constructed from $J$ in the following way
$$
\mathfrak F(J)=\CC\oplus\CC\oplus J\oplus J^*
$$
and write an arbitrary element $x\in \mathfrak{F}$ as a ``$2\times 2$ matrix''
\begin{equation}
\label{eq:2by2}
x=\begin{pmatrix}\alpha&A\\B&\beta\end{pmatrix}, \quad\text{where}
~\alpha, \beta\in\CC\quad\text{and}\quad A\in J, B\in J^*\;.
\end{equation}
We note that the Lie algebra $E_7$ admits a $3$-grading of the form 
\begin{equation}
\label{eq:E73grad}
\begin{aligned}
E_7&=J\oplus (E_6\oplus\CC G)\oplus J^*\;,\\
\deg(J)&=-1\;,\;\; \deg(E_6\oplus\CC G)=0\;,\;\; \deg(J^*)=1\;,
\end{aligned}
\end{equation}
where $G$ is the corresponding grading element, and following \cite{MR2194359} we will use the shortcut 
\begin{equation*}
\begin{aligned}
\Phi&=\Phi(\phi,X,Y,\nu)\in E_7\\
(\phi&\in E_6,\; X\in J,\; Y\in J^*,\; \nu\in\CC)\;,
\end{aligned}
\end{equation*}
to denote elements of $E_7$. The action of $E_7$ on $\mathfrak F$ is as follows, see \cite{MR0063358} and also \cite[\S 2.1]{MR2194359}.
\begin{proposition}
\label{eq:ftslieaction}
The representation of the Lie algebra $E_7$ on $\mathfrak F$ is given by
\begin{equation*}
\Phi(\phi,X,Y,\nu)\begin{pmatrix}\alpha&A\\B&\beta\end{pmatrix}=\begin{pmatrix}\alpha\nu+(X,B)&\phi A-\frac{1}{3}\nu A+Y\times B +\beta X \\-\phi^* B+\frac{1}{3}\nu B+X\times A+\alpha Y&-\beta\nu+(Y,A)\end{pmatrix}\;.
\end{equation*}
In particular the grading element of \eqref{eq:E73grad} is given by $G=\Phi(0,0,0,-\tfrac{3}{2})$.
\end{proposition}
Decomposing the tensor product $\mathfrak F\otimes\mathfrak F$ into irreducible representation  of 
$E_7$, one readily sees that there exist unique (up to multiples) $E_7$-equivariant maps
$$\left\{\cdot,\cdot\right\}:\mathfrak F\otimes\mathfrak F\to\CC$$ and 
$$\times:\mathfrak F\otimes\mathfrak F\to E_7.$$
The first map is the standard symplectic form
\begin{equation}
\left\{x,y\right\}=\alpha\delta-\beta\gamma+(A,D)-(B,C)\;,
\label{eq:E7symp}
\end{equation}
where
\begin{equation}
\label{eq:elF}
x=\begin{pmatrix}\alpha&A\\B&\beta\end{pmatrix}\;,\qquad y=\begin{pmatrix}\gamma&C\\D&\delta\end{pmatrix}\;,
\end{equation}
are elements of $\mathfrak F$. The second is the so-called {\it Freudenthal product}, an appropriate extension  of the operation \eqref{eq:linearizedsharp} on $J$ to the whole $\mathfrak F$, see e.g. \cite[\S 2]{MR2194359}.
\begin{proposition}
\label{prop:Freope}
For $x,y\in\mathfrak F$ as in \eqref{eq:elF}, the Freudenthal product is given by
$x\times y:=\Phi(\phi, X, Y, \nu)$, where
\begin{equation}
\begin{aligned}
\phi&=-\tfrac{1}{2}(A\vee D+C\vee B)\\
X&=-\tfrac{1}{4}(B\times D-\alpha C-\gamma A)\\
Y&=\tfrac{1}{4}(A\times C-\beta D-\delta B)\\
\nu&=\tfrac{1}{8}((A,D)+(C,B)-3(\alpha\delta+\beta\gamma)),\\
\end{aligned}
\end{equation}
and $A\vee B\in E_6$ is defined by $(A\vee B)C=\tfrac{1}{2}(B,C)A+\tfrac{1}{6}(A,B)C-\tfrac{1}{2} B\times(A\times C)$. The group of automorphism of the Freudenthal product is a connected simple Lie group of type $E_7$.
\end{proposition}
We are now ready to describe the KTS with derivation algebra $E_6\oplus\CC$.
\subsubsection{The first product}
In the following, we denote by $I:\mathfrak F\to\mathfrak F$ the paracomplex structure on $\mathfrak F$ defined by
$$
I\begin{pmatrix}\alpha&A\\B&\beta\end{pmatrix}=\begin{pmatrix}\alpha&A\\-B&-\beta\end{pmatrix}\;.
$$
Note that $I$ is invariant under the subalgebra $E_6\oplus\CC G$ of $E_7$ \eqref{eq:E73grad}.
\begin{theorem}
\label{thm:KTS-FTS}
The vector space $\mathfrak F$ with the triple product given by 
\begin{equation*}
\begin{split}
(xyz)&=-4x\times Iy(z)+\frac{1}{2}\left\{x,Iy\right\}z
\end{split}
\end{equation*}
for all $x,y,z\in\mathfrak F$ is a $K$-simple Kantor triple system with Tits-Kantor-Koecher Lie algebra $E_8$ and derivation algebra $\der(\mathfrak F)=E_6\oplus\CC G$.
\end{theorem}
\begin{proof}
The Lie algebra $\fg=E_8$ with the $5$-grading \eqref{eq:contactE8} has negatively graded part 
\begin{equation*}
\begin{aligned}
\fm&=\fg_{-2}\oplus\fg_{-1}\\
&=\CC\1\oplus\mathfrak F
\end{aligned}
\end{equation*}
with Lie brackets
$
[x,y]=\left\{x,y\right\}\1
$
for all $x,y\in\mathfrak F$. Clearly $\fg_0=E_7\oplus\CC E$ acts on $\fm$ by $0$-degree derivations and it is known that $E_8$ is the maximal transitive prolongation of $\fg_{\leq 0}$ \cite{MR1274961}.

We will denote elements of $\fg_{1}\simeq\mathfrak F$ by $\widehat{x},\widehat y:\fm\to\fg_{-1}\oplus\fg_{0}$ and fix a generator $\widehat{\1}:\fm\to\fg_{0}\oplus\fg_{1}$ of $\fg_{2}$. 
By $E_7$-equivariance, the adjoint action on $\fm$ is necessarily of the following form:
\begin{equation*}
\label{eq:g1E8}
\begin{split}
[\widehat x, y]&=c_1(x\times y)+c_2\left\{x,y\right\} E\;,\qquad
[\widehat x, \1]=x\;,
\end{split}
\end{equation*}
and
\begin{equation*}
\label{eq:g2E8}
\begin{split}
[\widehat \1, y]&=\widehat y\;,\qquad
[\widehat \1, \1]=c_3 E\;,
\end{split}
\end{equation*}
for some constants $c_1,c_2,c_3$ to be determined. First of all
\begin{equation*}
\label{eq:c_3}
\begin{aligned}
0&=[\widehat \1,[x,\1]]=[[\widehat \1,x],\1]+[x,[\widehat 1,\1]]\\
&=[\widehat x,\1]+c_3[x,E]=x+c_3x\;,
\end{aligned}
\end{equation*}
and similarly
\begin{equation*}
\label{eq:c_2}
\begin{aligned}
0&=[\widehat x,[y,\1]]=[c_2\left\{x,y\right\} E,\1]-\left\{x,y\right\}\1\\
&=-(2c_2+1)\left\{x,y\right\}\1\;,
\end{aligned}
\end{equation*}
for all $x,y\in\mathfrak F$, whence $c_3=-1$, $c_2=-\tfrac{1}{2}$. On the other hand
\begin{equation}
\label{eq:c_1I}
\begin{aligned}
{}[\widehat x,[y,z]]&=\left\{y,z\right\}x
\end{aligned}
\end{equation}
and
\begin{equation}
\label{eq:c_1II}
\begin{aligned}
{}[[\widehat x,y],z]+[y,[\widehat x,z]]&=[c_1(x\times y)+c_2\left\{x,y\right\} E,z]
-[c_1(x\times z)+c_2\left\{x,z\right\} E,y]\\
&=c_1x\times y(z)+\frac{1}{2}\left\{x,y\right\}z
-c_1x\times z(y)-\frac{1}{2}\left\{x,z\right\}y\;.
\end{aligned}
\end{equation}
Choose $x=y=\begin{pmatrix}1&0\\0&1\end{pmatrix}$ and $z=\begin{pmatrix}1&0\\0&-1\end{pmatrix}$ so that $x\times y=\Phi(0,0,0,-\tfrac{3}{4})$, $\left\{x,y\right\}=x\times z=0$, $\left\{x,z\right\}=\left\{y,z\right\}=-2$ and get $c_1=4$ equating \eqref{eq:c_1I} with \eqref{eq:c_1II}. The adjoint action of the positively-graded part of $\fg$ on $\fm$ has been described. The remaining non-trivial Lie brackets of $\fg$ are given by the natural adjoint action of $\fg_0$ and $[\widehat x,\widehat y]=\left\{x,y\right\}\widehat\1$ for all $x,y\in\mathfrak F$.

Using the explicit expressions of the Lie brackets, it is a straightforward task to check that
\[ 
\begin{array}{ll}
\sigma(\1)=-\widehat{\1}\;,&\sigma(x)=\widehat{Ix}\;,\\[0cm]
\sigma(\Phi(\phi,A,B,\nu))=\Phi(\phi,-A,-B,\nu)\;,&\sigma(E)=-E\;,\\[0cm]
\sigma(\widehat x)=Ix\;,&\sigma(\widehat \1)=-\1\;,
\end{array}
 \]
is a grade-reversing involution of $\fg$ and compute the triple product.
\end{proof}

\subsubsection{The second product}
To get the second KTS associated to the contact grading of $E_8$, we emply a description of $E_7$ dating back to Cartan \cite{MR0050516}.

The Lie algebra $E_7$ admits a symmetric irreducible decomposition 
$$E_7= \fsl(8,\CC)\oplus \Lambda^4(\CC^8)^*\;,$$ 
where the Lie subalgebra $\fsl(8,\CC)$ acts in the natural way on $\Lambda^4(\CC^8)^*$. 
To describe the brackets of two elements in $\Lambda^4(\CC^8)^*$, it is convenient to fix a volume $vol\in \Lambda^8(\CC^8)$, let
$\sharp:\Lambda^k(\CC^8)^*\to\Lambda^{8-k}(\CC^8)$ be the map which sends any $\xi$ to $i_\xi(vol)$ and $\flat:\Lambda^{8-k}(\CC^8)\to\Lambda^{k}(\CC^8)^*$ its inverse. The maps
$\sharp$ and $\flat$ are $\fsl(8,\CC)$-equivariant and can be thought as the analogs of the usual musical isomorphisms when only a volume is assigned.
 Let also $$\bullet:\Lambda^k(\CC^8)\otimes\Lambda^k(\CC^8)^*\to \fsl(8,\CC)$$
be the unique $\fsl(8,\CC)$-equivariant map for any $k=1,\ldots, 7$, which in our conventions is normalized so that
\begin{equation*}
\label{eq:convention1}
\be_{1\cdots k}\bullet \be^{1\cdots k}=\frac{8-k}{8}\sum_{i=1}^k\be_i\otimes \be^i-\frac{k}{8}\sum_{j=k+1}^8\be_j\otimes\be^j\;,
\end{equation*}
where $(\be_i)$ is the standard basis of $\CC^8$ and $(\be^i)$ the dual basis of $(\CC^8)^*$. With this in mind, the Lie bracket of $\alpha,\beta\in\Lambda^4(\CC^8)^*$ is the element
\begin{equation*}
\label{eq:convention2}
[\alpha,\beta]=\alpha^\sharp\bullet\beta
\end{equation*}
of $\fsl(8,\CC)$.

The contact grading of $\fg=E_8$ is given by $\fg_{0}=E_7\oplus\CC E$, $\fg_{\pm 2}\simeq\CC$ and 
$$\fg_{\pm 1}\simeq\Lambda^2(\CC^8)\oplus\Lambda^2(\CC^8)^*\;,$$ where
the negatively graded part $\fm=\fg_{-2}\oplus\fg_{-1}$ of $\fg$ has the Lie brackets
$$
[X,Y]=(x^*(y)-y^*(x))\1
$$
for all $X=(x,x^*)$, $Y=(y,y^*)$ in $\fg_{-1}=\Lambda^2(\CC^8)\oplus\Lambda^2(\CC^8)^*$.
The action of $\fg_0$ as derivations of $\fm$ is the natural one of $\fsl(8,\CC)\oplus\CC E$ together with
$$
[\alpha,X]=((\alpha\wedge x^*)^\sharp,i_{x}(\alpha))\;,
$$
where $\alpha\in\Lambda^4(\CC^8)^*$ and $X\in\fg_{-1}$. 

The next result follows from the fact that $\fg=E_8$ is the maximal prolongation of $\fm\oplus\fg_0$ \cite{MR1274961} and from direct computations using $\fsl(8,\CC)$-equivariance.
\begin{proposition}
For all $X=(x,x^*)\in\Lambda^2(\CC^8)\oplus\Lambda^2(\CC^8)^*$,
the operator $\widehat X:\fm\to\fg_{-1}\oplus\fg_{0}$ given by
\begin{equation*}
\begin{aligned}
{}[\widehat X,Y]&=\underbrace{(x\bullet y^*+y\bullet x^*)}_{\text{element of\;$\fsl(8,\CC)$}}
+\underbrace{(x^*\wedge y^*-(x\wedge y)^\flat)}_{\text{element of\;$\Lambda^4(\CC^8)^*$}}
-\frac{1}{2}\underbrace{(x^*(y)-y^*(x))E}_{\text{element of\;$\CC E$}}\;,\\
[\widehat X,\1]&=X\;,
\end{aligned}
\end{equation*}
where $Y\in\fg_{-1}$, is an element of the first prolongation $\fg_{1}$. Similarly $\widehat \1:\fm\to \fg_{0}\oplus\fg_{1}$ given by
\begin{equation*}
\begin{aligned}
{}[\widehat \1,Y]&=\widehat Y\;,\qquad
[\widehat \1,\1]=-E\;,
\end{aligned}
\end{equation*}
is a generator of $\fg_{2}$.
\end{proposition}
It is not difficult to see that the remaining Lie brackets of $\fg$ are given by the action of $E$ and 
\begin{equation}
[A,\widehat X]=\widehat{[A,X]}\;,\qquad[\alpha,\widehat X]=\widehat{[\alpha,X]}\;,\qquad  [\widehat X,\widehat Y]=\widehat{[X,Y]}
\end{equation}
where $A\in\fsl(8,\CC)$, $\alpha\in\Lambda^4(\CC^8)^*$ and $X,Y\in \Lambda^2(\CC^8)\oplus\Lambda^2(\CC^8)^*$.
\begin{theorem}
\label{thm:KTS-nonFTS}
The vector space $\Lambda^2(\CC^8)\oplus\Lambda^2(\CC^8)^*$ with the triple product 
\begin{equation*}
(XYZ)=i\mat{\imath_{z^*}(x\wedge y)+\frac{1}{2}(x^*(y)+y^*(x))z+(x\bullet y^*-y\bullet x^*)\cdot z+(x^*+\wedge y^*\wedge z^*)^\sharp\\
\imath_z(x^*\wedge y^*)+\frac{1}{2}(x^*(y)+y^*(x))z^*+(x\bullet y^*-y\bullet x^*)\cdot z^*+(x\wedge y\wedge z)^\flat
}
\end{equation*}
for all $X=(x,x^*)$, $Y=(y,y^*)$, $Z=(z,z^*)$ in $\Lambda^2(\CC^8)\oplus\Lambda^2(\CC^8)^*$ is a $K$-simple Kantor triple system with Tits-Kantor-Koecher Lie algebra $E_8$ and derivation algebra $\fsl(8,\CC)$.
\end{theorem}
\begin{proof}
The map defined by
\[ 
\begin{array}{lll}
\sigma(\1)=\widehat{\1}\;,&\sigma(X)=\widehat{(ix,-ix^*)}\;,&\\[0cm]
\sigma(A)=A\;,&\sigma(\alpha)=-\alpha\;\;,&\sigma(E)=-E\;,\\[0cm]
\sigma(\widehat X)=(-ix,ix^*)\;,&\sigma(\widehat \1)=\1\;,&
\end{array}
 \]
for all $X=(x,x^*)$, $A\in \fsl(8,\CC)$, $\alpha\in\Lambda^4(\CC^8)^*$, is a grade-reversing involution of $E_8$. 
\end{proof}
\section{The exceptional Kantor triple systems of special type}
\label{sec:specialKTS} 
\subsection{The case $\fg=E_6$}
\label{sec:specialE6}
\hfill\vskip0.1cm\par
The Lie algebra $E_6$ admits a special $5$-grading which is not of contact or of extended Poincar\'e type. It is described by the crossed Dynkin diagram
$$
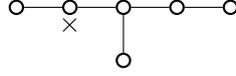
\begin{tikzpicture}
\node[root]   (1)                     {};
\node[xroot]   (2) [right=of 1] {} edge [-] (1);
\node[root]   (3) [right=of 2] {} edge [-] (2);
\node[root]   (4) [right=of 3] {} edge [-] (3);
\node[root]   (5) [right=of 4] {} edge [-] (4);
\node[root]   (6) [below=of 3] {} edge [-] (3);
\end{tikzpicture}
$$
and its graded components are $\fg_0=\fsl(5,\CC)\oplus\fsl(2,\CC)\oplus\CC E$, where $E$ is the grading element, and 
\begin{equation*}
\begin{aligned}
\fg_{-1}&=\Lambda^2(\CC^5)^*\boxtimes\CC^2\;,\qquad\fg_{1}=\Lambda^2\CC^5\boxtimes\CC^2\;,\\
\fg_{-2}&=\Lambda^4(\CC^5)^*\;,\qquad\;\;\;\;\;\;\;\;\,\fg_{2}=\Lambda^4\CC^5\;,
\end{aligned}
\end{equation*}
with their natural structure of $\fg_0$-modules. In the following, we denote forms in $\Lambda^2(\CC^5)^*$ (resp. $\Lambda^4(\CC^5)^*$) by $\alpha, \beta, \gamma$ (resp. $\xi,\phi, \psi$) and polyvectors by $\widetilde\alpha\in\Lambda^2\CC^5$, $\widetilde\xi\in\Lambda^4\CC^5$, etc. Finally $a,b,c\in\CC^2$
and we use the standard symplectic form $\omega$ on $\CC^2$ to identify $\fsl(2,\CC)$ with $S^2\CC^2$.

In order to describe the Lie brackets of $E_6$, we note that for $k=1,\ldots,4$ there exists a (unique up to constant) $\fsl(5,\CC)$-equivariant map
$$\bullet:\Lambda^k\CC^5\otimes\Lambda^k(\CC^5)^*\to\fsl(5,\CC)\;.$$ 
In our conventions, it is normalized so that
$$
\be_{1\cdots k}\bullet \be^{1\cdots k}=\frac{5-k}{5}\sum_{i=1}^k\be_i\otimes \be^i-\frac{k}{5}\sum_{j=k+1}^5\be_j\otimes\be^j\;,
$$
where $(\be_i)$ is the standard basis of $\CC^5$ and $(\be^i)$ the dual basis of $(\CC^5)^*$. The structure of graded Lie algebra on $\fm=\fg_{-2}\oplus\fg_{-1}$ is given by
$$
[\alpha\otimes a,\beta\otimes b]=\omega(a,b)\alpha\wedge\beta\;,
$$
while the adjoint action on $\fm$ of positive-degree elements is of the form
\begin{equation}
\label{eq:1E6s}
\begin{split}
[\widetilde\alpha\otimes a,\beta\otimes b]&=c_1\!\!\!\!\!\!\underbrace{\imath_{\widetilde\alpha}\beta\, a\odot b}_{\text{element of}\;\fsl(2,\CC)}\!\!\!+\;\; c_2\underbrace{\imath_{\widetilde\alpha}\beta\omega(a,b)\,E}_{\text{element of}\;\CC E}\;\;+\;\; c_3\!\!\underbrace{\omega(a,b)\widetilde\alpha\bullet\beta}_{\text{element of }\;\fsl(5,\CC)}\;,\\
[\widetilde\alpha\otimes a,\psi]&=\imath_{\widetilde\alpha}\psi\otimes a\;,
\end{split}
\end{equation}
and
\begin{equation}
\label{eq:2E6s}
\begin{split}
[\widetilde\xi,\beta\otimes b]&=-\imath_{\beta}\widetilde\xi\otimes b\;,\\
[\widetilde\xi,\psi]&=c_4\!\!\!\!\!\!\underbrace{\imath_{\widetilde\xi}\psi\,E}_{\text{element of}\;\CC E}\!\!\!\!+\;\;\,c_5\!\!\!\!\!\!\!\!\!\!\!\underbrace{\widetilde\xi\bullet\psi}_{\text{element of }\;\fsl(5,\CC)}\;,
\end{split}
\end{equation}
for some constants $c_1,\ldots,c_5$ to be determined.
\begin{proposition}
The constants in \eqref{eq:1E6s}-\eqref{eq:2E6s} are $c_1=\tfrac{1}{2}$, $c_2=-\tfrac{3}{10}$, $c_3=-1$, $c_4=\tfrac{3}{5}$ and $c_5=1$. 
\end{proposition}
\begin{proof}
As usual, elements of the first and second prolongation act as derivations on $\fm$. 

First, let us choose $a,b\in\CC^2$ such that $\omega(a,b)=1$ and compute
\begin{align*}
0&=[\widetilde\alpha\otimes a,[\beta\otimes b,\xi]]=[[\widetilde\alpha\otimes a,\beta\otimes b],\xi]+[\beta\otimes b,[\widetilde\alpha\otimes a,\xi]]\\
&=-2c_2(\imath_{\widetilde\alpha}\beta)\xi+c_3(\widetilde\alpha\bullet\beta)\cdot\xi-\beta\wedge\imath_{\widetilde\alpha}\xi
\end{align*}
for all $\widetilde\alpha\in\Lambda^2\CC^5$, $\beta\in\Lambda^2(\CC^5)^*$, $\xi\in\Lambda^4(\CC^5)^*$. Choosing suitable forms and polyvectors yields a regular non-homogeneous linear system in $c_2$, $c_3$. For instance if $\widetilde\alpha=\be_{12}$, $\beta=\be^{12}$ one gets
$$
-2c_2-\tfrac{2}{5}c_3=1\;,\qquad -2c_2+\tfrac{3}{5}c_3=0\;,
$$
taking $\xi=\be^{1234}$ and $\xi=\be^{2345}$, respectively. In other words $c_2=-\tfrac{3}{10}$, $c_3=-1$. The adjoint action of $\fg_1$ on $\fg_{-2}=[\fg_{-1},\fg_{-1}]$ gives $c_1=\tfrac{1}{2}$ with an analogous computation. 

We now note that 
\begin{equation*}
\begin{split}
0=&[\widetilde\xi,[\beta\otimes b,\psi]]=[[\widetilde\xi,\beta\otimes b],\psi]+[\beta\otimes b,[\widetilde\xi,\psi]]\\
&=-[\imath_{\beta}\widetilde\xi\otimes b,\psi]+c_5[\beta\otimes b,\widetilde\xi\bullet\psi]
+c_4(\imath_{\widetilde\xi}\psi)\beta\otimes b\\
&=-\imath_{(\imath_{\beta}\widetilde\xi)}\psi\otimes b-c_5(\widetilde\xi\bullet\psi)\cdot\beta\otimes b+c_4(\imath_{\widetilde\xi}\psi)\beta\otimes b
\end{split}
\end{equation*}
holds for all $b\in\CC^2$. Choosing $\widetilde\xi=\be_{1234}$, $\psi=\be^{1234}$ and, in turn, $\beta=\be^{12}$ and then $\beta=\be^{45}$, yields a system of linear equations
\begin{equation*}
\begin{split}
\tfrac{2}{5}c_5+c_4&=1\;,\qquad
-\tfrac{3}{5}c_5+c_4=0\;,
\end{split}
\end{equation*}
whose unique solution is $c_4=\tfrac{3}{5}$, $c_5=1$.
\end{proof}
The remaining Lie brackets follows easily, as $\fg=E_6$ is the maximal prolongation of $\fm=\fg_{-2}\oplus\fg_{-1}$. They are given by the natural action of $\fg_0$ on $\fg_p$, $p\geq 0$, and by 
$$
[\widetilde\alpha\otimes a,\widetilde\beta\otimes b]=-\omega(a,b)\widetilde\alpha\wedge\widetilde\beta
$$
for all $\widetilde\alpha,\widetilde\beta\in\Lambda^2\CC^5$ and $a,b\in\CC^2$.

By the results of \S\ref{sub:classKTS} there is only one grade-reversing involution, namely the Chevalley involution with derivation algebra $\fso(5,\CC)\oplus\fso(2,\CC)$.
To describe the associated KTS, we denote by $\eta$ the standard non-degenerate symmetric bilinear form on $\CC^p$ for $p=2, 5$ and extend the musical isomorphisms
$$
\flat:\CC^p\to(\CC^p)^*\;,\qquad\natural:(\CC^p)^*\to\CC^p\;,
$$
to forms and polyvectors in the obvious way. When $p=2$ we also consider the compatible complex structure $J:\CC^2\to\CC^2$ given by $\eta(a,b)=\omega(a,Jb)$ for all $a,b\in\CC^2$. 
\begin{theorem}
\label{thm:E6KTSspecial}
 The vector space $V=\Lambda^2(\CC^5)^*\otimes\CC^2$ with triple product
 \begin{align*}
((\alpha\otimes a)\;(\beta\otimes b)\;(\gamma\otimes c))&=-\frac{1}{2}\eta(\alpha,\beta)(\omega(a,c)\gamma\otimes Jb+\omega(Jb,c)\gamma\otimes a)\\
&\;\;\,\,+\frac{3}{10}\eta(\alpha,\beta)\eta(a,b)\gamma\otimes c-\eta(a,b)(\beta^{\sharp}\bullet\alpha)\cdot\gamma\otimes c
\end{align*}
 for all $\alpha,\beta,\gamma\in \Lambda^2(\CC^5)^*$, $a,b,c\in\CC^2$, is a K-simple Kantor triple system with derivation algebra
 $\der(V)=\fso(5,\CC)\oplus\fso(2,\CC)$ and Tits-Kantor-Koecher Lie algebra $\fg=E_6$.
\end{theorem}
\begin{proof}
Let $\sigma:\fg\to\fg$ be the grade-reversing map defined by
\begin{equation}
\label{eq:E6invspecial}
\begin{array}{ll}
\sigma(\xi)=-\xi^\natural\;,&\sigma(\alpha\otimes a)=\alpha^\natural\otimes Ja\;,\\[0cm]
\sigma(A)=-A^t\;,&\sigma(E)=-E\;,\\[0cm]
\sigma(\widetilde\alpha\otimes a)=-\widetilde\alpha^\flat\otimes Ja\;,&\sigma(\widetilde\xi)=-\widetilde\xi^\flat\;,
\end{array}
\end{equation}
for all $A\in\fsl(2,\CC)\oplus\fsl(5,\CC)$ and forms $\xi,\alpha$, polyvectors $\widetilde\xi, \widetilde\alpha$, $a\in\CC^2$. Clearly $\sigma^2=1$ and by the explicit expressions of the Lie brackets of $E_6$, one checks that $\sigma$ is a Lie algebra morphism. 
\end{proof}
\subsection{The case $\fg=E_7$}
\label{sec:endpaper}
\hfill\vskip0.1cm\par
The Lie algebra $\fg=E_7$ admits a special grading similar to that of $E_6$ in \S\ref{sec:specialE6}, i.e.,
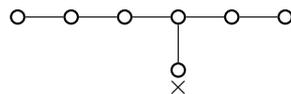
$$
\begin{tikzpicture}
\node[root]   (1)                     {};
\node[root]   (2) [right=of 1] {} edge [-] (1);
\node[root]   (3) [right=of 2] {} edge [-] (2);
\node[root]   (4) [right=of 3] {} edge [-] (3);
\node[root]   (5) [right=of 4] {} edge [-] (4);
\node[root]   (7) [right=of 5] {} edge [-] (5);
\node[xroot]   (6) [below=of 4] {} edge [-] (4);
\end{tikzpicture}
$$
with graded components $\fg_0\simeq \fsl(7,\CC)\oplus\mathbb C E,\ \fg_{-1}\simeq \Lambda^3(\CC^7)^*,\ \fg_{1}\simeq \Lambda^3\CC^7,\ \fg_{-2}\simeq \Lambda^6(\CC^7)^*$ and $\fg_{2}\simeq \Lambda^6\CC^7$. There is only one grade-reversing involution, i.e., the Chevalley involution. The symmetry algebra of the associated KTS is $\fso(7,\CC)$.

Let $\eta$ be the standard non-degenerate symmetric bilinear form on $\CC^7$, which we extend naturally to any $\Lambda^k\CC^7$, $k\geq 0$. 
As in \S\ref{sec:specialE6} we consider the $\fsl(7,\CC)$-equivariant projection
$$
\bullet:\Lambda^k\CC^7\otimes\Lambda^k(\CC^7)^*\to \fsl(7,\CC)
$$
normalized so that
$$
\be_{1\cdots k}\bullet \be^{1\cdots k}=\frac{7-k}{7}\sum_{i=1}^k\be_i\otimes \be^i-\frac{k}{7}\sum_{j=k+1}^5\be_j\otimes\be^j\;,
$$
where $(\be_i)$ is the standard basis of $\CC^7$ and $(\be^i)$ the dual basis of $(\CC^7)^*$.
We denote by $\natural$ and $\flat$ the musical isomorphisms, inverse to each other, associated to $\eta$.
\begin{theorem}
 \label{thm:E7NonSpin}
 The vector space $V=\Lambda^3(\CC^7)^*$ with triple product 
 $$(\alpha\beta\gamma)=\frac{2}{7}\eta(\alpha,\beta)\gamma-( \beta^\sharp\bullet\alpha)\cdot\gamma$$
 for all $\alpha,\ \beta,\ \gamma\in V$ is a K-simple Kantor triple system with symmetry algebra 
 $\der(V)=\stab_{\fsl(7,\CC)}(\eta)\simeq\fso(7,\CC)$ and Tits-Kantor-Koecher Lie algebra $\fg=E_7$.
\end{theorem}

\begin{proof}
The Lie brackets are given
 by the natural action of the grading element $E$, the standard action (resp. dual action) of the Lie algebra $\fsl(7,\CC)$ on $\Lambda^k\CC^7$ (resp. $\Lambda^k(\CC^7)^*$) for $k=3,6$ and
 $$\begin{array}{ll}
 [\alpha,\beta]=\alpha\wedge\beta\;, &[\widetilde \alpha, \widetilde \beta]=-\widetilde\alpha\wedge\widetilde\beta\;, \\[5pt]
[\widetilde \alpha,\beta]=-\frac{2}{7}\imath_{\widetilde\alpha}\beta\,E-(\widetilde\alpha\bullet\beta) \;, &\ [\widetilde \alpha,\xi]=i_{\widetilde\alpha}\xi\;,\\[5pt]
[\widetilde\xi,\alpha]=i_{\alpha}\widetilde\xi\;,&\ [\widetilde\xi,\psi]=-\frac{4}{7}\imath_{\widetilde\xi}\psi\,E+(\widetilde\xi\bullet\psi)\;,\\[5pt]
\end{array}
$$
with $\alpha,\beta\in\fg_{-1},\ \widetilde\alpha,\widetilde\beta\in\fg_1$, $\xi,\psi\in\fg_{-2},\ \widetilde\xi,\widetilde\psi\in\fg_2$. The grade reversing involution is
\begin{equation}
\label{eq:E7NonSpinInvol}
\begin{array}{ll}
\sigma(\xi)=-\xi^\natural\;,&\sigma(\alpha)=-{\alpha^\natural}\;,\\[0cm]
\sigma(A)=-A^t\;,&\sigma(E)=-E\;,\\[0cm]
\sigma(\widetilde \alpha)=-\widetilde \alpha^\flat\;,&\sigma(\widetilde\xi)=-\widetilde\xi^\flat
\end{array}
\end{equation}
and the triple product follows as usual.
\end{proof}
\section*{Acknowledgments}
The third author acknowledges support by a Marie-Curie fellowship of 
INdAM (Italy) and the project "Lie superalgebra theory and its applications" of the University of Bologna.
\bibliographystyle{plain}
\bibliography{SimpleKTS}
\end{document}